\documentclass[3p,times]{./elsarticle_arXiv}
\usepackage{graphicx}
\usepackage{amsmath,amsthm,amsbsy,amsfonts,amssymb}
\usepackage{calligra,calrsfs,mathrsfs}
\usepackage{multirow}
\usepackage{ulem,cancel}
\usepackage{cases}
\usepackage{pgfplots}
\usepackage{subcaption}
\usepackage{tikz}
\usepackage{tkz-euclide}
\usetikzlibrary[patterns]
\usepackage{hyperref}
\usepackage{algorithm}
\usepackage{algpseudocode}
\usepackage{dsfont}
\usepackage{caption}
\usepackage{epstopdf}
\usepackage{epsfig}
\usepackage{array}
\newcolumntype{P}[1]{>{\centering\arraybackslash}p{#1}}
\usepackage{enumerate}
\usepackage{float}
\usepackage[T1]{fontenc}
\usepackage[latin9]{inputenc}
\usepackage{geometry}
\usepgfplotslibrary{fillbetween}
\usetikzlibrary{patterns}
\usepackage{tikz-dimline}
\usepackage{color}
\usepackage{stmaryrd}
\usepackage{setspace}
\usepackage{amsthm}
\usepackage{textcomp}
\usepackage{hyperref}
\usepackage{soul,xcolor}
\usepackage{multirow}
\usepackage{ulem}
\usepackage[inline]{enumitem}

\newcommand{\ifcomment}{\iffalse}
\newcommand{\bs}[1]{\boldsymbol{#1}}

\newcommand{\avg  }[1]{\langle #1 \rangle}

\newcommand{\ti}[1]{\tilde{#1}}
\newcommand{\G}{\Gamma}

\newcommand{\Om}{\Omega}

\newcommand{\cT}{{\mathcal T}}
\newcommand{\tO}{\tilde{\Omega}_h}

\newcommand{\tG}{\tilde{\Gamma}_h}

\newcommand{\tGD}{\tilde{\Gamma}_{D,h}}
\newcommand{\tGN}{\tilde{\Gamma}_{N,h}}

\newcommand{\tx}{{\tilde{\bs{x}}}}
\newcommand{\tT}{{\tilde{\cT}_h}}
\newcommand{\tS}{S_{h}}

\newtheorem{thm}{Theorem}

\newtheorem{prop}{Proposition}
\newtheorem{lemma}{Lemma}
\newdefinition{rem}{Remark}
\newtheorem{assumption}{Assumption}

\newcolumntype{M}[1]{>{\centering\arraybackslash}m{#1}}
\newcommand{\Figures}{./Figures}

\biboptions{sort&compress}
\allowdisplaybreaks

\begin{document}
	
	
	\begin{frontmatter}

\title{The Second-Generation Shifted Boundary Method and Its Numerical Analysis}

\author[duke]{Nabil M. Atallah}
\ead{nabil.atallah@duke.edu}
\author[polito]{Claudio Canuto}
\ead{claudio.canuto@polito.it}
\author[duke]{Guglielmo Scovazzi\corref{ca}}
\ead{guglielmo.scovazzi@duke.edu}
\address[duke]{Department of Civil and Environmental Engineering, Duke University, Durham, North Carolina 27708, USA}
\address[polito]{Dipartimento di Scienze Matematiche, Politecnico di Torino, Corso Duca degli Abruzzi, 24
	10129 Torino, Italy}
\cortext[ca]{Corresponding author: Guglielmo Scovazzi}

\begin{abstract} 
Recently, the Shifted Boundary Method (SBM) was proposed within the class of unfitted (or immersed, or embedded) finite element methods.
By reformulating the original boundary value problem over a surrogate (approximate) computational domain, the SBM avoids integration over cut cells and the associated problematic issues regarding numerical stability and matrix conditioning. Accuracy is maintained by modifying the original boundary conditions using Taylor expansions. Hence the name of the method, that {\it shifts} the location and values of the boundary conditions.
In this article, we present enhanced variational SBM formulations for the Poisson and Stokes problems with improved flexibility and robustness. 
These simplified variational forms allow to relax some of the assumptions required by the mathematical proofs of stability and convergence of earlier implementations. 
First, we show that these new SBM implementations can be proved asymptotically stable and convergent even without the rather restrictive assumption that the inner product between the normals to the true and surrogate boundaries is positive.
Second, we show that it is not necessary to introduce a stabilization term involving the tangential derivatives of the solution at Dirichlet boundaries, therefore avoiding the calibration of an additional stabilization parameter.
Finally, we prove enhanced $L^{2}$-estimates without the cumbersome assumption - of earlier proofs - that the surrogate domain is convex. Instead we rely on a conventional assumption that the boundary of the true domain is smooth, which can also be replaced by requiring convexity of the true domain.
The aforementioned improvements open the way to a more general and efficient implementation of the Shifted Boundary Method, particularly in complex three-dimensional geometries.
We complement these theoretical developments with numerical experiments in two  and three dimensions.
\end{abstract}

\end{frontmatter}

\section{Introduction \label{sec:intro} }
In this article, we provide improved proofs for well-posedness, numerical stability and convergence of the {\it shifted boundary method} (SBM) for the Poisson and Stokes problems under a simplified set of assumptions which makes the SBM more generally applicable in the simulation of practical engineering problems of very complex geometry. 

We briefly recall the scope and motivation for the SBM, which falls in the broader category of unfitted (or embedded) finite element methods~\cite{boffi2003finite,burman2010ghost,hansbo2002unfitted,hollig2003finite,hollig2001weighted,ruberg2012subdivision,ruberg2014fixed,schott2015face,burman2018cut,burman2019dirichlet,burman2017cut,burman2010fictitious,burman2012fictitious,burman2014unfitted,burman2018shape,massing2015nitsche,burman2015cutfem,kamensky2017immersogeometric,xu2016tetrahedral,lozinski2016new}. 
Many of these methods require the geometric construction of the partial elements cut by the embedded boundary, which can be both algorithmically complicated and computationally intensive, due to data structures that are considerably more complex with respect to corresponding fitted finite element methods. 
Furthermore, integrating the variational forms on the characteristically irregular cut cells may also be difficult and advanced quadrature formulas might need to be employed~\cite{parvizian2007finite,duster2008finite}. Accordingly, it is typical for unfitted methods that a non-negligible portion of the overall wall-clock time for a simulation is spent handling the embedded boundary, when complex geometries are considered.

The SBM approach is instead aimed at avoiding integration over cut cells and all the problematic issues just mentioned, and belongs to the more specific class of approximate domain methods~\cite{bramble1972projection,bramble1996finite,bramble1994robust,cockburn2012solving,cockburn2014priori,cockburn2010boundary,bertoluzza2005fat,bertoluzza2011analysis,glowinski1994fictitious} for some examples). 
The SBM is built for minimal computational complexity, in that the location where boundary conditions are applied is {\it shifted} from the true to an approximate (surrogate) boundary, and, at the same time, modified ({\it shifted}) boundary conditions are applied in order to avoid a reduction in the convergence rates of the overall formulation. In fact, if the boundary conditions associated to the true domain are not appropriately modified on the surrogate domain, only first-order convergence is to be expected. The shifted boundary conditions are appropriately modified by means of Taylor expansions and are applied weakly, using a Nitsche strategy. This process yields a method which is simple, robust, accurate and efficient.

The shifted boundary method was proposed in~\cite{main2018shifted0} for the Poisson and Stokes flow problems and generalized in~\cite{main2018shifted} to the advection-diffusion and Navier-Stokes equations, and in~\cite{song2018shifted} to hyperbolic conservation laws. In~\cite{main2018shifted0} and~\cite{main2018shifted}, an analysis of the stability and accuracy of the SBM for the Poisson and advection-diffusion operators was also included, respectively. More recently, the authors of~\cite{atallah2020analysis} analyzed the stability and accuracy of the SBM for the Stokes flow equations, in an endeavor to complete the numerical analysis of the method for the fundamental differential operators that combine in the Navier-Stokes equations and many other linear and nonlinear partial differential equations of importance in engineering and physical sciences. 

In the present work, we propose {\it second generation} shifted boundary formulations for the Poisson and Stokes flow problems, and we include their mathematical analysis of stability and accuracy. 
These enhanced formulations are obtained by 1) discarding the assumption that the inner product between the normals to the true and surrogate boundaries must be positive~\cite{main2018shifted0} and 2) removing a boundary stabilization term constructed with tangential derivatives, which was initially considered necessary for the numerical stability of the method~\cite{main2018shifted0}. 
Particularly, the assumption that the inner product between the normals to the true and surrogate boundaries must be positive is typically not verified in three dimensional computations involving complex geometries, and the ability to avoid such restriction is paramount for the application of the SBM framework to general engineering problems. 
Furthermore, the new proofs of stability and accuracy proposed in this work provide a clear explanation of why the SBM is very effective in the robust treatment of complex geometry problems.

In addition, we provide enhanced proofs for the convergence rates in the $L^{2}$-norm, using the conventional requirement that the boundary of the true domain is smooth, as opposed to the restrictive assumption that the surrogate domain is convex~\cite{main2018shifted0,main2018shifted,atallah2020analysis}. We also note that the assumption of smoothness of the true boundary can be replaced by an assumption of convexity of the true domain.

Finally, in the case of the specific stabilized variational formulation utilized to treat the Stokes operator, we also discard one of the stabilization terms associated with the incompressibility condition. This modification is however less relevant for those practitioners who are interested in pairing the SBM with LBB-stable finite elements.

This article is organized as follows: Section~\ref{sec:sbm_intro} introduces the general SBM notation; the analysis of the SBM variational formulation of the Poisson problem is discussed in Section~\ref{sec:sbm_poisson}; the analysis of the SBM variational formulation of the Stokes problem is presented in Section~\ref{sec:sbm_stokes}; extensive numerical tests are presented in Section~\ref{sec:2DNumerical_Results} and Section~\ref{sec:3DNumerical_Results}; and finally, conclusions are summarized in Section~\ref{sec:summary}.

\section{The shifted boundary method} \label{sec:sbm_intro}
\noindent {\sl Notation.} Throughout the paper, we will denote by $L^{2}(\Om)$ the space of square integrable functions on $\Om$ and by $L^{2}_{0}(\Om)$ the space of square integrable functions with zero mean on $\Om$ (i.e., $q \in L^{2}_{0}(\Om)$ implies $\int_{\Om} q = 0$). We will use the Sobolev spaces $H^m(\Om)=W^{m,2}(\Om)$ of index of regularity $m \geq 0$ and index of summability 2, equipped with the (scaled) norm
\begin{equation}
\|v \|_{H^{m}(\Om)} 
= \left( \| \, v \, \|^2_{L^2(\Om)} + \sum_{k = 1}^{m} \| \, l(\Om)^k  \bs{D}^k v \, \|^2_{L^2(\Om)} \right)^{1/2} \; ,
\end{equation}
where $\bs{D}^{k}$ is the $k$th-order spatial derivative operator and $l(A)=\mathrm{meas}_{n_d}(A)^{1/n_d}$ is a characteristic length of the domain $A$ ($n_d=2,3$ indicates the number of spatial dimensions). Note that $H^0(\Om)=L^{2}(\Om)$.  As usual, we use a simplified notation for norms and semi-norms, i.e., we set $\| \, v  \, \|_{m,\Om}=\|\, v \, \|_{H^m(\Om)}$ and $| \, v \, |_{k,\Om}= 
\| \, \bs{D}^k v \,\|_{0,\Om}= \| \, \bs{D}^k v \, \|_{L^2(\Om)}$.
\subsection{The true domain, the surrogate domain and maps}
\label{sec:sbmDef}
Let $\Om$ be a connected open set in $\mathbb{R}^{n_d}$ with Lipschitz boundary. We consider a closed domain ${\cal D}$ such that $\text{clos}(\Om) \subseteq {\cal D}$ and we introduce a family $\cT_h$ of admissible and shape-regular triangulations of ${\cal D}$. Then, we restrict each triangulation by selecting those elements that are contained in $\text{clos}(\Om)$, i.e., we form
$$
\ti{\cT}_h := \{ T \in \cT_h : T \subset \text{clos}(\Om) \}\,.
$$ 
This identifies the {\sl surrogate domain}
$$
\tO := \text{int} \left(\bigcup_{T \in \ti{\cT}_h}  T \right) \subseteq \Om \,,
$$
with {\sl surrogate boundary} $\tG:=\partial \tO$ and outward-oriented unit normal vector $\ti{\bs{n}}$ to $\tG$. Obviously, $\ti{\cT}_h$ is an admissible and shape-regular triangulation of $\tO$ (see Figure~\ref{fig:SBM}).
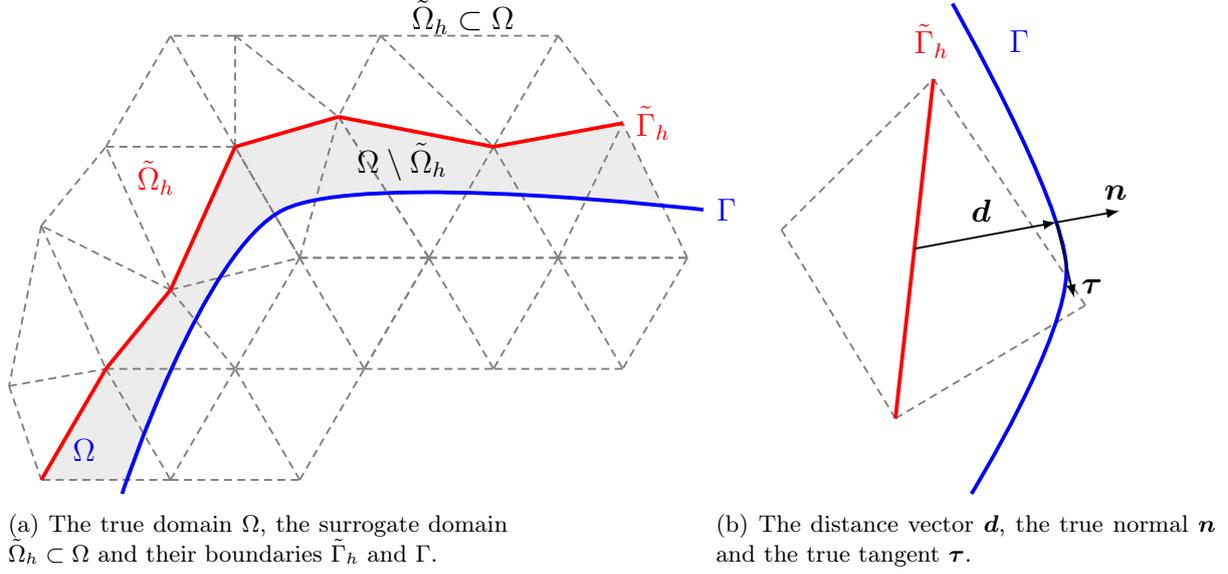
\begin{figure}
	\centering
	\begin{subfigure}[b]{.4\textwidth}\centering
		\begin{tikzpicture}[scale=0.85]
		\draw [black, draw=none,name path=surr] plot coordinates { (-2,-3.4641) (-1,-1.73205) (0,-0.5) (1,1.73205) (2.6,2.2) (5,1.73205) (7,2.1) (7.62,0.8) };
		\draw [blue, name path=true] plot[smooth] coordinates {(-0.7,-3.4641) (1.75,0.75) (8.25,0.75)};
		\tikzfillbetween[of=true and surr,split]{gray!15!};
		\draw[line width = 0.25mm,densely dashed,gray] (-1,1.73205) -- (0,3.4641);
		\draw[line width = 0.25mm,densely dashed,gray] (0,3.4641) -- (2,0);
		\draw[line width = 0.25mm,densely dashed,gray] (1,1.73205) -- (1,3.4641);
		\draw[line width = 0.25mm,densely dashed,gray] (1,3.4641) -- (0,3.4641);
		\draw[line width = 0.25mm,densely dashed,gray] (1,3.4641) -- (2.6,2.2);
		\draw[line width = 0.25mm,densely dashed,gray] (1,3.4641) -- (3.25,3.4641);
		\draw[line width = 0.25mm,densely dashed,gray] (3.25,3.4641) -- (2.6,2.2);
		\draw[line width = 0.25mm,densely dashed,gray] (3.25,3.4641) -- (5,1.73205);
		\draw[line width = 0.25mm,densely dashed,gray] (3.25,3.4641) -- (6,3.4641);
		\draw[line width = 0.25mm,densely dashed,gray] (6,3.4641) -- (5,1.73205);
		\draw[line width = 0.25mm,densely dashed,gray] (6,3.4641) -- (7,2.1);
		\draw[line width = 0.25mm,densely dashed,gray] (0,-0.5) -- (-2,0.5);
		\draw[line width = 0.25mm,densely dashed,gray] (-2,0.5) -- (-1,1.73205);
		\draw[line width = 0.25mm,densely dashed,gray] (-2,0.5) -- (-1,-1.73205);
		\draw[line width = 0.25mm,densely dashed,gray] (-2,0.5) -- (-2.5,-2);
		\draw[line width = 0.25mm,densely dashed,gray] (-2.5,-2) -- (-1,-1.73205);
		\draw[line width = 0.25mm,densely dashed,gray] (-2.5,-2) -- (-2,-3.4641);
		\draw[line width = 0.25mm,densely dashed,gray] (0,-0.5) -- (-1,1.73205);
		\draw[line width = 0.25mm,densely dashed,gray] (-1,1.73205) -- (1,1.73205);
		\draw[line width = 0.25mm,densely dashed,gray] (0,-0.5) -- (2,0);
		\draw[line width = 0.25mm,densely dashed,gray] (2,0) -- (1,1.73205);
		\draw[line width = 0.25mm,densely dashed,gray] (1,1.73205) -- (0,-0.5);
		\draw[line width = 0.25mm,densely dashed,gray] (2,0) -- (2.6,2.2);
		\draw[line width = 0.25mm,densely dashed,gray] (2.6,2.2) -- (1,1.73205);
		\draw[line width = 0.25mm,densely dashed,gray] (2,0) -- (4,0);
		\draw[line width = 0.25mm,densely dashed,gray] (4,0) -- (2.6,2.2);
		\draw[line width = 0.25mm,densely dashed,gray] (4,0) -- (2.6,2.2);
		\draw[line width = 0.25mm,densely dashed,gray] (2.6,2.2) -- (5,1.73205);
		\draw[line width = 0.25mm,densely dashed,gray] (5,1.73205) -- (4,0);
		\draw[line width = 0.25mm,densely dashed,gray] (4,0) -- (6,0);
		\draw[line width = 0.25mm,densely dashed,gray] (6,0) -- (5,1.73205);
		\draw[line width = 0.25mm,densely dashed,gray] (6,0) -- (7,2.1);
		\draw[line width = 0.25mm,densely dashed,gray] (7,2.1) -- (5,1.73205);
		\draw[line width = 0.25mm,densely dashed,gray] (6,0) -- (8,0);
		\draw[line width = 0.25mm,densely dashed,gray] (8,0) -- (7,2.1);
		\draw[line width = 0.25mm,densely dashed,gray] (0,-0.5) -- (-1,-1.73205);
		\draw[line width = 0.25mm,densely dashed,gray] (-1,-1.73205) -- (1,-1.73205);
		\draw[line width = 0.25mm,densely dashed,gray] (2,0) -- (1,-1.73205);
		\draw[line width = 0.25mm,densely dashed,gray] (1,-1.73205) -- (0,-0.5);
		\draw[line width = 0.25mm,densely dashed,gray] (2,0) -- (3,-1.73205);
		\draw[line width = 0.25mm,densely dashed,gray] (3,-1.73205) -- (1,-1.73205);
		\draw[line width = 0.25mm,densely dashed,gray] (4,0) -- (3,-1.73205);
		\draw[line width = 0.25mm,densely dashed,gray] (2,0) -- (4,0);
		\draw[line width = 0.25mm,densely dashed,gray] (4,0) -- (3,-1.73205);
		\draw[line width = 0.25mm,densely dashed,gray] (3,-1.73205) -- (5,-1.73205);
		\draw[line width = 0.25mm,densely dashed,gray] (5,-1.73205) -- (4,0);
		\draw[line width = 0.25mm,densely dashed,gray] (4,0) -- (6,0);
		\draw[line width = 0.25mm,densely dashed,gray] (6,0) -- (5,-1.73205);
		\draw[line width = 0.25mm,densely dashed,gray] (6,0) -- (7,-1.73205);
		\draw[line width = 0.25mm,densely dashed,gray] (7,-1.73205) -- (5,-1.73205);
		\draw[line width = 0.25mm,densely dashed,gray] (6,0) -- (8,0);
		\draw[line width = 0.25mm,densely dashed,gray] (8,0) -- (7,-1.73205);
		\draw[line width = 0.25mm,densely dashed,gray] (0,-3.4641) -- (-2,-3.4641);
		\draw[line width = 0.25mm,densely dashed,gray] (-2,-3.4641) -- (-1,-1.73205);
		\draw[line width = 0.25mm,densely dashed,gray]  (-1,-1.73205) -- (0,-3.4641);
		\draw[line width = 0.25mm,densely dashed,gray] (0,-3.4641) -- (1,-1.73205);
		\draw[line width = 0.25mm,densely dashed,gray] (0,-3.4641) -- (2,-3.4641);
		\draw[line width = 0.25mm,densely dashed,gray] (2,-3.4641) -- (1,-1.73205);
		\draw[line width = 0.25mm,densely dashed,gray] (2,-3.4641) -- (3,-1.73205);
		\draw [line width = 0.5mm,blue, name path=true] plot[smooth] coordinates {(-0.75,-3.681818) (1.75,0.75) (8.25,0.75)};
		\draw[line width = 0.5mm,red] (1,1.73205) -- (2.6,2.2);
		\draw[line width = 0.5mm,red] (2.6,2.2) -- (5,1.73205);
		\draw[line width = 0.5mm,red] (5,1.73205) --  (7,2.1);
		\draw[line width = 0.5mm,red] (1,1.73205) -- (0,-0.5);
		\draw[line width = 0.5mm,red] (0,-0.5) -- (-1,-1.73205);
		\draw[line width = 0.5mm,red] (-1,-1.73205) -- (-2,-3.4641);
		\node[text width=0.5cm] at (7.5,2.1) {\large${\color{red}\tG}$};
		\node[text width=3cm] at (1.25,1.25) {\large${\color{red}\tO}$};
		\node[text width=3cm] at (0.25,-3) {\large${\color{blue}\Om}$};
		\node[text width=0.5cm] at (8.75,0.75) {\large${\color{blue}\G}$};
		\node[text width=3cm] at (4.65,1.5) {\large$\Om \setminus \tO $};
		\node[text width=3cm] at (5.5,3.75) {\large$\tO  \subset \Om $};
		\end{tikzpicture}
		\caption{The true domain $\Om$, the surrogate domain $\tO \subset \Om$ and their boundaries $\ti{\G}_{h}$ and $\G$.}
		\label{fig:SBM}
	\end{subfigure}
	\hspace{2.5cm}
	\begin{subfigure}[b]{.4\textwidth}\centering
		\begin{tikzpicture}
		\draw[line width = 0.25mm,densely dashed,gray] (0,0.5) -- (-1.5,3);
		\draw[line width = 0.25mm,densely dashed,gray] (-1.5,3) -- (0.5,5);
		\draw[line width = 0.25mm,densely dashed,gray] (0,0.5) -- (2.5,2);
		\draw[line width = 0.25mm,densely dashed,gray] (2.5,2) -- (0.5,5);
		\draw[line width = 0.25mm,densely dashed,gray] (0.5,5) -- (0,0.5);
		\draw [line width = 0.5mm,blue, name path=true] plot[smooth] coordinates {(1,-0.5) (2.25,2.5) (0.75,6)};
		\draw[line width = 0.5mm,red] (0,0.5) -- (0.5,5);
		\node[text width=0.5cm] at (0.5,5.5) {\large${\color{red}\tG}$};
		\node[text width=0.5cm] at (1.75,5.5) {\large${\color{blue}\G}$};
		\node[text width=0.5cm] at (1.25,3.25) {\large$\bs{d}$};
		\node[text width=0.5cm] at (3,3.5) {\large$\bs{n}$};
		\node[text width=0.5cm] at (2.7,2.25) {\large$\bs{\tau}$};
		\draw[->,line width = 0.25mm,-latex] (0.25,2.75) -- (2.12,3.1);
		\draw[->,line width = 0.25mm,-latex] (2.12,3.1) -- (2.35,2.1);
		\draw[->,line width = 0.25mm,-latex] (2.12,3.1) -- (2.95,3.25);
		\end{tikzpicture}
		\caption{The distance vector $\bs{d}$, the true normal $\bs{n}$ and the true tangent $\bs{\tau}$.}
		\label{fig:ntd}
	\end{subfigure}
	\caption{The surrogate domain, its boundary, and the distance vector $\bs{d}$.}
	\label{fig:surrogates}
\end{figure}
We now select a mapping
\begin{subequations}\label{eq:defMmap}
	\begin{align}
	\bs{M}_{h}:&\; \tG \to \G \; ,  \\
	&\; \ti{\bs{x}} \mapsto \bs{x}   \; ,
	\end{align}
\end{subequations}
which associates to any point $\ti{\bs{x}} \in \tG$ on the surrogate boundary a point $\bs{x} = \bs{M}_{h}(\ti{\bs{x}})$ on the physical boundary $\Gamma$.  Whenever uniquely defined, the closest-point point projection of $\ti{\bs{x}}$ upon $\Gamma$ is a natural choice for $\bs{x}$, as shown e.g. in Figure~\ref{fig:ntd}. But more sophisticated choices may be locally preferable; we refer to \cite{TheoreticalPoissonAtallahCanutoScovazzi2020} for more details.
Through $\bs{M}_{h}$, a distance vector function $\bs{d}_{\bs{M}_{h}}$ can be defined as
\begin{align}
\label{eq:Mmap}
\bs{d}_{\bs{M}_{h}} (\ti{\bs{x}})
\, = \, 
\bs{x}-\ti{\bs{x}}
\, = \, 
[ \, \bs{M}_{h}-\bs{I} \, ] (\ti{\bs{x}})
\; .
\end{align}
For the sake of simplicity, we set $\bs{d} = \bs{d}_{\bs{M}_{h}} $ where $\bs{d} = \|\bs{d}\| \bs{\nu}$  and $\bs{\nu}$ is a unit vector. 
\begin{rem}
	If $\bs{M}_{h}(\ti{\bs{x}})$ does not belong to corners or edges, then $\bs{\nu}=\bs{n}$.
\end{rem}
\begin{rem}
	There are other strategies in the definition of the map $\bs{M}_{h}$ and, correspondingly, the distance vector $\bs{d}$. Among them is a level set description of the true boundary, in which $\bs{d}$ is defined by means of a distance function.
\end{rem}
In case the boundary $\G$ is partitioned into a Dirichlet boundary $\G_{D}$ and a Neumann boundary $\G_{N}$ with $\G = \overline{\G_{D} \cup \G_{N}}$ and $\G_{D} \cap \G_{N} = \emptyset$, we need to identify whether a surrogate edge $\ti{E} \subset \tG$ is associated with $\G_{D}$ or $\G_{N}$. To that end, we partition $\tG$ as $\overline{\tGD \cup \tGN}$ with $\tGD \cap \tGN = \emptyset$ using again a map $\bs{M}_{h}$, such that
\begin{align}\label{def:surrogateDir}
\tGD = \{ \ti{E} \subseteq \tG : \bs{M}_{h}(\ti{E}) \, \subseteq \, \G_{D} \}
\end{align}
and $\tGN = \tG \setminus \tGD$. 

Indicating by $h_T$ ($h^{i}_{T}$, resp.) the circumscribed diameter (inscribed diameter, resp.) of an element $T \in \ti{\cT}_h$ and by $h$ ($h^{i}$, resp.) the piecewise constant function in $\tO$ such that $h_{|T}=h_T$ ($h^{i}_{|T}= h^{i}_T$, resp.) for all $T \in \ti{\cT}_h$, we require that the distance $\| \, \bs{d} \, \|$ goes to zero slightly faster than $h$, as the grid is refined, according to the following 
\begin{assumption}
	\label{ass:d_asym}
	There exist constants $c_d>0$ and $\zeta >0$ such that
	\begin{equation}\label{eq:ass_lbb}
	\| \, \bs{d}(\ti{\bs{x}}) \, \| \leq c_d \, h_T \, \hat{h}_T^{\zeta} \;  \qquad \forall \ti{\bs{x}} \in \tG \cap T, \ \ T \in \ti{\cT}_h \; , 
	\end{equation}
	where
	\begin{equation}\label{eq:hhat}
	\hat{h}_T =  l(\tO)^{-1} \,  h_T   \; .
	\end{equation}
\end{assumption}
\noindent We also introduce the following mesh parameters
\begin{subequations}
\label{eq:mesh-param}
\begin{align}
h_{\tau} 
&:= \; 
( h_T \, h^{i}_{T})^{1/2} \, , \\
h_{\tG} 
&:= \; 
\max_{T\in \tT : T \cap \tG \not = \emptyset} h_T\,, \\
h_{\tO} 
&:= \; 
\max_{T\in \tT} h_T \, , \\
h_{\perp} 
&:= \; 
\frac{ \mathrm{meas}_{n_d}(T)}{\mathrm{meas}_{n_d - 1}(\ti{E})} \;  \quad \forall \ti{E} \in \tG, \ \ T \cap \ti{E} \neq \emptyset, \ \ T \in \ti{\cT}_h \,.
\end{align}
\end{subequations}
\begin{rem}
	The rate of decay of $\| \bs{d} \|$ needs only to be marginally faster than the one of $h$, that is, {\it $\zeta$ can be set to an arbitrarily small positive number}. For example, this condition can be realized in practice by (iteratively) subdividing each of the edges of the mesh into two, and then slightly shifting the location of the nodes on the surrogate boundary along the direction $\bs{d}$. 
\end{rem}
\begin{rem}
	When computing convergence rates in numerical experiments, we found that it was not necessary to enforce Assumption~\ref{ass:d_asym}, and that a standard mesh refinement in which every edge of the grid is split in half was sufficient. Assumption~\ref{ass:d_asym} should therefore be considered as a technical condition for the proofs rather than a practical condition for computations.
\end{rem}
 \begin{rem}
Assumption~\ref{ass:d_asym} effectively replaces the earlier and much more restrictive assumption $\inf_{\tG} \ti{\bs{n}} \cdot \bs{\nu}  > 0$ in~\cite{main2018shifted0}. The latter, as will be shown in the numerical tests of Section~\ref{sec:2DNumerical_Results} and Section~\ref{sec:3DNumerical_Results}, is typically not verified for complex geometries. The fact that stability and convergence can be established without this restriction is one of the main results in this paper (see Section~\ref{sec:sbm_poisson} and Section~\ref{sec:sbm_stokes}), and paves the way to the application of the SBM to very complex geometry problems. Incidentally, Assumption~\ref{ass:d_asym} is also one of the differences between the unfitted SBM approach presented here and the Universal Meshes Method~\cite{rangarajan2014universal}, an unrelated hybrid fitted/unfitted method that however utilizes closest-point projection algorithms and stability conditions analogous to $\inf_{\tG} \ti{\bs{n}} \cdot \bs{\nu}  > 0$. 
\end{rem}

\subsection{General strategy}
The SBM introduced in~\cite{main2018shifted0} discretizes the governing equations in $\tO$ rather than $\Om$. Consequently, the  challenge would be to consistently enforce the boundary conditions on $\tG$. To this end, the SBM resorts to a first-order Taylor expansion of the concerned variable at the surrogate boundary in order to {\it shift} the boundary condition from $\G$ to $\tG$. \\
To illustrate, consider a scalar field $u$ to be the exact solution to a partial differential equation in $\Om$ with a trace $g$ on $\G$. Assuming $u$ is sufficiently smooth in the strip between $\tG$  and $\G$ so as to admit a first-order Taylor expansion pointwise, we can write
\begin{align}
\label{eq:SBM_BC}
u(\tx)+ (\nabla u \cdot \bs{d} )(\tx) + (R (u, \bs{d}))(\tx) = g(\bs{x}) 
\; , \quad \quad \mbox{ on } \tG \;  .
\end{align}
where the remainder $R(u,\bs{d})$ satisfies
\begin{align}
\nonumber
|\, R(u,\bs{d}) \, | = o(\| \, \bs{d} \, \|) \qquad \text{as } \quad \| \, \bs{d} \, \| \to 0\,.
\end{align}
Introducing the function $\bar{g}(\tx):=g(\bs{M}_h (\tx))$ on $\tG$,  we see that the trace of $u$ on $\tG$ satisfies
\begin{equation}\label{eq:u-g}
u+ \nabla u \cdot \bs{d} - \bar{g} + R(u,\bs{d}) \ = \ \tS u - \bar{g} + R_h u = 0 \,,
\end{equation}
where we have introduced the boundary operator 
\begin{equation}\label{eq:def-bndS}
\tS v := v+ \nabla v \cdot \bs{d} \qquad  \text{on \ } \tG
\end{equation}
and $R_h u$ is a short-hand notation for the Taylor expansion remainder $R(u,\bs{d})$.
Neglecting the higher-order term (with respect to $\| \,  \bs{d} \, \| $) in \eqref{eq:u-g}, we obtain the final expression of the {\it shifted} boundary condition
\begin{equation}\label{eq:Finalu-g}
\tS u = \bar{g}  \, , \quad  \mbox{on } \tG \; ,
\end{equation}
which will be weakly enforced on the discretization $u_h$ of $u$ that we are going to introduce.
Similarly, for a vector field $\bs{u}$, we deduce that its trace $\bar{\bs{g}}$ on $\tG$ satisfies
\begin{equation}\label{eq:bsu-g}
\bs{S}_{h} \bs{u} + \bs{R}_h \bs{u} = \bar{\bs{g}} \,,
\end{equation}
where $\bs{S}_{h} \bs{v} := \bs{v}+ \nabla \bs{v} \, \bs{d}$ on $\tG$ and $\bs{R}_h \bs{u} $ is the Taylor expansion remainder of $\bs{u}$ on $\tG$.
Again, neglecting the higher-order term in \eqref{eq:bsu-g}, we obtain the {\it shifted} vector boundary condition 
\begin{equation}\label{eq:Finalbsu-g}
\bs{S}_{h} \bs{u} = \bar{\bs{g}} \, , \quad  \mbox{on } \tG \; .
\end{equation}

\section{The SBM for the Poisson equation \label{sec:sbm_poisson}}
The strong form of the Poisson problem with a  non-homogeneous Dirichlet boundary condition reads
\begin{subequations}
	\label{eq:SteadyPoisson}
\begin{align}
- \Delta u &=\; f  \qquad \text{\ \ in \ } \Om \; ,
 \\
u &=\;  u_D \qquad \text{on \ }  \G = \partial\Om\; ,
\end{align}
\end{subequations}
where $u$ is the primary variable, $u_D$ its value on the boundary $\G$ and $f$ a body force. 
\subsection{Existence, uniqueness and regularity of the infinite dimensional problem}
Let us denote by $H^{1/2}(\Om)$ a fractional trace space (typically associated with $H^1(\Om)$), and $H^{-1}(\Om)$ the dual (space) of 
$H^{1}_0(\Om)$.
The well-posedness of the infinite dimensional problem is discussed, for example, in~\cite{AErn:2004a} (Theorem 5.1, p. 80), with the following result:
\begin{thm}[Well-posedness of the exact problem]
\label{thm:PoissonExistenceUniqueness}
Let $\Om$ be a bounded and connected open subset of $\mathbb{R}^{n_d}$ with Lipschitz-continuous boundary $\G$.
Given $f \in H^{-1}(\Om)$ and $u_D \in H^{1/2}(\G)$, there exists a unique solution  $u \in H^1(\Om)$ of Problem \eqref{eq:SteadyPoisson}.
Furthermore, if the boundary $\G$ is of class $\mathcal{C}^{2}$,  $f \in L^2(\Om)$ and $u_D \in H^{3/2}(\G)$, then $u\in H^{2}(\Om) $ and satisfies
\begin{align}\label{eq:reg-estim}
\| \, u \, \|_{2,\Om} \leq\; C \, \left(\| \, f \, \|_{0,\Om} + \| \, u_D \, \|_{3/2,\G} \right) 
\end{align}
for a constant $C$ independent of $f$ and $u_D$. 
\end{thm}

As a point of departure in the development of the SBM discretization, we just assume $\Gamma$ to be Lipschitz-continuous, $f \in L^2(\Omega)$ and $u_D \in H^{1/2}(\Gamma)$. Later on, we will make stronger assumptions.

\subsection{Weak discrete formulation}
\label{sec:weak_discr_poisson}
Discretizing Problem~\eqref{eq:SteadyPoisson} in $\tO$ and enforcing~\eqref{eq:Finalu-g} on $\tG$ with $\bar{g} = \bar{u}_D$ through Nitsche's method~\cite{nitscheweak,arnold2002unified}, we deduce the following SBM Galerkin discretization of Problem~\eqref{eq:SteadyPoisson}: 
\begin{quote}
Find $u_h \in V_h(\tO)$ such that, $\forall w_h \in V_h(\tO)$
\begin{align}
\label{eq:DiscreteShiftedNitscheVariationalFormSteadyPoisson}
& ( \nabla u_h \, , \, \nabla w_h )_{\tO} 
- \avg{ \nabla u_h  \cdot \ti{\bs{n}} \, , \,  w_h}_{\tG}
- \avg{ \tS u_h  \, , \, \nabla w_h \cdot \ti{\bs{n}}}_{\tG}
+\avg{ \alpha \,  h_{\perp}^{-1} \,  \tS u_h \, , \,  \tS w_h }_{\tG} 
\nonumber \\ 
& \hskip 3.5cm    =\;
(f \, , \, w_h )_{\tO} 
- \avg{ \bar{u}_D \, , \, \nabla w_h \cdot \ti{\bs{n}} }_{\tG}
+\avg{ \alpha \,  h_{\perp}^{-1} \,  \bar{u}_D \, ,  \, \tS w_h  }_{\tG}
\, ,
\end{align}
 where $ V_h(\tO)  = \; \left\{ v_h \in C^0(\tO)  \ | \ {v_h}_{|T} \in \mathcal{P}^1(T)  \, , \, \forall T \in \ti{\mathcal{T}}_h \right\}  \, $. 
 \end{quote}
In what follows, besides the shape-regularity of the grids, we will assume that there exist two global constants $\xi_{1} $, $\xi_{2} \in \mathbb{R}^{+}$ such that $\xi_{1} \, h \leq h_{\perp} \leq \xi_{2} \, h$. 
With slight abuse of notation, we will always assume that $h$ and $h_{\perp}$ are interchangeable.
For the sake of completeness, we rewrite \eqref{eq:DiscreteShiftedNitscheVariationalFormSteadyPoisson} using the classical notation with linear and bilinear forms: 
\begin{quote}
Find $u_h \in V_h(\tO)$ such that, $\forall w_h \in V_h(\tO)$
\begin{subequations}
\label{eq:SB_Poisson_uns}
\begin{align}
a_h(u_h \, , \, w_h) 
&=\; 
l_h(w_h) 
\; ,
\end{align}
where
\begin{align}
\label{eq:UnsymmetricShiftedNitscheBilinearForm}
a_h(u_h \, , \, w_h) 
&=\; 
( \nabla u_h \, , \, \nabla w_h  )_{\tO} 
- \avg{ \nabla u_h  \cdot \ti{\bs{n}} \, , \, \tS w_h }_{\tG}
-  \avg{ \tS u_h \, , \, \nabla w_h \cdot \ti{\bs{n}} }_{\tG}
\nonumber \\
& \quad 
+ \avg{ \alpha \,  h^{-1} \,  \tS u_h \, , \, \tS w_h }_{\tG}
+ \avg{  \nabla u_h  \cdot \ti{\bs{n}} \, , \, \nabla w_h \cdot \bs{d} }_{\tG}
\; ,
\\[.2cm]
\label{eq:UnsymmetricShiftedNitscheRHS}
l_h(w_h) 
&=\; 
(f \, , \, w_h )_{\tO} 
- \avg{ \bar{u}_D \, , \, \nabla w_h \cdot \ti{\bs{n}} }_{\tG}
+\avg{ \alpha \, h^{-1} \,  \bar{u}_D \, ,  \, \tS w_h  }_{\tG} \; .
\end{align}
\end{subequations}
\end{quote}
\begin{rem}
Despite utilizing a symmetric form of Nitsche's method, the bilinear form $a_h(u_h \, , \, w_h)$ is not symmetric in general, because of the presence of the term $\avg{  \nabla u_h  \cdot \ti{\bs{n}} \, , \, \nabla w_h \cdot \bs{d} }_{\tG}$.
\end{rem}
\begin{rem}
Formulation~\eqref{eq:SB_Poisson_uns} does not include a tangential stabilization term that was introduced ~\cite{main2018shifted0} to help in the proof of coercivity.
This simplifies the implementation and avoids having to calibrate  an additional numerical parameter. 
It will be clearer from what follows that coercivity can be proved by simply relying on Assumption~\ref{ass:d_asym}.
\end{rem}

\subsection{Well-posedness and stability}
The first step in our analysis of the SBM is to prove that, {\it for sufficiently fine grids}, $a_h(u_h\, , \, w_h)$ is uniformly coercive. This will immediately imply the existence and uniqueness of the solution of the discrete SBM problem.
Later on, this coercivity result will be used to prove optimal error convergence in the natural norm. 
\begin{thm}[Coercivity]
\label{thm:coerc-aPoisson}
Consider the bilinear form $a_h(u_h\, , \, w_h)$ defined in~\eqref{eq:SB_Poisson_uns}. If the parameter $\alpha$ is sufficiently large and the quantity $\hat{h}_{\tG}$ defined in \eqref{eq:hhat} is sufficiently small, then there exists a constant $C_a>0$ independent of the mesh size, such that	
\begin{equation} \label{eq:17}
a_h(u_h\,, \,u_h) \geq  C_{a} \, \| \, u_h \, \|^2_{a}  \qquad \forall u_h \in V_h(\tO) \; ,
\end{equation}
where $ \| \, u_h \, \|_{a}^2  = \| \, \nabla u_h \, \|^2_{0,\tO}   + \| \, h^{-1/2} \, \tS u_h \, \|^2_{0,\tG}  \; . $
\end{thm}
\begin{proof}
By substitution, we have
\begin{align}
a_h(u_h \, , \, u_h) 
&=\; 
\| \, \nabla u_h \, \|^2_{0,\tO} 
- 2 \avg{\tS u_h \, , \, \nabla u_h  \cdot \ti{\bs{n}} }_{\tG} 
+ \avg{ \nabla u_h \cdot \bs{d} \, , \, \nabla u_h  \cdot \ti{\bs{n}} }_{\tG}
+ \alpha \, \| \, h^{-1/2} \, \tS u_h \, \|^2_{0, \tG}
 \; .
\end{align}
Young's $\epsilon$-inequality and  the discrete trace inequalities in Theorem \ref{theo:discretetrace} of \ref{sec:appendix_a} yield
\begin{align}
\label{eq:23}
a_h(u_h\, , \, u_h) 
&\geq\; 
(1-\epsilon_{1} \, C_{I}) \, \| \, \nabla u_h \, \|^2_{0,\tO} 
+ \avg{ \nabla u_h \cdot \bs{d} \, , \, \nabla u_h  \cdot \ti{\bs{n}} }_{\tG}
+ \left(\alpha - \epsilon^{-1}_{1}\right) \, \| \, h^{-1/2} \, \tS u_h \, \|^2_{0, \tG}
\; .
\end{align}
Applying Assumption~\ref{ass:d_asym} to \eqref{eq:23} and recalling that $\bs{d} = \| \, \bs{d} \, \| \bs{\nu}$ give
\begin{align}
a_h(u_h\, , \,u_h) 
&\geq\; 
(1-\epsilon_{1} \, C_{I}) \, \| \, \nabla u_h \, \|^2_{0,\tO} 
- \| \, \|\bs{d} \,\|^{1/2} \, \nabla u_h \cdot \bs{\nu} \, \|_{0,\tG} \, \| \, \| 
\, \bs{d} \, \|^{1/2} \, \nabla u_h \cdot \ti{\bs{n}} \, \|_{0,\tG} 
\nonumber \\
&\phantom{\geq}\;
+ \left(\alpha - \epsilon^{-1}_{1} \, \right) \, \| \, h^{-1/2} \, \tS u_h \, \|^2_{0, \tG}
\nonumber \\
&\geq\;
(1-\epsilon_{1} \, C_{I}) \, \| \, \nabla u_h \, \|^2_{0,\tO} 
- \| \, \| \, \bs{d}\, \|^{1/2} \, \nabla u_h \, \|^{2}_{0, \tG} 
+ \left(\alpha -  \epsilon^{-1}_{1} \right) \, \| \, h^{-1/2} \, \tS u_h \,  \|^2_{0, \tG}
\nonumber \\
&\geq\;
(1-\epsilon_{1}\, C_{I} -c_{d} \, \hat{h}_{\tG}^{\zeta} \, C_{I}) \,	\| \,  \nabla u_h \, \|^2_{0,\tO} 
+ \left(\alpha -  \epsilon^{-1}_{1} \right) \, \| \, h^{-1/2} \, \tS u_h \, \|^2_{0, \tG}
\; .
\end{align}
Taking $\epsilon_{1} = (2 \, C_{I})^{-1}$ and considering sufficiently fine grids so that $c_{d} \, \hat{h}_{\tG} \leq (4 \, C_{I})^{-1}$ imply
\begin{align}
a_h(u_h\, , \,u_h) 
&\geq\;
4^{-1}	\, \| \, \nabla u_h \, \|^2_{0,\tO} 
+ \left(\alpha - 2C_{I} \right) \, \| \, h^{-1/2} \, \tS u_h \, \|^2_{0, \tG}
\; .
\end{align}
Enforcing $\alpha > 2 \, C_{I}$ yields the coercivity statement \eqref{eq:17} with  $C_{a}  =\;  \min \left(  1/4 \, , \,  \alpha - 2C_{I} \right) \; .$
\end{proof}
\begin{rem}
In the case of linear polynomials, the discretization scheme \eqref{eq:SB_Poisson_uns} bears some similarities to the one proposed in \cite{bramble1972projection}, a reference that the authors of~\cite{main2018shifted0,main2018shifted} were unaware of.
In the present work, however, we admit a much greater freedom in the choice of the mapping $\bs{M}_{h}: \tO \to \Omega$, that is, our vector $\bs{d}$ {\it does not need} to be aligned with $\ti{\bs{n}}$ as in \cite{bramble1972projection} and can be chosen in such a way to account for general domains whose boundaries have a finite number of corners and/or edges~\cite{TheoreticalPoissonAtallahCanutoScovazzi2020}. 
Although this difference seems mild at first sight, in reality it makes the work in~\cite{bramble1972projection} of not easy applicability in practical engineering problems, and this might explain why a work dating almost half a century ago has gone relatively unnoticed in the community.
On the other hand, our Assumption \ref{ass:d_asym} is comparable to the general assumption made in  \cite{bramble1972projection} on the distance between the surrogate and physical boundaries; correspondingly, our coercivity result, Theorem \ref{thm:coerc-aPoisson} above, is similar to Lemma 6 in~\cite{bramble1972projection}.

Interestingly, much of the emphasis in~\cite{bramble1972projection} and later works~\cite{bramble1996finite,bramble1994robust,burman2018cut} was on perturbations of body-fitted grids, for which $\|\bs{d}\| \sim h^2$, despite many of the results of~\cite{bramble1972projection} apply in a broader sense. From our perspective, the case $\|\bs{d}\| \sim h^{1+\zeta}$ - for $\zeta$ positive and arbitrarily small - is significantly more interesting in computational engineering applications.
\end{rem}

\subsection{Consistency and convergence analysis  \label{sec:ShiftedNitscheAccuracyAnalysis}}
From now on, we pose the following regularity assumption on the exact solution $u$ of our Dirichlet problem (see \cite{TheoreticalPoissonAtallahCanutoScovazzi2020}):
\begin{assumption}
	\label{ass:reg-dom}
Assume that $\G$ is of class $\mathcal{C}^{2}$,  $f \in L^2(\Om)$ and $u_D \in H^{3/2}(\G)$.
\end{assumption}
As a consequence, Theorem \ref{thm:PoissonExistenceUniqueness} guarantees $u \in H^2(\Omega)$ with  the estimate \eqref{eq:reg-estim}. This assumption allows us to keep technicalities at a minimum in the subsequent consistency and convergence analysis. However, we mention that it can be weakened to include e.g. domains with a finite number of corners and/or edges (see \cite{TheoreticalPoissonAtallahCanutoScovazzi2020}), at the expense of more cumbersome mathematical arguments.

The proof of convergence of the SBM for Problem~\eqref{eq:SteadyPoisson} using the natural norm $\| \, \cdot \, \|_{a}$, relies on Strang's Second Lemma. 
With this goal, we introduce the infinite dimensional space 
\begin{align}
V(\tO;\tT) & = V_h(\tO) +  H^2(\tO)\;  \subset H^2(\tO;\ti{\mathcal{T}}_h)\, \;  , 
\end{align}
 an extension of the finite dimensional space $V_h(\tO)$ that contains the exact solution $u$, that is $u \in V(\tO;\tT)$. 
Here $H^2(\tO;\ti{\mathcal{T}}_h)= \prod_{T \in \ti{\mathcal{T}}_h} H^2(T)$ with `broken' norm $\| \cdot \|_{2,\tO, \ti{\mathcal{T}}_h}= \sum_{T \in \ti{\mathcal{T}}_h} \| \cdot \|_{2,T}$.
It is easily checked that the form $a_h(\cdot,\cdot)$ is well-defined also on the space $V(\tO;\tT) \times V_h(\tO)$. We associate to  $V(\tO;\tT)$ the norm
\begin{align}
\| \, v \, \|^2_{V(\tO;\tT)} 
& = 
\| \, v \, \|^2_{a} 
+  |\, h\, w \, |^2_{2,\tO,\ti{\mathcal{T}}_h} \, ,
\end{align}
and we note that if $v  \in V_h(\tO)$, then $\| \, v \, \|_{V(\tO;\tT)} = \| \,v \, \|_{a}$.
At this point, we are ready to state  
\begin{lemma}[Strang's Second Lemma]
	\label{thm:SecondStrangLemma}
	If $u_h \in V_h(\tO)$ is the solution of \eqref{eq:SB_Poisson_uns}, then
\begin{subequations}
\begin{align}\label{eq:strang}
\| \,u - u_h \, \|_{V(\tO;\tT)} 
& \leq 
\left( 1 + C^{-1}_{a} \, \|\, a_h \, \|_{V_h(\tO) \times V_h(\tO)} \right) \, E_{a,h}(u)
+ C^{-1}_{a} \, E_{c,h}(u)
\; .
\end{align}
where 
\begin{equation}\label{eq:strang-2}
E_{a,h}(u) = \inf_{w_h \in V_h(\tO)} \| \, u -w_h \, \|_{V(\tO;\tT)}
\end{equation}
is the approximation error and
\begin{equation}\label{eq:strang-3}
E_{c,h}(u)  = \sup_{w_h \in V_h(\tO)} \frac{| \, l_h(w_h)-a_h(u\, , \,w_h) \, |}{\| \, w_h \, \|_{V(\tO;\tT)}}
\end{equation}
\end{subequations}
is the consistency error.
\begin{proof}
The proof is classic, see, e.g., \cite{AErn:2004a} for details.
\end{proof}
\end{lemma}
From Theorem~\ref{thm:coerc-aPoisson}, we already have that $C_{a}$ is a constant independent of the mesh size. Hence, to estimate the discretization error in the norm $\| \cdot \|_{V(\tO;\tT)}$ we need to prove that $\|\, a_h \, \|_{V_h(\tO) \times V_h(\tO)}$ is bounded from above and to estimate the approximation and consistency errors in terms of the mesh size $h_{\tO}$. 
\begin{prop}[Boundedness]
	\label{thm:PoissonContinuity}
	There exists a constant $C_{\cal A} >0$, independent of the mesh size, such that 
	\begin{align}
	a_h(u, w ) 
	\leq 
	C_{\cal A}\, \|\,  u \, \|_{V(\tO;\tT)} \; \|\,  w \, \|_{V(\tO;\tT)}  \qquad \forall u,w \in V(\tO;\tT)\;. 
	\end{align}
\end{prop}
\begin{proof}
Using  the Cauchy-Schwartz inequality and Theorem~\ref{theo:BoundaryTraceIneq}, we obtain:
\begin{align}
a_h(u\, , \, w) 
&=\; 
( \nabla u \, , \, \nabla w )_{\tO} 
- \avg{ \nabla u  \cdot \ti{\bs{n}} \, , \, \tS w }_{\tG}
- \avg{ \tS u \, , \,  \nabla w \cdot \ti{\bs{n}}}_{\tG}
+\avg{ \nabla u \cdot \ti{\bs{n}} \, , \,  \nabla w \cdot \bs{d} }_{\tG}
+\avg{ \alpha \, h^{-1} \, \tS u  \, , \, \tS w }_{\tG}
\nonumber \\
&\leq\; 
 \| \, \nabla u \, \|_{0,\tO} \, \| \, \nabla w \, \|_{0,\tO}
+ \| \, h^{1/2} \, \nabla u  \cdot \ti{\bs{n}} \, \|_{0,\tG}  \, \| \, h^{-1/2} \, \tS w \, \|_{0,\tG} 
+  \| \, h^{-1/2} \, \tS u \, \|_{0,\tG}  \, \| \, h^{1/2} \,  \nabla w \cdot \ti{\bs{n}} \, \|_{0,\tG}
\nonumber \\
& \quad 
+ \|  \, h^{1/2} \, \nabla u  \cdot \ti{\bs{n}} \, \|_{0,\tG} \, \|  \, h^{1/2} \,  \nabla w \cdot \bs{\nu} \, \|_{0,\tG}
+ \alpha \,  \| \, h^{-1/2} \, \tS u \, \|_{0,\tG} \, \| \, h^{-1/2} \, \tS w \, \|_{0,\tG}
\nonumber \\
&\leq\; 
C_{\cal A}\, \|\,  u \, \|_{V(\tO;\tT)} \; \|\,  w \, \|_{V(\tO;\tT)} \;  , 
\end{align}
with $C_{\cal A}= 1+2C_I + C^{2}_{I}+\alpha$. 
\end{proof}
\noindent The previous result implies that
\begin{align}
\|\, a_h \, \|_{V(\tO;\tT) \times V(\tO;\tT)}  =  \sup_{u \in V(\tO;\tT)} \;  \sup_{w_h \in V(\tO;\tT)} \;  \frac{a_h(u\, , \, w) }{\|\,  u \, \|_{V(\tO;\tT)} \; \|\,  w \, \|_{V(\tO;\tT)} }
 \leq 
 C_{\cal A} \; .
\end{align}
\begin{prop}[Approximability]
\label{thm:PoissonApproximation}
There exists a constant $C_{APP}>0$, independent of the mesh size, such that
\begin{equation}\label{eq:Approximation}
E_{a,h}(u) \leq C_{APP} \, h_{\tO} \, \| \, \nabla (\nabla u) \, \|_{0, \tO} \qquad \forall u \in V(\tO;\tT) \, .
\end{equation}
\end{prop}
\begin{proof} 
Let $ w_h= \mathcal{I}_h u $ in~\eqref{eq:strang-2},  where $\mathcal{I}_h u$ is the standard piecewise-linear Lagrange interpolant of $u$ on the triangulation $\ti{\cal T}_h$. Consequently, the goal is to estimate 
\begin{align}
\label{eq:interpolationError}
\| \, u -\mathcal{I}_h u \, \|_{V(\tO;\tT)} 
&=\;
\| \, \nabla (u -\mathcal{I}_h u) \, \|_{0,\tO} 
+ \| \, h^{-1/2} \, \tS (u -\mathcal{I}_h u) \, \|_{0,\tG}
+ |\, h\, (u -\mathcal{I}_h u) \, |_{2,\tO,\ti{\mathcal{T}}_h}
\nonumber \\
&\leq \;
\| \, \nabla (u -\mathcal{I}_h u) \, \|_{0,\tO} 
+ \| \, h^{-1/2} \, \tS (u -\mathcal{I}_h u) \, \|_{0,\tG}
+h_{\tO} \, \| \, \nabla (\nabla u) \, \|_{0, \tO} \; .
\end{align}
We begin by stating a classical interpolation result
\begin{align}
\label{eq:ClassInterpolation}
\| \, h^{-1} \, (u -\mathcal{I}_h u)  \, \|_{0,\tO} 
+ \|  \,  \nabla (u -\mathcal{I}_h u)  \, \|_{0,\tO} 
\leq 
C_1 \,  h_{\tO}  \, \| \, \nabla (\nabla u) \, \|_{0, \tO} \;,
\end{align}
where $C_1$ is a positive constant independent of the mesh size. Applying Theorem~\ref{theo:BoundaryTraceIneq} and Assumption~\ref{ass:d_asym}, 
\begin{align}
\label{eq:interpolationBound}
\| \, h^{-1/2} \, \tS (u -\mathcal{I}_h u) \, \|_{0,\tG} 
& \leq \;
 \| \, h^{-1/2} \, (u -\mathcal{I}_h u) \, \|_{0,\tG} 
+ \hat{h}_{\tG}^{\zeta} \, \| \, h^{1/2} \, \nabla (u -\mathcal{I}_h u) \, \|_{0,\tG} 
\nonumber \\  
&  \leq\;
C_{I} \left( \,  \| \, h^{-1} \, (u -\mathcal{I}_h u) \, \|_{0,\tO} 
+ ( 1 + \hat{h}_{\tG}^{\zeta} ) \, \| \, \nabla (u -\mathcal{I}_h u) \, \|_{0,\tO} 
\right.
\nonumber \\ 
&  \phantom{\leq} \;
\qquad \ \ \left. +  h_{\tO}   \, \hat{h}_{\tG}^{\zeta} \| \, \nabla (\nabla u) \, \|_{0, \tO}  \right)
\nonumber \\ 
&  \leq
C_{2} \, h_{\tO} \, \| \, \nabla (\nabla u) \, \|_{0, \tO} \; .
\end{align}
Replacing~\eqref{eq:interpolationBound} in~\eqref{eq:interpolationError} gives~\eqref{eq:Approximation} with $C_{APP} = \max(1,C_1,C_2)$. 
\end{proof}
In order to bound the consistency error $E_{c,h}(u)$, the following Lemma is needed. 
\begin{lemma}\label{lem:errorDir}
Under Assumptions~\ref{ass:d_asym} and~\ref{ass:reg-dom}, there exists a constant $C_D>0$ independent of the mesh size such that, for $h_{\tG}$ sufficiently small, it holds
\begin{align}
\label{eq:taylorerror}
\| \, h^{-1/2} R_h u  \, \|_{0, \tG} 
=
\| \, h^{-1/2}( \tS u - \bar{{u}}_D) \, \|_{0, \tG} 
\leq\;  
C_D \, h_{\tG}  \, \| \, \nabla(\nabla u) \, \|_{0,\Om \setminus \tO} \; . 
\end{align}
\end{lemma}
\begin{proof}
The result follows from a careful bound of the remainder in the first-order Taylor expansion \eqref{eq:SBM_BC}. For a detailed proof in the more general setting, see Proposition 3 in~\cite{TheoreticalPoissonAtallahCanutoScovazzi2020}. 
\end{proof}
\begin{rem}
	The estimate in Lemma~\ref{lem:errorDir} can be modified to handle general domains containing corners and/or edges. Naturally, this entails relaxing Assumption~\ref{ass:reg-dom} (see Lemma 3 in~\cite{TheoreticalPoissonAtallahCanutoScovazzi2020}).  
\end{rem}
\begin{rem}
	Lemma~\ref{lem:errorDir} can be proved without Assumption~\ref{ass:reg-dom}, by introducing additional assumptions on the mapping $\bs{M}_{h}:\tG \to \G$ (see~\cite[Assumption 5]{atallah2020analysis}).
\end{rem}
\begin{prop}[Consistency error]
\label{thm:PoissonConsistency}
Under the hypotheses of Lemma~\ref{lem:errorDir},  there exists a constant $C_{PBL}>0$ independent of the mesh size and $u$ such that
\begin{equation}\label{eq:PoissonConsistency}
 E_{c,h}(u) \leq\; C_{PBL} \, h_{\tG} \,  \| \, \nabla (\nabla u) \, \|_{0, \Om \setminus \tO} \; .
	\end{equation}
\end{prop}
\begin{proof}
Integrating by parts the bilinear form $a_h(\cdot\, , \, \cdot)$ given in \eqref{eq:UnsymmetricShiftedNitscheBilinearForm} and applying Lemma~\ref{lem:errorDir} yield
\begin{align}
\label{eq:IntegByPartConsistency}
|a_h(u\, , \,w_h)-l_h(w_h) |
&=\; 
|- \avg{ \, h^{-1/2} \, (\tS u - \bar{u}_D) \, , \, \, h^{1/2} \, \nabla w_h \cdot \ti{\bs{n}} }_{\tG}
+ \alpha \, \avg{\, h^{-1/2} \, (\tS u - \bar{u}_D)  \, , \, h^{-1/2} \, \tS w_h }_{\tG} |
\nonumber \\
& \leq \; 
C_{PBL} \, h_{\tG} \, \| \, \nabla(\nabla u) \, \|_{0,\Om \setminus \tO}  \, \| \, w_h  \, \|_{V(\tO;\tT)} \ \;.
\end{align}
The proof is concluded by replacing~\eqref{eq:IntegByPartConsistency} in~\eqref{eq:strang-3}.
\end{proof}
\begin{thm}[Convergence in the natural norm]
\label{thm:PoissonConvergenceNatural}
Under Assumptions \ref{ass:d_asym} and \ref{ass:reg-dom}, and the condition that $h_{\tG}$ is sufficiently small, the numerical solution $u_h$ of the SBM \eqref{eq:DiscreteShiftedNitscheVariationalFormSteadyPoisson} satisfies the following error estimate:
\begin{align}
\label{eq:convergenceEstimate}
\| \, u - u_h \, \|_{V(\tO;\tT)}  \leq C\, h_{\tO} \, \| \, \nabla (\nabla u) \,  \|_{0, \Omega}  
	\; ,
\end{align}
where $u$ is the exact solution of Problem \eqref{eq:SteadyPoisson} and $C>0$ is a constant independent of the mesh size and the solution. 
\end{thm}
\begin{proof}
Combining Strang's abstract error estimate \eqref{eq:strang} with the stability bound \eqref{eq:17} and error estimates in Proposition~\ref{thm:PoissonApproximation} and Proposition~\ref{thm:PoissonConsistency} gives~\eqref{eq:convergenceEstimate}.
\end{proof}
\subsection{$L^2$-error estimate by a duality argument \label{sec:L2-temperature}}
In this section, we provide an estimate of the $L^{2}$-norm of the discretization error $u - u_{h}$ based on the ideas recently proposed in ~\cite{TheoreticalPoissonAtallahCanutoScovazzi2020}. This is a considerable improvement to the same work given in~\cite{main2018shifted0} and~\cite{atallah2020analysis} as we avoid the convexity assumption of the surrogate domain $\tO$. We start with the following two preliminary results:
\begin{lemma}\label{lemma:non-symmetry}
It holds
\begin{align}\label{eq:non-symmetry}
a_h(w\, , \, v)-a_h(v\, , \, w) = 
\avg{  \nabla w  \cdot \ti{\bs{n}} \, , \, \nabla v \cdot \bs{d} }_{\tG} 
- \avg{  \nabla v  \cdot \ti{\bs{n}} \, , \, \nabla w \cdot \bs{d} }_{\tG} 
\, \qquad \forall v,w \in V(\tO;\tT) \,.
\end{align}
\end{lemma}
\begin{proof}
By substitution in~\eqref{eq:UnsymmetricShiftedNitscheBilinearForm}.
\end{proof}
Lemma~\ref{lemma:non-symmetry} quantifies the symmetry discrepancy in the bilinear form $a_h(\cdot \, , \, \cdot) $.
The following lemma quantifies the gap in Galerkin orthogonality.
\begin{lemma}\label{lemma:non-Gal-orth} 
Let $u$ be the solution of \eqref{eq:SteadyPoisson}, and $u_h$ the numerical solution of \eqref{eq:DiscreteShiftedNitscheVariationalFormSteadyPoisson}. It holds
\begin{align}\label{eq:non-Gal-Orth}
a_h(u-u_h\, , \,  v_h) 
= 
\avg{R_h u \, , \, \nabla v_h \cdot \ti{\bs{n}} }_{0,\tG} 
- \avg{ \alpha \, h^{-1} \, R_h u \, , \, \tS v_h}_{0,\tG} 
\,  \qquad \forall v_h \in V_h(\tO) \; .
\end{align}
\end{lemma}
\begin{proof}
From~\eqref{eq:SB_Poisson_uns}, and following the same steps as in Proposition~\ref{thm:PoissonConsistency}, we obtain
\begin{align}
a_h(u-u_h\, , \, v_h)  &= a_h(u\, , \, v_h) - a_h(u_h\, , \, v_h) 
\nonumber \\ 
&=
a_h(u\, , \, v_h) - l_h(v_h) 
\nonumber \\ 
&=
- \avg{ \tS u - \bar{u}_D \, , \,  \nabla v_h \cdot \ti{\bs{n}} }_{\tG}
+ \avg{\alpha \, h^{-1}  \, (\tS u - \bar{u}_D)  \, , \, \tS v_h }_{\tG} \, .
\end{align}
Finally, from~\eqref{eq:u-g} with $\bar{g} = \bar{u}_D$, we have $\tS u - \bar{u}_D = - R_h u $, which concludes the proof.
\end{proof}
\begin{thm}[Enhanced $L^{2}$-error estimate]
\label{thm:PoissonDuality}
Assume the hypotheses of Theorem~\ref{thm:PoissonConvergenceNatural} to hold. Then, the numerical solution $u_h$ of \eqref{eq:DiscreteShiftedNitscheVariationalFormSteadyPoisson} satisfies the following error estimate:  
\begin{align}
\label{eq:SGSBM_duality_final}
\| \, u - u_h \, \|_{\tO} 
\leq & \;
C \,  h_{\tO}^{3/2} \, l(\tO)^{1/2}  \, \| \, \nabla (\nabla u) \, \|_{0,\Om}  
 \; ,
\end{align}
where $C$ is a positive constant independent of mesh size.
\end{thm}
\begin{proof}
Given $z \in L^2(\tO)$, let $\bar{z} \in L^2(\Om)$ be its extension by $0$ outside $\tO$ and let $\psi$  be the solution to the following homogeneous Dirichlet problem on $\Om$:
\begin{subequations}
\label{eq:strong_dual}
\begin{align}
-\Delta \psi
&=\; 
\bar{z}
\,  \ \ \!  \qquad  \qquad \, \mbox{in } \Om \; , 
\\
\psi
&=\; 
0
\, \ \ \! \qquad \qquad  \mbox{on } \G \; .
\end{align}
\end{subequations}
Recalling Theorem \ref{thm:PoissonExistenceUniqueness}, the stated assumptions in addition to the fact that $\bar{z} \in L^2(\Om)$ imply the regularity result $\psi \in H^2(\Om) \cap H^1_0(\Om)$, with the following bound  
\begin{align}\label{eq:reg-dual}
\| \, \psi \, \|_{2, \tO} 
\leq \;
 \| \, \psi \, \|_{2, \Om} 
\leq\;
Q \, \| \, \bar{z} \, \|_{0, \Om} 
=\;
Q \, \| \, z \, \|_{0, \tO}  \; ,
\end{align}
where $Q>0$ is a non-dimensional constant independent of $\bar{z}$ and the mesh size.
The same arguments that led to~\eqref{eq:u-g} show that $\psi$ satisfies
\begin{align}
\label{eq:dualTaylor}
\tS \psi + R_h \psi  = 0 \qquad \mbox{on } \tG \; .
\end{align}
We now apply Proposition 3 in~\cite{TheoreticalPoissonAtallahCanutoScovazzi2020}, which together with \eqref{eq:reg-dual} yields the bound
\begin{align}
\label{eq:prop3_theoretical_cc}
\| \, h^{-1/2} \, R_h \psi \, \|_{0,\tG} 
\leq 
C_{DR} \, h_{\G} \, | \, \psi \, |_{2,\Om \setminus \tO} 
\leq 
C_{DR} \, Q \, h_{\tO} \, \| \, z  \, \|_{0,\tO} \, ,
\end{align} 
where $C_{DR}$ is a positive constant independent of the mesh size.  

Next, consider any $q \in  V(\tO;\tT)$  and note that  $\psi$ also satisfies the following variational statement:
\begin{align}
(z \, , \, q)_{\tO}
= \;
-(\Delta \psi \, , \, q)_{\tO}
=
( \nabla \psi \, , \, \nabla q)_{\tO} 
- \avg{  \nabla \psi \cdot \ti{\bs{n}} \, , \, q }_{\tG} \; .
\end{align}
Adding residual terms on $\tG$ that vanish by definition if applied to the exact solution, we have:
\begin{align}
(z \, , \, q)_{\tO}
& = \;
( \nabla \psi \, , \, \nabla q)_{\tO} 
- \avg{  \nabla \psi \cdot \ti{\bs{n}} \, , \, q }_{\tG} 
- \avg{  \tS \psi + R_h \psi \, , \, \nabla q \cdot \ti{\bs{n}} }_{\tG} 
+ \avg{ \alpha \, h^{-1} \,  (\tS \psi + R_h \psi) \, , \, \tS q  }_{\tG} 
\nonumber \\ 
& = \;
a_h(\psi \, , \,  q)
- \avg{  R_h \psi \, , \, \nabla q \cdot \ti{\bs{n}} }_{\tG} 
+ \avg{ \alpha \, h^{-1}\,  R_h \psi \, , \, \tS q  }_{\tG} 
\nonumber \\ 
& = \;
a_h(q \, , \, \psi)
+ \avg{  \nabla \psi  \cdot \ti{\bs{n}} \, , \, \nabla q \cdot \bs{d} }_{\tG} 
- \avg{   \nabla q  \cdot \ti{\bs{n}} \, , \,  \nabla \psi \cdot \bs{d} }_{\tG} 
- \avg{  R_h \psi \, , \, \nabla q \cdot \ti{\bs{n}} }_{\tG} 
\nonumber \\ 
& \phantom{=} \;
+ \avg{ \alpha \, h^{-1} \,  R_h \psi \, , \, \tS q  }_{\tG}  \; ,
\end{align}
where in the last equality we used Lemma~\ref{lemma:non-symmetry}. Picking $q = z = e_{u} := u - u_h$ and using Lemma~\ref{lemma:non-Gal-orth} with $v_h = \psi_{I} := \mathcal{I}_h \psi$,
\begin{subequations}
\begin{align}
\label{eq:dua1}
\| \, e_u \, \|^{2}_{0,\tO}
& = \;
a_h(e_u \, , \, \psi)
+ E_{sym}(e_u \, , \, \psi)
+ E_{rem}(e_u \, , \, \psi)
\nonumber \\ 
& = \;
a_h(e_u \, , \, \psi - \psi_{I})
+ E_{sym}(e_u \, , \, \psi)
+ E_{rem}(e_u \, , \, \psi)
+ E_{ort}(u \, , \, \psi_{I}) \; ,
\end{align}
with 
\begin{align}
E_{sym}(e_u \, , \, \psi) 
& : = \;  
\avg{  \nabla \psi  \cdot \ti{\bs{n}} \, , \, \nabla e_u  \cdot \bs{d} }_{\tG} 
- \avg{   \nabla e_u   \cdot \ti{\bs{n}} \, , \,  \nabla \psi \cdot \bs{d} }_{\tG}  \; ,
\\[.2cm]
E_{rem}(e_u \, , \, \psi) 
& : = \;  
- \avg{  R_h \psi \, , \, \nabla e_u \cdot \ti{\bs{n}} }_{\tG} 
+ \avg{ \alpha \, h^{-1} \,  R_h \psi \, , \, \tS e_u  }_{\tG}  \; ,
\\[.2cm]
E_{ort}(u \, , \, \psi_{I}) 
& : = \;  
\avg{R_h u \, , \, \nabla \psi_{I} \cdot \ti{\bs{n}} }_{0,\tG} 
- \avg{ \alpha \,  h^{-1} \, R_h u \, , \, \tS \psi_{I}}_{0,\tG} \; .
\end{align}
\end{subequations}

Next, we proceed to bound the four error terms on the right-hand side of \eqref{eq:dua1}.
Applying Proposition~\ref{thm:PoissonContinuity} and 
Proposition~\ref{thm:PoissonApproximation} (with $\psi$ in place of $u$), yields 
\begin{subequations}
\begin{align}
\label{eq:eBilinear}
a_h(e_u \, , \,  \psi-\psi_I ) 
&\leq\;
C_{\cal A}\, \|\,  e_u \, \|_{V(\tO;\tT)}  \, \|\,  \psi - \psi_I \, \|_{V(\tO;\tT)} 
\nonumber \\
&\leq\;
C_{\cal A}\, C_{APP} \,  h_{\tO} \, \|\,  e_u \, \|_{V(\tO;\tT)}  \, \| \, \nabla (\nabla \psi) \, \|_{0, \tO} 
\nonumber \\ 
& \leq\;
C_3 \, h_{\tO} \,  \| \,e_u \, \|_{V(\tO;\tT)} \, \| \, e_u \, \|_{0,\tO} \;.
\end{align}
Recalling Assumption~\ref{ass:d_asym}, Lemma~\ref{lem:errorDir}, equation~\eqref{eq:prop3_theoretical_cc}, Theorem~\ref{thm:TraceTheorem}, and Theorem~\ref{theo:BoundaryTraceIneq},
\begin{align}
\label{eq:eSym}
E_{sym}(e_u \, , \, \psi)
&\leq\;
c_{d} \, \hat{h}_{\tG}^{\zeta} \, 
\left( 
\|  \, h^{1/2} \, \nabla e_u  \cdot \bs{\nu} \, \|_{0,\tG}  \, \| \, h^{1/2} \, \nabla \psi  \cdot \ti{\bs{n}} \, \|_{0,\tG}  
+ \|  \, h^{1/2} \, \nabla e_u  \cdot \ti{\bs{n}} \, \|_{0,\tG}  \, \| \, h^{1/2} \, \nabla \psi  \cdot \bs{\nu} \, \|_{0,\tG}  
\right)
\nonumber \\
&\leq\;
c_{d} \, C \,  C_{I} \, h_{\tO}^{1/2} \, l(\tO)^{-3/2} \, 
\| \,e_u \, \|_{V(\tO;\tT)} \, \| \, \psi \, \|_{2,\tO} 
\nonumber \\ 
& \leq\;
C_4 \,  h_{\tO}^{1/2} \, l(\tO)^{1/2} \, 
 \| \,e_u \, \|_{V(\tO;\tT)} \, \| \, e_u \, \|_{0,\tO} \;.
\\[.2cm]
\label{eq:eRem}
E_{rem}(e_u \, , \, \psi)
&\leq\;
\left( 
\|  \, h^{1/2} \, \nabla e_u  \cdot \ti{\bs{n}} \, \|_{0,\tG}  
+ \alpha \, \|  \, h^{-1/2} \, \tS e_u \, \|_{0,\tG} 
\right) 
\| \, h^{-1/2} \,  R_h \psi \, \|_{0,\tG}  
\nonumber \\
&\leq\;
\left( C_I + \alpha \right)  \, 
\| \,e_u \, \|_{V(\tO;\tT)}  \, 
\| \, h^{-1/2} \,  R_h \psi \, \|_{0,\tG}   
\nonumber \\
&\leq\;
C_{5} \, h_{\tO} \, 
\| \,e_u \, \|_{V(\tO;\tT)}  \, \| \, e_u \, \|_{0,\tO} \;.
\\[.2cm]
\label{eq:eOrt}
E_{ort}(u \, , \, \psi_{I}) 
& = \; 
- \avg{R_h u \, , \, \nabla ( \psi - \psi_{I} ) \cdot \ti{\bs{n}} - \nabla \psi  \cdot \ti{\bs{n}} }_{0,\tG} 
+ \avg{ \alpha \,  h^{-1} \, R_h u \, , \, \tS ( \psi - \psi_{I}) - \tS  \psi}_{0,\tG} 
\nonumber \\ 
& \leq \;
\| \, h^{-1/2} \,  R_h u \, \|_{0,\tG} \, 
\left( 
\|  \, h^{1/2} \, \nabla ( \psi - \psi_{I} ) \cdot \ti{\bs{n}}  \, \|_{0,\tG}
+   \|  \, h^{1/2} \, \nabla \psi \cdot \ti{\bs{n}}  \, \|_{0,\tG}
\right)
\nonumber \\
& \phantom{=} \;
+ \alpha \, \| \, h^{-1/2} \,  R_h u \, \|_{0,\tG} \, 
 \left( 
\|  \, h^{-1/2} \, \tS ( \psi - \psi_{I} )  \, \|_{0,\tG}
+  \|  \, h^{-1/2} \, \tS \psi  \, \|_{0,\tG}
\right)
\nonumber \\
& \leq \; 
C_6 \, h_{\tO}^{3/2}  \, l(\tO)^{1/2} \,  
\| \, \nabla(\nabla u) \, \|_{0,\Om \setminus \tO} \, 
 \|  \, e_u  \, \|_{0,\tO}  \, .
\end{align}
\end{subequations}
 Thus, substituting~\eqref{eq:eBilinear},~\eqref{eq:eSym},~\eqref{eq:eRem} and~\eqref{eq:eOrt} in~\eqref{eq:dua1}, we obtain 
 \begin{align}
\| \, u - u_h \, \|_{0,\tO} 
\leq 
C_{AN} \, 
 h_{\tO}^{1/2} \, l(\tO)^{1/2} \,  \| \, u -u_h \, \|_{V(\tO;\tT)} + C_6 \, h_{\tO}^{3/2}  \, l(\tO)^{1/2} \,  
\| \, \nabla(\nabla u) \, \|_{0,\Om \setminus \tO} 
\; ,
\end{align}
where the the right hand side can be bound by a direct application of Theorem~\ref{thm:PoissonConvergenceNatural} to conclude the proof.
\end{proof}

\begin{rem}
The bound obtained in Theorem~\ref{thm:PoissonDuality} is suboptimal, since for a body-fitted Nitsche discretization one would obtain quadratic convergence in the $L^2$-norm of the error. 
However, it is not clear if the above estimate is sharp, since in computations we always observe optimal, second-order convergence rates.
A careful inspection of the above proof indicates that the only non-optimal bound is \eqref{eq:eSym}; it is likely that further cancellations occur in this term, while $\tG$ approaches $\G$ is a smooth way.
\end{rem}

\section{The SBM for the Stokes flow equations \label{sec:sbm_stokes}}
The strong form of the Stokes flow equations with non-homogeneous Dirichlet and Neumann boundary conditions read
\begin{subequations}
\label{eq:SteadyStokes}
\begin{align}
- \nabla \cdot ( 2 \mu \, \bs{\varepsilon}(\bs{u}) - p \bs{I}  ) &=\; \bs{f} 
\qquad \text{\ \ in \ } \Om \; , \\
\nabla \cdot  \bs{u}  &=\; 0
\qquad \text{\ \ in \ } \Om \; , \\
\bs{u} &=\; \bs{u}_D
\qquad \text{on \ }  \G_{D} \; , \\ 
\label{eq:tractionB}
( 2 \mu \, \bs{\varepsilon}(\bs{u}) - p \bs{I}  ) \cdot \bs{n} &=\; \bs{t}_N
\qquad \text{on \ }  \G_{N} \; ,
\end{align}
\end{subequations}
where $\bs{\varepsilon}(\bs{u}) = 1/2 (\nabla \bs{u} + \nabla \bs{u}^T)$ is the velocity strain tensor (i.e., the symmetric gradient of the velocity), $\mu>0$ is the dynamic viscosity, $p$ is the pressure, $\bs{f}$ is a body force, $\bs{u}_D$ is the value of the velocity on the Dirichlet boundary $\G_D \not = \emptyset$ and $\bs{t}_N$ is the vector-valued normal stress on the Neumann boundary $\G_N$ (where $\partial\Om = \G = \overline{\G_D \cup \G_N}$ with $\G_D \cap \G_N  = \emptyset$). The Stokes flow represents a prototype for the application of the SBM to systems of differential equations in mixed form. 
\subsection{Existence, uniqueness and regularity of the infinite dimensional problem}
 We recall well-known facts about the solution of the Stokes problem above (see e.g. \cite{girault2012finite}).
\begin{thm}
\label{thm:StokesExistenceUniqueness}
Let $\Om$ be a bounded and connected open subset of of $\mathbb{R}^{n_d}$ with Lipschitz boundary $\G$. Assume $\bs{f} \in L^2(\Om)^{n_d}$, $\bs{u}_D \in H^{1/2}(\G_D)^{n_d}$ such that $\int_{\G} \bs{u}_D \cdot \bs{n} = 0$ if $\Gamma_N=\emptyset$, and $\bs{t}_N \in L^2(\Gamma_N)^{n_d}$. Then, Problem \eqref{eq:SteadyStokes} admits a unique solution $[\bs{u}\, , \, p] \in H^1(\Om)^{n_d} \times L^2(\Om)$ satisfying $\int_\Omega p =0$ if  $\Gamma_N=\emptyset$.
In addition, if $\Gamma$ is of class $\mathcal{C}^{2}$ and $\Gamma_N=\emptyset$, and if $\bs{u}_D \in H^{3/2}(\G)^{n_d}$, then $[\bs{u}\, , \, p] \in H^{2}(\Om)^{n_d} \times H^{1}(\Om)$  with the bound 
\begin{align}
\| \, \bs{u} \, \|_{2,\Om} + \| \, p \, \|_{1,\Om} \leq C \, \left( \| \, \bs{f} \, \|_{0,\Om} + \| \, \bs{u}_D \,  \|_{3/2,\G} \right) \; .
\end{align}
\end{thm}

\subsection{Weak discrete formulation}\label{sec:discrSBM}
Before stating the weak discrete formulation, we will make the following
\begin{assumption}
	\label{ass:NeumannAssumption}
	The Neumann boundary is body-fitted, that is $\tGN=\G_N$ (see Section~\ref{sec:sbmDef}).
\end{assumption}
\begin{rem}
	Assumption~\ref{ass:NeumannAssumption}, which at first look may seem restrictive, is actually most frequently verified in applications involving the Stokes and Navier-Stokes equations. In this context, Neumann conditions are simply inflow and outflow conditions, and are typically applied on a portion of the boundary that has been meshed using a body-fitted grid. Note also that Assumption~\ref{ass:NeumannAssumption} can be relaxed in practical computations, as shown for example in~\cite{main2018shifted}, where numerical results with embedded inflows/outflows appear to be stable and accurate.
\end{rem}
We introduce next the discrete spaces $\bs{V}_h(\tO)$ and $Q_h(\tO)$, for the velocity and the pressure, respectively. We assume that a stable and convergent base formulation for the Stokes flow exist for these spaces in the case of body-fitted grids. For example, if we consider the piecewise linear spaces
\begin{subequations}
\begin{align}
\bs{V}_h(\tO) & = \; \left\{ \bs{v}_h \in C^0(\tO)^{n_d} \ | \  {\bs{v}_h}_{|T} \in  \mathcal{P}^1(T)^{n_d} \, , \, \forall T \in \ti{\mathcal{T}}_h \right\}  \, , \\
Q_h(\tO) & = \; \left\{ q_h \in C^0(\tO)  \ | \ {q_h}_{|T} \in \mathcal{P}^1(T)  \, , \, \forall T \in \ti{\mathcal{T}}_h \right\}  \, ,
	\end{align}
\end{subequations}
the stabilized formulation of Hughes {\it et al.}~\cite{Hughes198785} will satisfy these assumptions. For the sake of simplicity, we will use this formulation in what follows, but alternative choices are possible, such as, for example, discontinuous Galerkin spaces. 
In the case of pure Dirichlet conditions, that is $\G_N=\emptyset$, the space $Q_h(\tO) $ needs to be modified as
\begin{align}
Q_{h}(\tO) & = \; \left\{ q_h \in {Q_h}(\tO)  \ |   \int_{\tO} q_h = 0  \right\}  \, .
\end{align}
It is also convenient to introduce the product space ${\bs{W}_h}(\tO) = {\bs{V}_h}(\tO) \times {Q_h}(\tO)$.\\ 

\noindent Discretizing Problem~\eqref{eq:SteadyStokes} in $\tO$, enforcing~\eqref{eq:Finalbsu-g} on $\tGD$ (see definition in Section~\ref{sec:sbmDef}) with $\bar{\bs{g}} = \bar{\bs{u}}_D$, applying~\eqref{eq:tractionB} on $\G_{N}$ (see Assumption~\ref{ass:NeumannAssumption}) and adopting an unsymmetric form of the velocity strain and pressure gradient terms, we deduce the following SBM weak form of \eqref{eq:SteadyStokes}: 
\begin{quote}
Find $[\bs{u}_h\, , \, p_h] \in \bs{W}_h(\tO)$ such that, $\forall [\bs{w}_h\, , \, q_h] \in \bs{W}_h(\tO)$,
\begin{align}
\label{eq:ShiftedNitscheVariationalFormSteadyStokes}
0=&\;
(  2 \mu \, \bs{\varepsilon}(\bs{u}_h) \, , \,  \bs{\varepsilon}( \bs{w}_h) )_{\tO} 
-(  p_h \, , \, \nabla \cdot \bs{w}_h  )_{\tO} 
+(  \nabla \cdot \bs{u}_h \, , \, q_h )_{\tO} 
-(  \bs{f} \, , \, \bs{w}_h )_{\tO} 
\nonumber \\
&\;
-\avg{  2 \mu \, \bs{\varepsilon}(\bs{u}_h) - p_h \bs{I}  \, , \,  \bs{w}_h \otimes \ti{\bs{n}} }_{\tGD}
- \avg{ ( \bs{S}_{h} \bs{u}_h - \bar{\bs{u}}_D ) \otimes \ti{\bs{n}}  \, , \, 2 \mu \, \bs{\varepsilon}(\bs{w}_h) + q_h \bs{I} }_{\tGD}
\nonumber \\
& \;
+\alpha \, \avg{ 2 \mu \, h_{\perp}^{-1}   \, ( \bs{S}_{h} \bs{u}_h - \bar{\bs{u}}_D ) \, , \, \bs{S}_{h} \bs{w}_h  }_{\tGD} 
- \avg{  \bs{t}_N \, , \, \bs{w}_h }_{\G_N}
\nonumber \\
& \;
+ \gamma \sum_{T \in \ti{\cal T}_h} \left(  h_{\tau}^{2} \, (2\mu)^{-1} \, \left(- \nabla \cdot (2 \mu \, \bs{\varepsilon}(\bs{u}_h)) + \nabla p_h - \bs{f} \right)  \, , \, \nabla q_h  \right)_{T} 
\; ,
\end{align}
\end{quote}
where the Nitsche's stabilization parameter $\alpha  > 0$ helps in the weak imposition of the Dirichlet boundary condition,
whereas the parameter $\gamma>0$ scales a pressure stabilization term required by equal-order velocity/pressure pairs~\cite{Hughes198785}.
As for the weak Poisson problem discussed in Section~\ref{sec:weak_discr_poisson}, we assume there exist constants $C_{r}$, $\xi_{1} $, $\xi_{2} \in \mathbb{R}^{+}$ such that $(1/\sqrt{C_r}) \, h \leq h_\tau \leq h$ and $\xi_{1} \, h \leq h_{\perp} \leq \xi_{2} \, h$. Again, with slight abuse of notation, we will assume $h$, $h_{\tau}$ and $h_{\perp}$ are interchangeable. 
 \begin{rem}
 The variational statement \eqref{eq:ShiftedNitscheVariationalFormSteadyStokes} does not include a stabilization term involving the tangential derivative of the Dirichlet boundary conditions nor a stabilization term on the incompressibility constraint, as was the case in~\cite{main2018shifted0,main2018shifted,atallah2020analysis}. The main reason for their introduction in addition to assuming $\inf_{\tG} \ti{\bs{n}} \cdot \bs{\nu}  > 0$  was to attain coercivity of the bilinear form $a(\bs{u}_h\, , \, \bs{w}_h)$. However, as it will be clearer from what follows, asymptotic coercivity can be proved by simply relying on Assumption~\ref{ass:d_asym}.
 \end{rem}
 \begin{rem}
 The proposed algorithm can be shown to satisfy statements of global conservation of mass and momentum. We refer the reader to~\cite{atallah2020analysis} for more details.
 \end{rem}
\noindent The variational statement~\eqref{eq:ShiftedNitscheVariationalFormSteadyStokes} can be succinctly expressed as: 
 \begin{quote}
 Find $[\bs{u}_h\, , \, p_h] \in \bs{W}_h(\tO)$ such that, $\forall [\bs{w}_h\, , \, q_h] \in \bs{W}_h(\tO)$,
\begin{subequations} 
	\label{eq:SBMequations}
\begin{align}
\mathcal{B}([\bs{u}_h\, ,\,p_h];[\bs{w}_h\, , \, q_h] )
&=\;
\mathcal{L}([\bs{w}_h\, , \, q_h] )  \; , 
\end{align}
where
\begin{align}
\mathcal{B}([\bs{u}_h \, , \, p_h];[\bs{w}_h\, , \, q_h] )
& = \;
a(\bs{u}_h\, , \, \bs{w}_h) 
+ b(p_h\, , \, \bs{w}_h) 
- b(q_h\, , \, \bs{u}_h) 
- \bar{b}( \bs{u}_h \, , \, q_h) 
+ c(p_h\, , \, q_h) \; , 
\label{eq:Bform}
\\[.1cm]
\mathcal{L}([\bs{w}_h\, , \, q_h] )
& = \;
l_f(\bs{w}_h) 
+ l_g(q_h) \; .
\end{align}
with
\begin{align}
\label{eq:aBilinearStokes}
a(\bs{u}_h \, , \, \bs{w}_h) 
&=\;
(  2 \mu \, \bs{\varepsilon}(\bs{u}_h) \, , \,  \bs{\varepsilon}( \bs{w}_h) )_{\tO} 
- \avg{  2 \mu \, \bs{\varepsilon}(\bs{u}_h) \, , \,  \bs{w}_h \otimes \ti{\bs{n}}   }_{\tGD}
\nonumber \\
&\phantom{0=}\;
- \avg{ \bs{S}_{h} \bs{u}_h \otimes \ti{\bs{n}}  \, , \, 2 \mu \, \bs{\varepsilon}(\bs{w}_h) }_{\tGD}
+\alpha \, \avg{ 2 \mu \, h^{-1} \, \bs{S}_{h} \bs{u}_h \, , \,   \bs{S}_{h} \bs{w}_h }_{\tGD} \; ,
\\[.1cm]
\label{eq:bBilinearStokes}
b(p_h\, , \, \bs{w}_h)  
&=\;
-(  p_h \, , \,  \nabla \cdot \bs{w}_h  )_{\tO} 
+\avg{  p_h  \, , \, \bs{w}_h \cdot \ti{\bs{n}}   }_{\tGD} \; ,
\\[.1cm]
\label{eq:bBarBilinearStokes}
\bar{b}( \bs{u}_h  \, , \, q_h ) 
& =  \;
\avg{  ( \nabla \bs{u}_h \, \bs{d})  \cdot  \ti{\bs{n}}  \, , \, q_h  }_{\tGD}
+ \gamma \sum_{T \in \ti{\cal T}_h} \left(  h^{2} \, (2\mu)^{-1} \, \nabla \cdot (2 \mu \, \bs{\varepsilon}(\bs{u}_h))   \, , \, \nabla q_h \right)_{T} \; ,
\\[.1cm]
c(p_h \, , \, q_h) 
&=\; 
\gamma \,  (  h^{2} \, (2\mu)^{-1} \,  \nabla p_h \, , \,   \nabla q_h  )_{\tO}  \; ,
\\[.1cm]
l_f(\bs{w}_h)
&=\;
( \bs{f} \, , \, \bs{w}_h  )_{\tO} 
- \avg{ \bar{\bs{u}}_D \otimes \ti{\bs{n}}  \, , \, 2 \mu \, \bs{\varepsilon}(\bs{w}_h)   }_{\tGD}
+\alpha \, \avg{ 2 \mu \, h^{-1} \, \bar{\bs{u}}_D  \, , \,  \bs{S}_{h} \bs{w}_h }_{\tGD} 
\nonumber \\
&\phantom{=} \;
+ \avg{ \bs{t}_N \, , \, \bs{w}_h  }_{\G_N} \; ,
\\[.1cm]
l_g(q_h)
&=\;
- \avg{  \bar{\bs{u}}_D  \cdot \ti{\bs{n}}  \, , \, q_h }_{\tGD}
+ \gamma \,  (  h^2 \, (2\mu)^{-1} \, \bs{f} \, , \,  \nabla q_h )_{\tO}  \; .
\end{align}
\end{subequations}
\end{quote}
\subsection{Well-posedness and stability}\label{sec:ShiftedNitscheStablityAnalysis}

The first step in our analysis is to prove that the bilinear form $a(\cdot,\cdot)$ is coercive, under suitable assumptions. 
In a second step, we will establish that the bilinear form $\mathcal{B}(\cdot,\cdot)$ satisfies a uniform inf-sup condition. 
This will immediately imply the existence and uniqueness of the solution of the discrete SBM problem, and will be lately used to prove its convergence to the exact solution, with optimal error estimates in an appropriate natural norm. 
We start by proving an intermediate technical result.
\begin{lemma}\label{lemma:SBM-korn-norm} Let $\bar{C}_K $ be the constant in the Korn inequality \eqref{eq:korn2}. Then, $\forall \bs{u}_h \in \bs{V}_h(\tO)$, 
\begin{align}
\label{eq:SBMKorn-coercivity}
l(\tO)^{-2} \, \| \, \bs{u}_h \, \|^{2}_{0,\tO} 
+ \| \, \nabla \bs{u}_h \, \|^{2}_{0,\tO} 
&\leq \; 
{2} \, \bar{C}_K^2  
\left(  
\| \, h^{-1/2} \, \bs{S}_{h} \bs{u}_h \, \|^2_{0,\tGD} 
+ ( c_{d}\, C_{I} \, \hat{h}_{\tGD}^{\zeta}  )^{2} \, \| \, \nabla \bs{u}_h  \,  \|^2_{0,\tO}  
\right.
\nonumber \\
&\phantom{\leq} \; 
\left.
\qquad \quad + 1/2 \, \| \, \bs{\varepsilon}(\bs{u}_h) \, \|^{2}_{0,\tO} 
\right)
\; .
\end{align}
\end{lemma}
\proof 
Korn's inequality~\eqref{eq:korn2} yields
\begin{align}
\label{SBM_korn_p2}
l(\tO)^{-2} \, \| \, \bs{u}_h \, \|^{2}_{0,\tO}  
+ \|  \, \nabla \bs{u}_h \, \|^{2}_{0,\tO} 
\leq 
\bar{C}_K^2 
\left( 
\| \, h^{-1/2} \, \bs{u}_h \, \|^{2}_{0,\tGD}  
+ \| \, \bs{\varepsilon}(\bs{u}_h) \, \|^{2}_{0,\tO} 
\right)
\; .
\end{align}
Using the triangle inequality, Assumption~\ref{ass:d_asym} and Theorem~\ref{theo:discretetrace}, 
\begin{align}
\| \, h^{-1/2} \, \bs{u}_h \, \|_{0,\tGD} 
\leq&\;
 \| \, h^{-1/2} \,\bs{S}_{h} \bs{u}_h \, \|_{0,\tGD} 
 + \| \, h^{-1/2} \, \nabla \bs{u}_h \, \bs{d} \, \|_{0,\tGD} 
\nonumber \\
\leq&\;
\| \, h^{-1/2} \, \bs{S}_{h} \bs{u}_h \, \|_{0,\tGD} 
+ c_d \, C_{I} \, \hat{h}_{\tGD}^{\zeta} \, \|  \, \nabla \bs{u}_h \,  \|_{0,\tO} 
\; .
\end{align}
Thus, 
\begin{align}
\label{eq:kornproofFinal}
\| \, h^{-1/2} \, \bs{u}_h \, \|^{2}_{0,\tGD} \leq  2 \left( \| \, h^{-1/2} \, \bs{S}_{h} \bs{u}_h \, \|^{2}_{0,\tGD}  + ( c_{d}\, C_{I} \, \hat{h}_{\tGD}^{\zeta}  )^{2} \, \| \, \nabla \bs{u}_h  \, \|^{2}_{0,\tO}  \right) \; .
\end{align}
Substituting~\eqref{eq:kornproofFinal} into~\eqref{SBM_korn_p2} completes the proof.
\endproof
\begin{thm}[Coercivity]\label{thm:coerc-a}
Consider the bilinear form $a(\cdot, \cdot)$ defined in~\eqref{eq:aBilinearStokes}. 
If the parameter $\alpha$  is sufficiently large and the quantity $c_d \,  \hat{h}_{\tGD}^{\zeta}$ is sufficiently small, there exists a constant $C_a>0$ independent of the mesh size, such that
\begin{equation} \label{eq:40}
a(\bs{u}_h\, , \, \bs{u}_h)  
\geq 
C_{a} \, \| \, \bs{u}_h \, \|^2_{a} 
\, \qquad 
\forall \bs{u}_h \in \bs{V}_h(\tO),
\end{equation}
where 
\begin{align} \label{eq:def-anorm}
 \| \, \bs{u}_h \, \|_{a}^2
&=\;
l(\tO)^{-2} \, \| \, (2 \mu)^{1/2} \, \bs{u}_h \, \|^{2}_{0,\tO} 
+ \| \, (2 \mu)^{1/2} \, \nabla \bs{u}_h \|^{2}_{0,\tO}   
+ \| \, (2 \mu \, h^{-1})^{1/2} \, \bs{S}_{h} \bs{u}_h \,  \|^2_{0,\tGD}
\nonumber \\
&\phantom{s=}\;
+ \| \, (2 \mu \, \|\, \bs{d} \, \|)^{1/2} \, \nabla \bs{u}_h \, \|^{2}_{0,\tGD}
 \; .
\end{align}
\end{thm}

\begin{proof}
Substituting $\bs{u}_h$ for $\bs{w}_h$ in~\eqref{eq:aBilinearStokes} yields
\begin{align}
\label{eq:aCoercS1}
 a(\bs{u}_h\, , \, \bs{u}_h)  
&=\;
\| \, (2 \mu)^{1/2} \, \bs{\varepsilon}(\bs{u}_h) \, \|^2_{0,\tO}  
+\alpha \, \| \, (2 \mu \, h^{-1})^{1/2} \, \bs{S}_{h} \bs{u}_h \, \|^2_{0,\tGD} 
-  2 \, \avg{ 2 \mu \, \bs{\varepsilon}(\bs{u}_h) \, , \, \bs{S}_{h} \bs{u}_h \otimes \ti{\bs{n}} }_{\tGD}
\nonumber \\
&\phantom{0=}\;
+\avg{ \nabla \bs{u}_h \, (\bs{d} \otimes \ti{\bs{n}}) \, , \, 2 \mu \, \bs{\varepsilon}(\bs{u}_h) }_{\tGD}
\; , 
\end{align}
where the term  $\avg{ \nabla \bs{u}_h \, (\bs{d} \otimes \ti{\bs{n}}) \, , \, 2 \mu \, \bs{\varepsilon}(\bs{u}_h) }_{\tGD}$ has been added and subtracted.
Using Young's $\epsilon$-inequality and the third discrete trace inequality \eqref{eq:TraceInequality1} yields
\begin{subequations}
\begin{align}
\label{eq:aCoercS2}
| 2 \, \avg{ 2 \mu \, \bs{\varepsilon}(\bs{u}_h) \, , \, \bs{S}_{h} \bs{u}_h \otimes \ti{\bs{n}} }_{\tGD} |
&\leq\;
\epsilon_{1} \, \| \, (2 \mu \, h)^{1/2} \, \bs{\varepsilon}(\bs{u}_h) \ti{\bs{n}} \, \|^2_{0,\tGD} 
+\epsilon_{1}^{-1} \,  \| \, (2 \mu \, h^{-1})^{1/2} \, \bs{S}_{h} \bs{u}_h \, \|^2_{0,\tGD}
\nonumber \\
&\leq\;
\epsilon_{1} \, C_I \,
\| \, (2 \mu)^{1/2} \, \bs{\varepsilon}(\bs{u}_h) \, \|^2_{0,\tO} 
+\epsilon_{1}^{-1} \,  \| \, (2 \mu \, h^{-1})^{1/2} \, \bs{S}_{h} \bs{u}_h \, \|^2_{0,\tGD} \; ,
\end{align}
\begin{align}
\label{eq:aCoercS3}
|\avg{ \nabla \bs{u}_h \, (\bs{d} \otimes \ti{\bs{n}}) \, , \, 2 \mu \, \bs{\varepsilon}(\bs{u}_h) }_{\tGD} |
&\leq \;
1/2 \, c_d \, \hat{h}_{\tGD}^{\zeta} \, 
\left(
\| \, (2 \mu \, h)^{1/2} \, \nabla \bs{u}_h \bs{\nu} \, \|^{2}_{0,\tGD}
+ \| \, (2 \mu \, h)^{1/2} \, \bs{\varepsilon}(\bs{u}_h)\ti{\bs{n}} \, \|^{2}_{0,\tGD}
\right) 
\nonumber \\
&\leq\;
1/2 \, c_d \, C_I  \, \hat{h}_{\tGD}^{\zeta}  
\left( \| \, (2 \mu)^{1/2} \, \nabla \bs{u}_h \, \|^2 _{0,\tO} +
\| \, (2 \mu)^{1/2} \, \bs{\varepsilon}(\bs{u}_h) \, \|^2_{0,\tO} \right) \,.
\end{align}
\end{subequations}
Substituting~\eqref{eq:aCoercS2} and~\eqref{eq:aCoercS3}  into~\eqref{eq:aCoercS1}, we obtain
\begin{align}
  a(\bs{u}_h\, , \, \bs{u}_h) 
&\geq\;
\left(1- \epsilon_{1} \, C_I - 1/2\, c_d \, C_I  \, \hat{h}_{\tGD}^{\zeta}  \right) \, \| \, (2 \mu)^{1/2} \, \bs{\varepsilon}(\bs{u}_h) \, \|^2_{0,\tO}  
+\left( \alpha - \epsilon_{1}^{-1} \right) \,  \| \, (2 \mu \, h^{-1})^{1/2} \, \bs{S}_{h} \bs{u}_h \, \|^2_{0,\tGD}
\nonumber \\
&\phantom{=}\;
-1/2 \, c_d \, \hat{h}_{\tGD}^{\zeta} \,  C_I \,  \| \, (2 \mu)^{1/2} \, \nabla \bs{u}_h \, \|^2 _{0,\tO} \;.
\end{align}
If we choose $\epsilon_1= (4 \, C_I)^{-1}$, we have that, for sufficiently refined grids, $ c_d \, C_I \, \hat{h}_{\tGD}^{\zeta} \leq 1/2$ and
\begin{align}
\label{eq:aCoercS4}
 a(\bs{u}_h\, , \, \bs{u}_h)  
&\geq\;
1/2 \, \| \, (2 \mu)^{1/2} \, \bs{\varepsilon}(\bs{u}_h) \, \|^2_{0,\tO}
+\left( \alpha - 4C_I \right) \, \| \, (2 \mu \, h^{-1})^{1/2} \, \bs{S}_{h} \bs{u}_h \, \|^2_{0,\tGD}
\nonumber \\
\phantom{=}&\;
\quad 
-1/2 \, c_d \, C_I  \, \hat{h}_{\tGD}^{\zeta} \,  \| \, (2 \mu)^{1/2} \, \nabla \bs{u}_h \, \|^2 _{0,\tO} \;.
\end{align}
Replacing the first term in~\eqref{eq:aCoercS4} with the result \eqref{eq:SBMKorn-coercivity} of Lemma~\ref{lemma:SBM-korn-norm},
\begin{align}
 a(\bs{u}_h \, , \, \bs{u}_h)
&\geq\;
\left( ({2} \, \bar{C}_K^2)^{-1} - ( c_{d}\, C_{I} \, \hat{h}_{\tGD}^{\zeta} )^{2} - 1/2 c_d \, C_I \, \hat{h}_{\tGD}^{\zeta}  \right) 
\| \, (2 \mu)^{1/2} \, \nabla \bs{u}_h \, \|^{2}_{0,\tO}
\nonumber \\
&\phantom{=}\;
+  ({2} \, \bar{C}_K^2 )^{-1} \, l(\tO)^{-2} \| \, (2 \mu)^{1/2} \, \bs{u}_h \, \|^{2}_{0,\tO}
+\left( \alpha - 4C_I - 1\right) \, \| \, (2 \mu \, h^{-1})^{1/2} \, \bs{S}_{h} \bs{u}_h \, \|^2_{0,\tGD} \; .
\end{align}
Now, for a sufficiently fine grid, we can assume $ \left( ({2} \, \bar{C}_K^2)^{-1} - ( c_{d}\, C_{I} \, \hat{h}_{\tGD}^{\zeta}  )^{2} - 1/2 \, c_d \, C_I  \, \hat{h}_{\tGD}^{\zeta}  \, \right) \geq ({4} \, \bar{C}_K^2)^{-1} $ so that
\begin{align}
a(\bs{u}_h \, , \, \bs{u}_h)
&\geq\;
({4} \, \bar{C}_K^2)^{-1} \, 
\| \, (2 \mu)^{1/2} \, \nabla \bs{u}_h \, \|^{2}_{0,\tO}
+  ({2} \, \bar{C}_K^2 )^{-1} \, l(\tO)^{-2} \, \| \, (2 \mu)^{1/2} \, \bs{u}_h \, \|^{2}_{0,\tO}
\nonumber \\
&\phantom{=}\;
+\left( \alpha - 4C_I - 1\right) \, \| \, (2 \mu \, h^{-1})^{1/2} \, \bs{S}_{h} \bs{u}_h \, \|^2_{0,\tGD} \; .
\end{align}
Finally, noting that $\| \, (2 \mu)^{1/2} \, \nabla \bs{u}_h \, \|^{2}_{0,\tO} \geq C^{-1}_{I} \, \| \, (2 \mu \, \|\, \bs{d} \, \|)^{1/2} \, \nabla \bs{u}_h \, \|^{2}_{0,\tGD} $ and choosing  $\alpha > 4\,C_I +1$ we obtain the desired coercivity bound with $ C_a =  \min \left( \alpha - 4\,C_I -1\, , \, ({8}\, \bar{C}_K^2)^{-1}, ({8} \,C_{I} \, \bar{C}_K^2)^{-1} \right) \, .$ 
\end{proof}
Apart from the fact that we are setting $\bs{d} = \| \, \bs{d} \, \| \bs{\nu}$ instead of $\bs{d} = \| \, \bs{d} \, \| \bs{n}$, this new result represents the main difference from~\cite{atallah2020analysis}. This indicates that the assumption $\inf_{\tGD} \ti{\bs{n}} \cdot \bs{\nu}  > 0$ in addition to the tangential and incompressibility constraint stabilization terms introduced in the earlier versions of the SBM method~\cite{main2018shifted0,main2018shifted,atallah2020analysis} is just a set of sufficient (yet not necessary) conditions to obtain numerical stability.
\begin{thm}\label{thm:a-norm}
The quantity $\| \, \bs{u}_h \, \|_{a}$ defined in \eqref{eq:def-anorm} is a norm on $\bs{V}_h(\tO)$, equivalent to the norm $\| \, \bs{u}_h \, \|_{H^1(\tO)}$ (although not uniformly with respect to the mesh size).
\end{thm}
\proof 
To prove equivalence between $\| \, \bs{u}_h \, \|_{1,\tO} $ and $\| \,  \bs{u}_h \, \|_{a}$ means to show that $C_{1} \, \| \, \bs{u}_h \, \|_{1,\tO}  \leq \| \, \bs{u}_h \, \|_{a} \leq C_{2}(h) \, \| \, \bs{u}_h \, \|_{1,\tO} $ for some scalars $C_{1}, C_{2} > 0$.
%
We have
\begin{align}
 \| \, \bs{u}_h \, \|_{1,\tO} 
 &=  
 \| \, \bs{u}_h \, \|_{0,\tO} + l(\tO)\,\| \, \nabla \bs{u}_h \, \|_{0,\tO}
\nonumber \\ 
& =
l(\tO) \, (2\mu)^{-1/2} \, \left( l(\tO)^{-1} \, \| \, (2\mu)^{1/2} \, \bs{u}_h \, \|_{0,\tO} + \| \, (2\mu)^{1/2} \, \nabla \bs{u}_h \, \|_{0,\tO} \right) 
\nonumber \\ 
& \leq
l(\tO) \, (2\mu)^{-1/2} \, \| \, \bs{u}_h \, \|_{a}
 \end{align}
 which gives the first inequality  with $C_{1} = l(\tO)^{-1}\, (2\mu)^{1/2}$. To show the second inequality, we use the discrete trace inequalities of Theorem~\ref{theo:discretetrace} and Theorem~\ref{thm:TraceTheorem} to get
 \begin{align}
\| \, \bs{u}_h \, \|_{a} 
 & = 
 l(\tO)^{-1} \, \| \, (2 \mu)^{1/2} \, \bs{u}_h \, \|_{0,\tO} 
 + \| \, (2 \mu)^{1/2} \, \nabla \bs{u}_h \|_{0,\tO}   
 + \| \, (2 \mu \, h^{-1})^{1/2} \, \bs{S}_{h} \bs{u}_h \,  \|_{0,\tGD}
 \nonumber \\
 & \phantom{\leq}
 + \| \, (2 \mu \, \|\, \bs{d} \, \|)^{1/2} \, \nabla \bs{u}_h \, \|_{0,\tGD}
 \nonumber \\
 & \leq
  l(\tO)^{-1} \, \| \, (2 \mu)^{1/2} \, \bs{u}_h \, \|_{0,\tO} 
 + (1+2 c_{d}^{1/2} \, C_{I} \, \hat{h}^{\zeta/2} ) \, \| \, (2 \mu)^{1/2} \, \nabla \bs{u}_h \|_{0,\tO}   
 + \| \, (2 \mu \, h^{-1})^{1/2} \, \bs{u}_h \,  \|_{0,\tGD}
 \nonumber \\
 & \leq
  (l(\tO)^{-1}+C\,l(\tO)^{-1/2} \,h^{-1/2} ) \, \| \, (2 \mu)^{1/2} \, \bs{u}_h \, \|_{0,\tO} 
 \nonumber \\
 & \phantom{\leq}
  + (1+2 c_{d}^{1/2} \, C_{I} \, \hat{h}^{\zeta/2} +C\,l(\tO)^{1/2} \,h^{-1/2}) \, \| \, (2 \mu)^{1/2} \, \nabla \bs{u}_h \|_{0,\tO}   
  \nonumber \\
 & \leq
 3 \, l(\tO)^{-1/2} \,h^{-1/2} \,(2 \mu)^{1/2}\, \| \, \bs{u}_h \, \|_{1,\tO} \,,
 \end{align}
 which concludes the proof with $C_{2}(h) =  3 \, l(\tO)^{-1/2} \,h^{-1/2} \,(2 \mu)^{1/2}$.
 \endproof

The coercivity property of the form $a$ allows us to prove a uniform inf-sup condition for the form $\mathcal{B}$, thus yielding the LBB-stability of the proposed SBM variational formulation of the Stokes problem.
 We skip the proof, as it does not contain significant differences with respect to that of Theorem 3  in~\cite{atallah2020analysis}, where the interested reader can find detailed derivations.
\begin{thm}[LBB, inf-sup condition]
\label{thm:StokesInfSup}
If the parameter $\alpha$ is sufficiently large and the quantity $c_d \, \hat{h}_{\tGD}^{\zeta}$ is sufficiently small, there exists a constant $\alpha_{LBB} >0$, independent of the mesh size, such that for any pair $[ \bs{u}_h\, , \, p_h] \in \bs{W}_h(\tO)$ one can find a pair $[ \bs{w}_h\, , \, q_h] \in \bs{W}_h(\tO)$ satisfying
\begin{align}\label{eq:StokesInfSup}
\mathcal{B}([\bs{u}_h\, , \, p_h]; [ \bs{w}_h \, , \, q_h] ) 
\geq 
\alpha_{LBB} \, \| \, [ \bs{u}_h \, , \, p_h] \, \|_{\cal B} \, \| \, [ \bs{w}_h \, , \, q_h] \, \|_{\cal B} \; ,
\end{align}
where
\begin{align}
\| \, [ \bs{u}_h\, , \, p_h] \, \|_{\cal B}^2  =  \| \,\bs{u}_h \, \|_{a}^2  +  \|\,  (2 \mu)^{-1/2} \, p_h \,  \|_{0,\tO}^2  + \| \,  (2 \mu)^{-1/2} \, h \, \nabla p_h \, \|^{2}_{0,\tO}
\; .
\end{align}
\end{thm}
Next, we focus on the convergence of the SBM discretization in a natural norm. Precisely, we set 
$$
\bs{W}(\tO;\tT)= \bs{V}(\tO;\tT) \times Q(\tO;\tT)\;,
$$ 
with
	\begin{subequations}
	\begin{align}
	\bs{V}(\tO;\tT) & =  \bs{V}_h(\tO) + H^2(\tO)^{n_d}\; , \\
	Q(\tO;\tT)   & = \;  Q_h(\tO) + Q(\tO)\;  ,
	\end{align}
where $Q(\tO) = H^1(\tO)$ if $\G_{N} \neq \emptyset$ or $Q(\tO) = H^1(\tO) \cap L^2_0(\tO)$ if $\G_{N} = \emptyset$ . We equip $\bs{W}(\tO;\tT)$ with the norm
	\begin{align}
	\| \, [ \bs{v} \, , \, q ] \, \|^2_{\bs{W}(\tO;\tT)} & = \| \, [ \bs{v}\, , \, q] \, \|^2_{\mathcal{B}} +  
	\| \, (2 \mu)^{1/2} \,  h \, \nabla  \bs{\varepsilon}(\bs{v}) \, \|^2_{0,\tO;\ti{\mathcal{T}}_h}  
	\; .
	\end{align}
	\end{subequations}
Note that if $[ \bs{v}_h \, , \, q_h ]  \in \bs{W}_h(\tO)$, then $\| \, [ \bs{v}_h \, , \, q_h ] \, \|_{\bs{W}(\tO;\tT)} = \| \, [ \bs{v}_h \, , \, q_h ] \, \|_{\mathcal{B}}$. Also note that
$\bs{V}(\tO;\tT) \subset H^2(\tO;\ti{\mathcal{T}}_h)^{n_d}$, where the latter space is the subset of $H^1(\tO)^{n_d}$ of the functions with broken $H^2$-regularity on the triangulation $\ti{\mathcal{T}}_h$.\\
The analysis of the consistency error uses the following identity.
\begin{lemma}[Consistency error]
	\label{lem:SBMorthogonality}
Let the exact solution of the Stokes problem \eqref{eq:SteadyStokes} satisfy $[ \bs{u}\,, \, p] \in H^2(\Om)^{n_d} \times H^1(\Om)$, with $p$ chosen to satisfy $p_{|\tO}\in L^2_0(\tO)$  if $\G_N = \emptyset$.
For any $[\bs{v}_h\, , \, \omega_h] \in \bs{W}_h(\tO) $, it holds that
\begin{align}
\mathcal{B}( [ \bs{u} - \bs{u}_h \, , \,  p - p_h ]; [\bs{v}_h\, , \,  \omega_h] ) 
&=
\avg{ \bs{R}_{h} \bs{u} \otimes \ti{\bs{n}}  \, , \, 2 \mu \, \bs{\varepsilon}(\bs{v}_h) + \omega_h \bs{I} }_{\tGD}
- \alpha \, \avg{ 2 \mu \, h^{-1}   \, \bs{R}_{h} \bs{u} \, , \, \bs{S}_{h} \bs{v}_h \,   }_{\tGD} \; .
\end{align}
\end{lemma}
\begin{proof}
From \eqref{eq:SBMequations}, we get
\begin{align}
\mathcal{B}( [ \bs{u} - \bs{u}_h \, , \,  p - p_h ]; [\bs{v}_h\, , \,  \omega_h] ) 
& = \;
\mathcal{B}( [ \bs{u}  \, , \,  p  ]; [\bs{v}_h\, , \,  \omega_h] ) 
-\mathcal{B}( [  \bs{u}_h \, , \,  p_h ]; [\bs{v}_h\, , \,  \omega_h] ) 
\nonumber \\ 
& = \;
\mathcal{B}( [ \bs{u}  \, , \,  p  ]; [\bs{v}_h\, , \,  \omega_h] ) 
- \mathcal{L}( [\bs{v}_h\, , \, \omega_h] ) \; .
\end{align}
Integrating by parts $(  2 \mu \, \bs{\varepsilon}(\bs{u}) \, , \,  \bs{\varepsilon}( \bs{v}_h) )_{\tO}$ and $(  p \, , \, \nabla \cdot \bs{v}_h  )_{\tO}$ in $\mathcal{B}( [ \bs{u} \, , \,  p ]; [\bs{v}_h\, , \,  \bs{\omega}_h] )$ and recalling~\eqref{eq:bsu-g}, we obtain
\begin{align}
\mathcal{B}( [ \bs{u} - \bs{u}_h \, , \,  p - p_h ]; [\bs{v}_h\, , \,  \omega_h] ) 
& =  \;
- \avg{ ( \bs{S}_{h} \bs{u} - \bar{\bs{u}}_D) \otimes \ti{\bs{n}}  \, , \, 2 \mu \, \bs{\varepsilon}(\bs{v}_h) + \omega_h \bs{I} }_{\tGD}
\nonumber \\
& \phantom{=} \; \; \;
+\alpha \, \avg{ 2 \mu \, h^{-1}   \, (\bs{S}_{h} \bs{u} - \bar{\bs{u}}_D) \, , \, \bs{S}_{h} \bs{v}_h \,   }_{\tGD} 
\nonumber \\
& = \; 
 \avg{ \bs{R}_{h} \bs{u} \otimes \ti{\bs{n}}  \, , \, 2 \mu \, \bs{\varepsilon}(\bs{v}_h) + \omega_h \bs{I} }_{\tGD}
- \alpha \, \avg{ 2 \mu \, h^{-1}   \, \bs{R}_{h} \bs{u} \, , \, \bs{S}_{h} \bs{v}_h \,   }_{\tGD}  \; ,
\end{align}
which concludes the proof.
\end{proof}

Convergence in the norm ${\bs{W}(\tO;\tT)}$ is established in the following theorem.

\begin{thm}[Convergence in the natural norm]
	\label{thm:StokesConvergenceNatural}
	Suppose that  $\G_D$ is of class $\mathcal{C}^{2}$, 
	and that the exact solution of the Stokes problem \eqref{eq:SteadyStokes} satisfies $[ \bs{u}\,, \, p] \in (H^2(\Om))^{n_d} \times H^1(\Om)$; in addition, if $\Gamma_N=\emptyset$, choose $p$ satisfying $p_{|\tO}\in L^2_0(\tO)$.
	Suppose also that Assumption~\ref{ass:d_asym}, Assumption~\ref{ass:NeumannAssumption} and the hypotheses of Theorem~\ref{thm:StokesInfSup} hold. 
	Then, the SBM numerical solution $[\bs{u}_h\, , \, p_h]$ of \eqref{eq:SBMequations} satisfies the following error estimate:
	\begin{align}\label{eq:convergenceNN}
	\| \,[\bs{u} - \bs{u}_h \, , \, p -p_h ]  \, \|_{\bs{W}(\tO;\tT)}  
	&\leq\;  
	C\, h_{\tO} \left( \| \, \nabla (\nabla \bs{u}) \, \|_{0, \Omega} + \| \, \nabla p \, \|_{0,\tO} \right)
	\; ,
	\end{align}
	where $C>0$ is a constant independent of the mesh size and the solution. 	%
\end{thm}
\begin{proof}
The proof relies on Strang's Lemma and the analysis of the consistency errors, as done in Sect. \ref{sec:ShiftedNitscheAccuracyAnalysis} for the Poisson problem; in particular, one uses Lemma \ref{lem:SBMorthogonality} and the estimate of Lemma \ref{lem:errorDir}, applied to each component of the velocity. We refer to the similar proof of Theorem 4 in~\cite{atallah2020analysis} for the technical details.
\end{proof}
\begin{rem}
Should the exact solution have a lower regularity than the one assumed in the Theorem (due to the presence of corners or edges, or of mixed Dirichlet and Neumann boundary conditions), the exponent of $h_{\tO}$ in \eqref{eq:convergenceNN} would be $<1$. We refer again to Lemma 3 in~\cite{TheoreticalPoissonAtallahCanutoScovazzi2020} for the necessary changes.
\end{rem}
%
 
\medskip
Finally, we propose an enhanced $L^{2}$ estimate for the velocity error that considerably improves over the one presented in~\cite{atallah2020analysis} in that we {\it do not} rely on the restrictive and unlikely assumption that the surrogate domain $\tO$ is convex. 

For simplicity, hereafter we assume $\Gamma_N=\emptyset$, although extensions are feasible at the cost of an increased technical burden.

%
\begin{thm}[Enhanced $L^{2}$-error estimate for the velocity $\bs{u}_h$]
\label{thm:StokesConvergenceL2}
Assume the hypotheses of Theorem~\ref{thm:StokesConvergenceNatural} hold, and in addition let $\Gamma_N=\emptyset$. Then, the numerical velocity $\bs{u}_h$ produced by SBM satisfies the following error estimate:
\begin{align}\label{eq:L2estimate}
\| \, \bs{u}-\bs{u}_h \, \|_{0,\tO}  
& \leq \;
C  \, h_{\tO}^{3/2} \, l(\tO)^{1/2} \, \mu^{-1/2} \, \left(\| \, \nabla (\nabla \bs{u}) \, \|_{0, \Om}  + \|\, \nabla p \, \|_{0,\tO} \right)
\; .
\end{align}
where $C$ is a positive constant independent of the mesh size and the solution.
\end{thm}
\begin{proof}
Given $\bs{z} \in L^2(\tO)^{n_d}$, let $\bar{\bs{z}} \in L^2(\Om)^{n_d}$ be its extension by $\bs{0}$ outside $\tO$ and let $[\bs{\psi} \, , \, \lambda] $ be the solution of the following homogeneous Dirichlet problem in $\Om$:
\begin{subequations}
\label{eq:SteadyStokes_dual}
\begin{align}
- \nabla \cdot ( 2 \mu \, \bs{\varepsilon}(\bs{\psi}) + \lambda \bs{I}  ) 
&=\; \mu \, \bar{\bs{z}} 
\qquad \! \text{in \ } \Om \; , \\
- \nabla \cdot  \bs{\psi}  &=\; 0
\qquad \text{\ \ in \ } \Om \; , \\
\bs{\psi} &=\; \bs{0}
\qquad \text{\ \ on \ }  \G \; .
\end{align}
\end{subequations}
The stated assumptions in addition to the fact that $\bar{\bs{z}} \in L^2(\Om)^{n_d}$ imply the regularity result $[\bs{\psi} \, , \, \lambda] \in H^2(\Om)^{n_d} \times H^1(\Omega)$, with the following bound
\begin{align}
\label{eq:dual_reg_arg}
 \| \, \mu^{1/2} \, \bs{\psi} \, \|_{2,\tO} 
+ \| \, \mu^{-1/2} \, \lambda \, \|_{1,\tO} 
\leq \;
 \| \, \mu^{1/2} \, \bs{\psi} \, \|_{2,\Om} 
+  \| \, \mu^{-1/2} \, \lambda \, \|_{1,\Om} 
\leq \;
Q \, \| \, \mu^{1/2} \, \bar{\bs{z}} \, \|_{0,\Om}
=
Q \, \| \, \mu^{1/2} \, \bs{z} \, \|_{0,\tO}
\; ,
\end{align}
where $Q>0$ is a non-dimensional constant independent of $\bar{\bs{z}}$ and the mesh size. \\
The same arguments that led to~\eqref{eq:bsu-g} show that on $\tG$ $\bs{\psi}$ satisfies
\begin{align}
\bs{S}_{h} \bs{\psi} + \bs{R}_{h} \bs{\psi}  = \bs{0} \;.
\end{align}
Since by assumption $\Gamma$ is of class ${\mathcal C}^2$ and $\bs{\psi}$ is in $H^2(\Om)^{n_d}$, we can apply Lemma \ref{lem:errorDir} to each component of $\bs{\psi}$, obtaining
\begin{align}
\label{eq:prop3_theoretical_vector_cc}
\| \, h^{-1/2} \, \bs{R}_h \bs{\psi} \, \|_{0,\tG} 
\leq 
C_{DR} \, h_{\G} \, | \, \bs{\psi} \, |_{2,\Om \setminus \tO} 
\leq 
C_{DR} \, Q \, h_{\tO} \, \| \, \bs{z}  \, \|_{0,\tO} \, ,
\end{align} 
where $C_{DR}$ is a positive constant independent of the mesh size.
Next, consider $[\bs{w}\, , \, q] \in  \bs{W}(\tO;\tT) $ and note that the pair $[\bs{\psi} \, , \, \lambda]$ also satisfies the following variational statement: 
\begin{align}
\mu \, ( \bs{z} \, , \, \bs{w} )_{\tO} 
& = \;
- ( \nabla \cdot ( 2 \mu \, \bs{\varepsilon}(\bs{\psi}) + \lambda \bs{I}  )  \, , \, \bs{w} )_{\tO} 
- ( \nabla \cdot  \bs{\psi}  \, , \, q )_{\tO}
\nonumber \\ 
& = \;
(   2 \mu \, \bs{\varepsilon}(\bs{\psi})   \, , \, \bs{\varepsilon}(\bs{w}) )_{\tO} 
+ (   \lambda  \, , \, \nabla \cdot \bs{w} )_{\tO} 
- ( \nabla \cdot  \bs{\psi}  \, , \, q )_{\tO}
-\avg{ 2 \mu \, \bs{\varepsilon}(\bs{\psi}) + \lambda \bs{I} \, , \, \bs{w} \otimes \ti{\bs{n}}}_{\tG}
\nonumber \\ 
& =  \;
(   2 \mu \, \bs{\varepsilon}(\bs{\psi})   \, , \, \bs{\varepsilon}(\bs{w}) )_{\tO} 
+ (   \lambda  \, , \, \nabla \cdot \bs{w} )_{\tO} 
- ( \nabla \cdot  \bs{\psi}  \, , \, q )_{\tO}
-\avg{ 2 \mu \, \bs{\varepsilon}(\bs{\psi}) + \lambda \bs{I} \, , \, \bs{S}_{h} \bs{w} \otimes \ti{\bs{n}}}_{\tG} 
\nonumber \\
& \phantom{=} \;
+ \avg{ 2 \mu \, \bs{\varepsilon}(\bs{\psi}) + \lambda \bs{I} \, , \, (\nabla \bs{w} \, \bs{d}) \otimes \ti{\bs{n}}}_{\tG}  \; . 	 
\end{align}
Adding residual terms that vanish by definition when applied to the exact solution, we have
\begin{align}
\mu \, ( \bs{z} \, , \, \bs{w} )_{\tO} 
& = \;
(   2 \mu \, \bs{\varepsilon}(\bs{\psi})   \, , \, \bs{\varepsilon}(\bs{w}) )_{\tO} 
+ (   \lambda  \, , \, \nabla \cdot \bs{w} )_{\tO} 
- ( \nabla \cdot  \bs{\psi}  \, , \, q )_{\tO}
-\avg{ 2 \mu \, \bs{\varepsilon}(\bs{\psi}) + \lambda \bs{I} \, , \, \bs{S}_{h} \bs{w} \otimes \ti{\bs{n}}}_{\tG} 
\nonumber \\
& \phantom{=} \;
+ \avg{ 2 \mu \, \bs{\varepsilon}(\bs{\psi}) + \lambda \bs{I} \, , \, (\nabla \bs{w} \, \bs{d}) \otimes \ti{\bs{n}}}_{\tG} 
- \avg{ (\bs{S}_{h} \bs{\psi} + \bs{R}_{h} \bs{\psi} ) \otimes \ti{\bs{n}} \, , \, 2 \mu \, \bs{\varepsilon}(\bs{w}) - q \bs{I}  }_{\tG}
\nonumber \\
& \phantom{=}  \;
+ \alpha \, \avg{  2 \mu \, h^{-1} \, (\bs{S}_{h} \bs{\psi} + \bs{R}_{h} \bs{\psi} ) \, , \, \bs{S}_{h} \bs{w} }_{\tG}
\nonumber \\ 
& = \;
\mathcal{B}([\bs{w} \, ,\,q];[\bs{\psi}\, , \, \lambda] )
- \gamma \sum_{T \in \ti{\cal T}_h} \left(  h^{2} \, (2\mu)^{-1} \, \left(\nabla q -   \nabla \cdot (2 \mu \, \bs{\varepsilon}(\bs{w})) \right)  \, , \, \nabla \lambda  \right)_{T} 
\nonumber \\ 
& \phantom{=}  \;
+ \avg{ 2 \mu \, \bs{\varepsilon}(\bs{\psi}) + \lambda \bs{I} \, , \, (\nabla \bs{w} \, \bs{d}) \otimes \ti{\bs{n}}}_{\tG} 
- \avg{ (\nabla \bs{\psi} \, \bs{d}) \otimes \ti{\bs{n}} \, , \, 2 \mu \, \bs{\varepsilon}(\bs{w}) - q \bs{I} }_{\tG} 
\nonumber \\ 
& \phantom{=}  \;
- \avg{ \bs{R}_{h} \bs{\psi} \otimes \ti{\bs{n}} \, , \, 2 \mu \, \bs{\varepsilon}(\bs{w}) - q \bs{I} }_{\tG}
+ \alpha \, \avg{ 2 \mu \, h^{-1} \,  \bs{R}_{h} \bs{\psi} \, , \, \bs{S}_{h} \bs{w} }_{\tG}
\end{align}
Let us pick $\bs{w} = \bs{z} = \bs{e}_{\bs{u}} := \bs{u} - \bs{u}_h $, $q = e_{p} := p - p_h $. Using Lemma~\ref{lem:SBMorthogonality} with $[\bs{v}_h\, , \,  \omega_h] = [\bs{\psi}_{I} \, , \, \lambda_{I}] := [\mathcal{I}_h(\bs{\psi})\, , \, \mathcal{I}_h(\lambda)]$, where $\mathcal{I}_h$ denotes the Scott-Zhang interpolant at the nodes of the triangulation, we write
\begin{subequations}
\begin{align}
\label{eq:duality_terms}
 \mu \, \| \, \bs{e}_{\bs{u}} \, \|^{2}_{\tO}
& =  \;
\mathcal{B}([ \bs{e}_{\bs{u}} \, ,\, e_p];[\bs{\psi} - \bs{\psi}_I \, , \, \lambda - \lambda_I ] )
+ \mathcal{E}_{\rm stab}( [ \bs{e}_{\bs{u}} \, ,\, e_p ]; [ \bs{0} \, , \, \lambda ] )
+ \mathcal{E}_{\rm sym}( [ \bs{e}_{\bs{u}} \, ,\, e_p ]; [ \bs{\psi} \, , \, \lambda ] )
\nonumber \\
& \phantom{=} \; 
+ \mathcal{E}_{\rm rem}( [ \bs{e}_{\bs{u}} \, ,\, e_p ]; [ \bs{\psi} \, , \, \lambda ] )
+ \mathcal{E}_{\rm ort}( [ \bs{u} \, ,\, p ]; [ \bs{\psi}_I \, , \, \lambda_I ] )
\end{align}
with
\begin{align}
\mathcal{E}_{\rm stab}( [ \bs{e}_{\bs{u}} \, ,\, e_p ]; [ \bs{0} \, , \, \lambda ] )
&:=\;
- \gamma \sum_{T \in \ti{\cal T}_h} \left(  h^{2} \, (2\mu)^{-1} \, \left(\nabla e_p -   \nabla \cdot (2 \mu \, \bs{\varepsilon}(\bs{e}_{\bs{u}})) \right)  \, , \, \nabla \lambda  \right)_{T} 
\; ,
\\[.2cm]
\mathcal{E}_{\rm sym}( [ \bs{e}_{\bs{u}} \, ,\, e_p ]; [ \bs{\psi} \, , \, \lambda ] )
&:=\;
\avg{ 2 \mu \, \bs{\varepsilon}(\bs{\psi}) + \lambda \bs{I} \, , \, (\nabla \bs{e}_{\bs{u}} \, \bs{d}) \otimes \ti{\bs{n}}}_{\tG} 
- \avg{ (\nabla \bs{\psi} \, \bs{d}) \otimes \ti{\bs{n}} \, , \, 2 \mu \, \bs{\varepsilon}(\bs{e}_{\bs{u}}) - e_p \bs{I} }_{\tG} 
\; ,
\\[.2cm]
\mathcal{E}_{\rm rem}( [ \bs{e}_{\bs{u}} \, ,\, e_p ]; [ \bs{\psi} \, , \, \lambda ] )
&:=\;
- \avg{ \bs{R}_{h} \bs{\psi} \otimes \ti{\bs{n}} \, , \, 2 \mu \, \bs{\varepsilon}(\bs{e}_{\bs{u}}) - e_p  \bs{I} }_{\tG}
+ \alpha \, \avg{ 2 \mu \, h^{-1} \,  \bs{R}_{h} \bs{\psi} \, , \, \bs{S}_{h} \bs{e}_{\bs{u}} }_{\tG} 
\; ,
\\[.2cm]
\mathcal{E}_{\rm ort}( [ \bs{u} \, ,\, p ]; [ \bs{\psi}_I \, , \, \lambda_I ] ) 
&:= \;
\avg{ \bs{R}_{h} \bs{u} \otimes \ti{\bs{n}}  \, , \, 2 \mu \, \bs{\varepsilon}(\bs{\psi}_I) + \lambda_I \bs{I} }_{\tG}
- \alpha \, \avg{ 2 \mu \, h^{-1}   \, \bs{R}_{h} \bs{u} \, , \, \bs{S}_{h} \bs{\psi}_I   }_{\tG}
\; .
\end{align}
\end{subequations}
We proceed to bound the error terms on the right-hand side of \eqref{eq:duality_terms}.
Recalling Proposition 3 and Proposition 5 in~\cite{atallah2020analysis}, we have, for suitable constants,
\begin{subequations}
\begin{align}
\label{eq:eBilinearStokes}
|\, \mathcal{B}( [ \bs{e}_{\bs{u}} \, ,\, e_p];[\bs{\psi} - \bs{\psi}_I \, , \, \lambda - \lambda_I ] )\, |
& \leq \;
C_{\cal B} \, \| \, [ \bs{e}_{\bs{u}} \, ,\, e_p] \, \|_{\bs{W}(\tO;\tT)} 
\, \| \, [\bs{\psi} - \bs{\psi}_I \, , \, \lambda - \lambda_I ]  \, \|_{\bs{W}(\tO;\tT)}
\nonumber \\
& \leq \;
 C_{\cal B}\, C_{APP} \, h_{\tO} \left( \, \mu^{1/2} \, | \, \bs{\psi} \, |_{2, \tO} + \mu^{-1/2} \, | \,  \lambda \, |_{1,\tO} \right) \, 
\| \, [ \bs{e}_{\bs{u}} \, ,\, e_p] \, \|_{\bs{W}(\tO;\tT)} 
\nonumber \\
& \leq \;
C_{7} \, h_{\tO} \, \mu^{1/2} \, \| \, [ \bs{e}_{\bs{u}} \, ,\, e_p] \, \|_{\bs{W}(\tO;\tT)} \, \| \, \bs{e}_{\bs{u}} \, \|_{0, \tO} \; .
\end{align}
From the definition of the norm $\bs{W}(\tO;\tT)$, we immediately get
\begin{align}
\label{eq:eStabStokes}
| \, \mathcal{E}_{\rm stab}( [ \bs{e}_{\bs{u}} \, ,\, e_p ]; [ \bs{0} \, , \, \lambda ] ) \, |
&\leq\;
\|  \, (2 \mu)^{-1/2} \, h \, \nabla \lambda \, \|_{0,\tO} 
\left( 
\|  \, (2 \mu)^{-1/2} \, h \, \nabla e_{p} \, \|_{0,\tO}  
+ \|  \, (2 \mu)^{1/2} \, h \, \nabla \cdot \bs{\varepsilon}(\bs{e}_{\bs{u}}) \, \|_{0,\tO}  
\right)
\nonumber \\ 
& \leq\;
C_8 \,  h_{\tO} \, \mu^{1/2} \,
\| \, [ \bs{e}_{\bs{u}} \, ,\, e_p] \, \|_{\bs{W}(\tO;\tT)} \, 
\| \, \bs{e}_{\bs{u}} \, \|_{0, \tO} \; .
\end{align}
Recalling Assumption~\ref{ass:d_asym}, Theorem~\ref{thm:TraceTheorem}, and Theorem~\ref{theo:BoundaryTraceIneq}, we obtain
\begin{align} 
\label{eq:eSymStokes}
| \, \mathcal{E}_{\rm sym}( [ \bs{e}_{\bs{u}} \, ,\, e_p ]; [ \bs{\psi} \, , \, \lambda ] ) \, |
&\leq\;
{c_d} \, \hat{h}_{\tG}^{\zeta} \, 
\|  \, (2 \mu \, h)^{1/2} \, \nabla \bs{e}_{\bs{u}}  \cdot \bs{\nu} \, \|_{0,\tG}  \, 
\left( 
\| \, (2 \mu \, h)^{1/2} \, \bs{\varepsilon}(\bs{\psi}) \, \|_{0,\tG}  
+ \| \, (2 \mu \, h^{-1})^{-1/2} \,  \lambda \, \|_{0,\tG} 
\right)
\nonumber \\
& \phantom{=} \;
+ {c_d} \, \hat{h}_{\tG}^{\zeta} \, 
\| \, (2 \mu \, h)^{1/2} \, \nabla \bs{\psi}  \, \|_{0,\tG} 
\left(
\|  \, (2 \mu \, h)^{1/2} \, \bs{\varepsilon}(\bs{e}_{\bs{u}}) \, \|_{0,\tG}  \, 
+ \|  \, (2 \mu \, h^{-1})^{-1/2} \, e_p \, \|_{0,\tG}  \, 
\right)
\nonumber \\
&\leq\;
{c_d} \, C  \, C_{I} \,  h_{\tO}^{1/2} \,
\| \, [ \bs{e}_{\bs{u}} \, ,\, e_p]\,\|_{\bs{W}(\tO;\tT)} 
\, \times
\nonumber \\
&\phantom{\leq} \; \; \; \times
\left( \, l(\tO)^{-3/2} \, \mu^{1/2} \, \| \,  \bs{\psi} \, \|_{2,\tO}  
+  l(\tO)^{-1/2} \,\mu^{-1/2} \, \| \,  \lambda \, \|_{1,\tO}
\right)
\nonumber \\ 
& \leq\;
C_9 \,  h_{\tO}^{1/2} \, l(\tO)^{1/2} \,  \mu^{1/2} \,
\| \, [ \bs{e}_{\bs{u}} \, ,\, e_p] \, \|_{\bs{W}(\tO;\tT)} \, 
\| \, \bs{e}_{\bs{u}} \, \|_{0, \tO} \; .
\end{align}
Invoking now the bound \eqref{eq:prop3_theoretical_vector_cc}, we have
\begin{align} 
\label{eq:eRemStokes}
| \, \mathcal{E}_{\rm rem}( [ \bs{e}_{\bs{u}} \, ,\, e_p ]; [ \bs{\psi} \, , \, \lambda ] ) \, |
&\leq\;
\left( 
\|  \, (2 \mu \, h)^{1/2} \,\bs{\varepsilon}(\bs{e}_{\bs{u}}) \, \|_{0,\tG} 
+ \|  \, (2 \mu \, h^{-1})^{-1/2} \, e_p \, \|_{0,\tG} 
\right)
\| \, (2 \mu \, h^{-1})^{1/2} \,  \bs{R}_{h} \bs{\psi} \, \|_{0,\tG}  
\nonumber \\
& \phantom{=} \;
+ \alpha \, \|  \, (2 \mu \, h^{-1})^{1/2} \, \bs{S}_{h} \bs{e}_{\bs{u}} \, \|_{0,\tG} \,
\| \, (2 \mu \, h^{-1})^{1/2} \,  \bs{R}_{h} \bs{\psi} \, \|_{0,\tG}  
\nonumber \\
&\leq\;
\left( 2C_I + \alpha \right)  \, \mu^{1/2} \, 
\| \, [ \bs{e}_{\bs{u}} \, ,\, e_p] \, \|_{\bs{W}(\tO;\tT)} \,
\| \,  h^{-1/2} \,  \bs{R}_{h} \bs{\psi} \, \|_{0,\tG} 
\nonumber \\
&\leq\;
C_{10} \, h_{\tO} \, \mu^{1/2} \, 
\| \, [ \bs{e}_{\bs{u}} \, ,\, e_p] \, \|_{\bs{W}(\tO;\tT)} \,
\| \, \bs{e}_{\bs{u}} \, \|_{0, \tO}  
 \; .
\end{align}
At last, recalling the classical error estimates for the Scott-Zhang interpolant, we get
\begin{align}
\label{eq:eOrtStokes}
| \, \mathcal{E}_{\rm ort}( [ \bs{u} \, ,\, p ]; [ \bs{\psi}_I \, , \, \lambda_I ] ) \, |
& = \; 
| \, -\avg{\bs{R}_{h} \bs{u} \otimes \ti{\bs{n}} \, , \, 2 \mu \, \bs{\varepsilon}(\bs{\psi} - \bs{\psi}_I) + (\lambda - \lambda_I) \bs{I} - 2 \mu \, \bs{\varepsilon}(\bs{\psi}) - \lambda  \bs{I} }_{0,\tG} 
\nonumber \\
& \phantom{=} \;
+ \alpha \, \avg{  2 \mu \, h^{-1} \, \bs{R}_{h} \bs{u} \, , \, \bs{S}_{h}  ( \bs{\psi} - \bs{\psi}_I) - \bs{\psi} }_{0,\tG} \, |
\nonumber \\ 
& \leq \;
\| \, (2 \mu \, h^{-1})^{1/2} \, \bs{R}_{h} \bs{u} \, \|_{0,\tG} \, 
\left( 
\|  \, (2 \mu \, h)^{1/2} \, \bs{\varepsilon}(\bs{\psi} - \bs{\psi}_I)  \, \|_{0,\tG}
+   \|  \, (2 \mu \, h)^{1/2} \, \bs{\varepsilon}(\bs{\psi})   \, \|_{0,\tG}
\right)
\nonumber \\
& \phantom{=} \;
+\| \, (2 \mu \, h^{-1})^{1/2} \, \bs{R}_{h} \bs{u} \, \|_{0,\tG} \, 
\left( 
\|  \, (2 \mu \, h^{-1})^{-1/2} \, (\lambda - \lambda_I)  \, \|_{0,\tG}
+   \|  \, (2 \mu \, h^{-1})^{-1/2} \, \lambda   \, \|_{0,\tG}
\right)
\nonumber \\
& \phantom{=} \;
 + \alpha \, \| \, (2 \mu \, h^{-1})^{1/2} \, \bs{R}_{h} \bs{u} \, \|_{0,\tG} \, 
\, \times
\nonumber \\
&\phantom{\leq} \; \; \; \times
\left( 
\|  \, (2 \mu \, h^{-1})^{1/2} \, \bs{S}_{h}  ( \bs{\psi} - \bs{\psi}_I)  \, \|_{0,\tG}
+  \|  \, (2 \mu \, h^{-1})^{1/2} \, \bs{S}_{h}  \bs{\psi}  \, \|_{0,\tG}
\right)
\nonumber \\
& \leq \; 
C_{11} \, h_{\tO}^{3/2}  \, l(\tO)^{1/2} \,  \mu \, 
\| \, \bs{e}_{\bs{u}} \, \|_{0, \tO}
\| \,  \nabla(\nabla \bs{u}) \, \|_{0,\Om \setminus \tO} \; . 
\end{align}
\end{subequations}
Thus, combining~\eqref{eq:eBilinearStokes},~\eqref{eq:eStabStokes},~\eqref{eq:eSymStokes},~\eqref{eq:eRemStokes}, and~\eqref{eq:eOrtStokes} in~\eqref{eq:duality_terms} yields
\begin{align}
\| \, \bs{u} - \bs{u}_h \, \|_{0,\tO} 
\leq 
C_{ANS} \, 
 h_{\tO}^{1/2} \, l(\tO)^{1/2} \, \mu^{-1/2} \,
 \| \, [ \bs{u} - \bs{u}_h \, ,\, p - p_h] \, \|_{\bs{W}(\tO;\tT)} \, 
\;.
\end{align}
The right-hand side can be bounded using Theorem~\ref{thm:StokesConvergenceNatural}, which concludes the proof.
\end{proof}

\begin{rem}
	The previous bound is clearly sub-optimal due to the terms in $\mathcal{E}_{\rm sym}( [ \bs{e}_{\bs{u}} \, ,\, e_p ]; [ \bs{\psi} \, , \, \lambda ] )$. However, it is not clear at the moment if the above estimate is sharp, since in computations we always observe optimal, second-order convergence rates. 
\end{rem}
%


\begin{table}[hbt!]\centering
	\begin{tabular}{M{1.8cm} | M{6cm}  | M{5.8cm}  }
		\multicolumn{3}{c}{}  \\  
		Mesh Size & No. of surrogate edges with $\bs{\nu} \cdot \bs{n} \leq 0$  & Percentage of total surrogate edges   \\
		\hline
		4.00E-02 & 1 & 4.35\%   \\
		2.00E-02 & 1 & 2.33\%   \\
		1.00E-02 & 5 & 5.43\%   \\
		5.00E-03 & 9 & 5.06\%  \\
		2.50E-03 & 23 & 6.35\%  \\
		1.25E-03 & 38 & 5.38\%   \\
	\end{tabular}
	\caption{Number of surrogate edges with $\bs{\nu} \cdot \bs{n} \leq 0$, as the computational grid of Figure~\ref{fig:trapezoid_all} is refined. \label{table1} }
\end{table}
\section{Two-dimensional numerical tests}
\label{sec:2DNumerical_Results}
In this section, we perform convergence tests comparing - on grids of similar size - the new proposed SBM formulation with the corresponding body-fitted formulation for the Poisson and Stokes flow equations. 
Our numerical tests are performed on a domain $\Om$ given by the right trapezoid of height $s = 1$ and bases $b_{1} = 0.6$ and $b_{2} = 0.4$ as shown in Figure~\ref{fig:trapezoid}. 
To verify that the SBM formulations do not require the geometric resolution condition $\inf_{\tG} \bs{\nu} \cdot \bs{n} > 0$, the computational grids are carefully built so as to violate it, as shown in Figure~\ref{fig:surrogate}. 
In particular, the computational grids are constructed by splitting into {\it four} equal-area triangles each of the rectangular elements in a Cartesian grid of aspect ratio 5:1.
 Table~\ref{table1} accounts for the absolute and relative number of surrogate edges of each mesh for which the geometric resolution assumption is violated. 
The grids used for the body-fitted method are virtually identical to those used for the SBM, in the sense that they are obtained by moving the true boundary $\G$ by $-10^{-15}$ vertically, so as to attain $\tG = \G$ (see Figure~\ref{fig:surrogate}). 

The algebraic system of equations is solved using a smoothed aggregation multigrid method as preconditioner from Sandia's ML Library, part of the Trilinos Software Project~\cite{Tong2000MLSandia}.
\begin{figure}
	\centering
\begin{subfigure}[b]{.45\textwidth}\centering
\begin{tikzpicture}[scale=0.8]
\draw[line width = 0.5mm,black,name path=s1] (0.0,0.0) -- (3,0.0);
\draw[line width = 0.5mm,black,name path=s1] (3.0,0.0) -- (4.5,7.5);
\draw[line width = 0.5mm,black,name path=s1] (4.5,7.5) -- (0.0,7.5) ;
\draw[line width = 0.5mm,black,name path=s2] (0,7.5) -- (0.0,0.0);
\draw [black, draw=none,name path=p1] plot coordinates { (0.0,0.0) (3,0.0) (4.5,7.5) (0.0,7.5) (0.0,0.0) };
\draw [black, draw=none,name path=p2] plot coordinates { (0.0,7.5) (0.0,0.0) };
\tikzfillbetween[of=p1 and p2,split]{gray!25!};
\dimline[extension start length =0,extension end length = 0,line style = {line width=0.75}]{(0,-0.5)}{(3,-0.5)}{$b_{2}$};
\dimline[extension start length =0,extension end length = 0,line style = {line width=0.75}]{(-0.5,0)}{(-0.5,7.5)}{$s$};
\dimline[extension start length =0,extension end length = 0,line style = {line width=0.75}]{(0,8)}{(4.5,8)}{$b_{1}$};
\end{tikzpicture}
\caption{The domain $\Om$, a right trapezoid with $s = 1$, $b_{1} = 0.6$ and $b_{2} = 0.4$.}
\label{fig:trapezoid}
\end{subfigure}
\hspace{0.5cm}
\begin{subfigure}[b]{.45\textwidth}\centering
\includegraphics[width=0.65\textwidth]{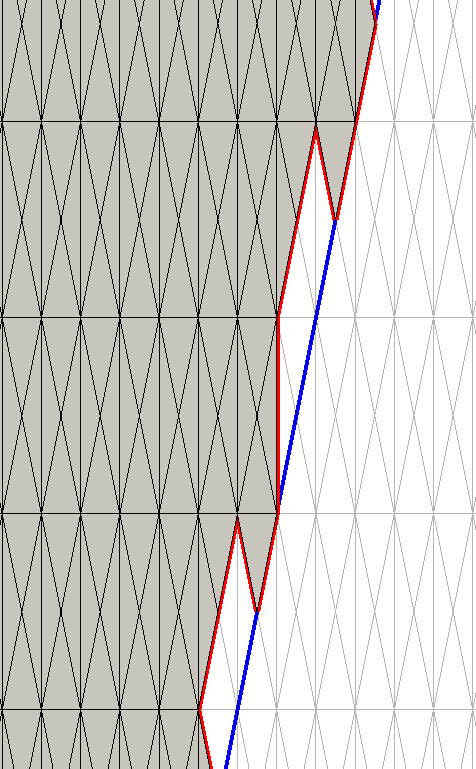}
\caption{Zoomed view: the surrogate domain $\ti{\Om}$ (grey), the true boundary ${\G}$ (blue), and the surrogate boundary $\tG$ (red).}
\label{fig:surrogate}
\end{subfigure}
\caption{The true domain $\Om$ (left) and the surrogate domain $\tO$ (right).}
\label{fig:trapezoid_all}
\end{figure}

\subsection{Poisson problem}
\label{sec:Poisson}
In this first test, we considered Poisson's equation defined on $\Om$ with the manufactured solution
\begin{equation}
u(x,y) = y \sin(2 \pi x) - x  \cos(2 \pi y) \; . 
\end{equation}
Dirichlet boundary conditions are applied on all boundaries and the Nitsche penalty parameter is set as $\alpha = 10$.
Figures~\ref{fig:poisson_solution} and~\ref{fig:poisson_conv} show the numerical solution  and $L^{2}$-error rates for the SBM and body-fitted variational forms.
It is apparent that violating the condition $\inf_{\tG} \bs{\nu} \cdot \bs{n} > 0$ has no effect on the convergence rate of the SBM. In fact, the $L^2$-norm of the SBM error converges quadratically, and therefore faster than the rate $3/2$ theoretically predicted, as shown in Table~\ref{table2}. Such behavior has been observed in all computations performed to date, including the ones reported in~\cite{main2018shifted0,main2018shifted,atallah2020analysis}.
\begin{figure}[bt]
	\centering
	\hspace{-0.6cm}
\begin{subfigure}[b]{0.08\textheight} \centering
\includegraphics[height=2.35\textwidth]{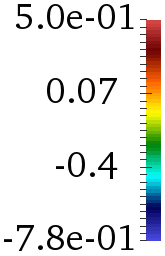}
\caption*{}
\label{fig:poisson_scale}
\end{subfigure}
\quad
\begin{subfigure}[b]{.35\textwidth}\centering
\includegraphics[width=0.75\textwidth]{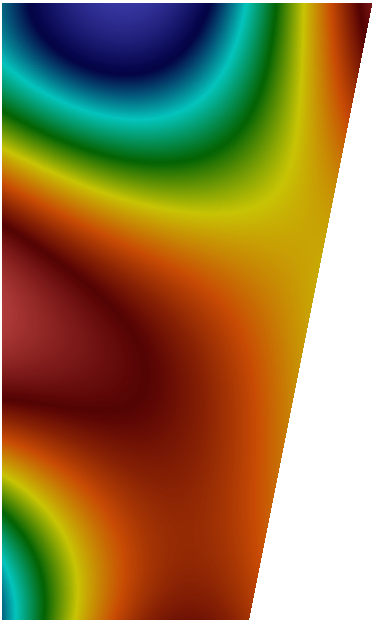}
\caption{Numerical solution.}
\label{fig:poisson_solution}
\end{subfigure}
\begin{subfigure}[b]{.5\textwidth}\centering
	\centering
\begin{tikzpicture}[scale= 1]
\begin{loglogaxis}[
grid=major,x post scale=1.0,y post scale =1.0,
xlabel={$h$},
legend style={font=\footnotesize,legend cell align=left},
legend pos= north west,
axis equal	]
\addplot[ mark=x,blue,line width=1pt] coordinates {
	(1.25E-03,4.96E-06)
	(2.50E-03,1.98E-05)
	(5.00E-03,7.92E-05)
	(1.00E-02,3.16E-04)
	(2.00E-02,1.26E-03)
	(4.00E-02,4.95E-03)
}; 
\addplot[ mark=x,red,line width=1pt] coordinates {
	(1.25E-03,4.98E-06)
	(2.50E-03,1.99E-05)
	(5.00E-03,7.96E-05)
	(1.00E-02,3.19E-04)
	(2.00E-02,1.28E-03)
	(4.00E-02,5.12E-03)
}; 
	\addplot[black,line width=1pt] coordinates {
	(3E-3,0.00001)
	(6E-3,0.00004)
	(6E-3,0.00001)
	(3E-3,0.00001)
}; 
\legend{SBM,body-fitted}
\end{loglogaxis}
\node[text width=0.2cm] at (3.5,1.4) {\footnotesize $2$};
\node[text width=0.2cm] at (3.15,0.76) {\footnotesize $1$};
\end{tikzpicture}
\caption{Convergence rate of $ \, \| \, u - u^h) \, \|_{0,\Om} $.}
\label{fig:poisson_conv}
	\end{subfigure}
	\caption{Poisson problem: plots of the solution and error convergence rates.}
\end{figure}
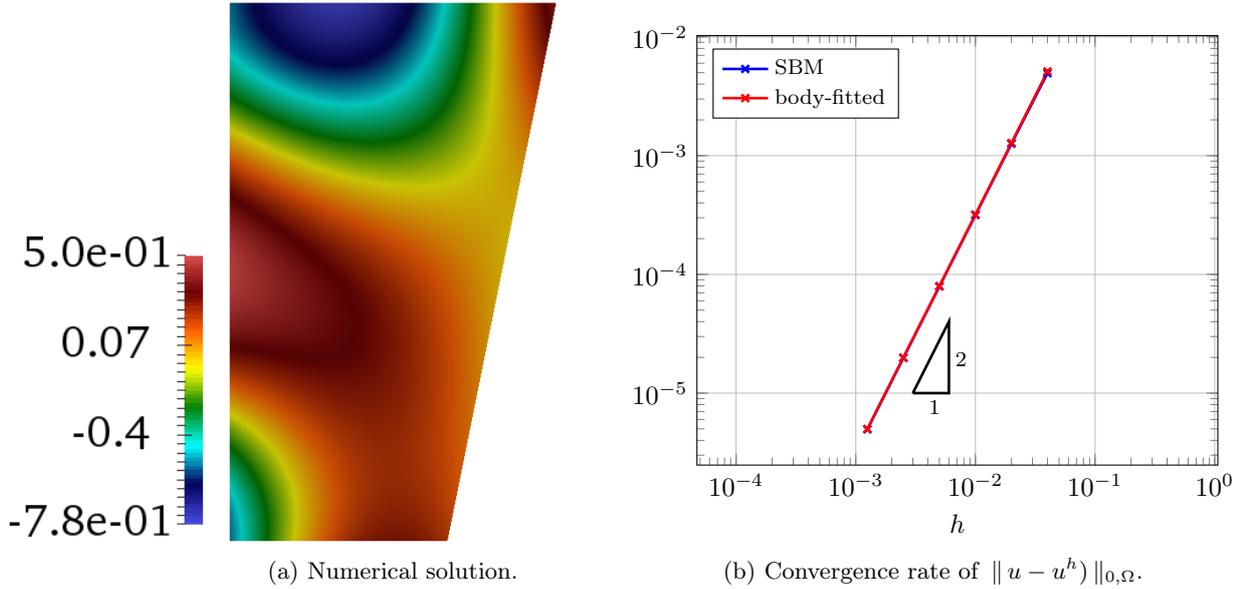
\begin{table}[bt]\centering
	\begin{tabular}{M{1.8cm} | M{2.3cm}  M{1.5cm} | M{2.3cm}  M{1.5cm} }
		& \multicolumn{2}{c|}{Body-fitted}  & \multicolumn{2}{c}{ SBM }  \\ 
		Mesh Size & $ \| \, u - u^h \, \|_{0,\Om} $ & Rate & $ \| \, u - u^h \, \|_{0,\tO} $ & Rate \\
		\hline
		4.00E-02 & 4.95E-03 & - & 5.12E-03 & -\\
		2.00E-02 & 1.26E-03 & 2.00  & 1.28E-03 & 2.00\\
		1.00E-02 & 3.16E-04 & 2.00  & 3.19E-04 & 2.00\\
		5.00E-03 & 7.92E-05 & 2.00  & 7.96E-05 & 2.00\\
		2.50E-03 & 1.98E-05 & 2.00  & 1.99E-05 & 2.00 \\
		1.25E-03 & 4.96E-06 & 2.00 & 4.98E-06 & 2.00\\
	\end{tabular}
	\caption{ Convergence rates for the Poisson equation using the body-fitted and the SBM approaches. \label{table2} }
\end{table}

\begin{figure}[t!]\centering	
	\label{fig:Stokes_solutions}
	\begin{subfigure}[b]{0.08\textheight} \centering
		\includegraphics[height=2\textwidth]{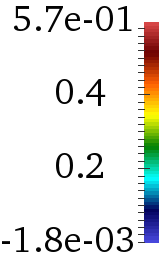}
		\caption*{}
		\label{fig:pressure_scale}
	\end{subfigure}
		\begin{subfigure}[b]{.27\textwidth}\centering
	\includegraphics[width=0.75\linewidth]{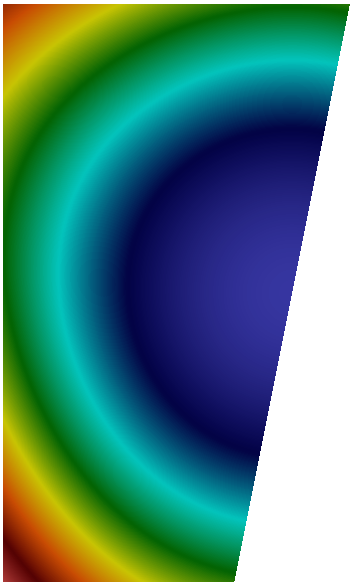}
		\caption{Pressure numerical solution.}
		\label{fig:efig1}
	\end{subfigure}	
\qquad
	\begin{subfigure}[b]{0.08\textheight} \centering
	\includegraphics[height=2.35\textwidth]{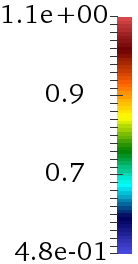}
	\caption*{}
	\label{fig:velocity_scale}
	\end{subfigure}
	\begin{subfigure}[b]{.27\textwidth}\centering
	\includegraphics[width=0.75\linewidth]{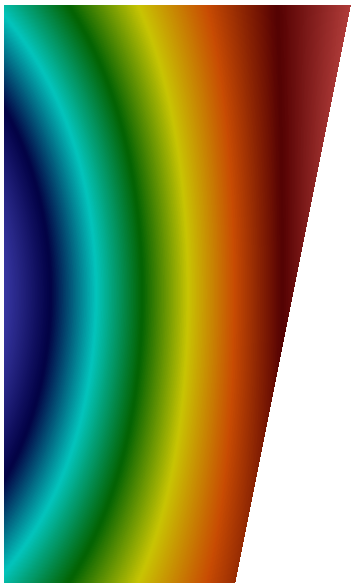}
	\caption{Velocity numerical solution.}
	\label{fig:efig2}
	\end{subfigure}	
\caption{Stokes flow problem: Solution plots.}
\end{figure}
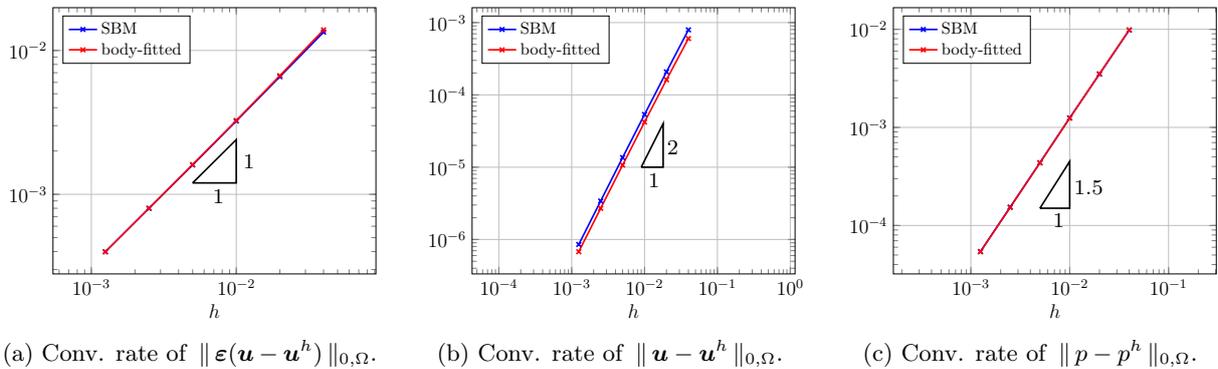
\begin{figure}[t!]\centering	
\begin{subfigure}[b]{.3\linewidth}\centering
	\centering
	\begin{tikzpicture}[scale= 0.62]
	\begin{loglogaxis}[
	grid=major,x post scale=1.0,y post scale =1.0,
	tick label style={font=\large},
	xlabel style={font=\large},
	xlabel={$h$},
	legend style={font=\normalsize,legend cell align=left},
	legend pos= north west,
	axis equal	]
	\addplot[ mark=x,blue,line width=1pt] coordinates {
		(1.25E-03,3.99E-04)
		(2.50E-03,7.99E-04)
		(5.00E-03,1.60E-03)
		(1.00E-02,3.23E-03)
		(2.00E-02, 6.57E-03)
		(4.00E-02,1.34E-02)
	}; 
	\addplot[ mark=x,red,line width=1pt] coordinates {
		(1.25E-03,3.99E-04)
		(2.50E-03,8.01E-04)
		(5.00E-03,1.61E-03)
		(1.00E-02,3.26E-03)
		(2.00E-02,6.68E-03  )
		(4.00E-02,1.39E-02)
	}; 
	\addplot[black,line width=1pt] coordinates {
		(5E-3,0.00120)
		(10E-3,0.00240)
		(10E-3,0.00120)
		(5E-3,0.00120)
	}; 
	\legend{SBM,body-fitted}
	\end{loglogaxis}
	\node[text width=0.2cm] at (4.2,2.4) {\footnotesize $1$};
	\node[text width=0.2cm] at (3.55,1.65) {\footnotesize $1$};
	\end{tikzpicture}
	\caption{Conv. rate of $ \, \| \, \bs{\varepsilon}(\bs{u}- \bs{u}^h) \, \|_{0,\Om} $.}
	\label{fig:efig5}
\end{subfigure}
\hspace{0.36cm}
\begin{subfigure}[b]{.3\linewidth}\centering
			\centering
\begin{tikzpicture}[scale= 0.62]
\begin{loglogaxis}[
grid=major,x post scale=1.0,y post scale =1.0,
tick label style={font=\large},
xlabel style={font=\large},
xlabel={$h$},
legend style={font=\normalsize,legend cell align=left},
legend pos= north west,
axis equal	]
\addplot[ mark=x,blue,line width=1pt] coordinates {
	(1.25E-03,8.54E-07)
	(2.50E-03,3.41E-06)
	(5.00E-03,1.36E-05)
	(1.00E-02,5.36E-05)
	(2.00E-02,2.08E-04)
	(4.00E-02,7.93E-04)
}; 
\addplot[ mark=x,red,line width=1pt] coordinates {
	(1.25E-03,6.77E-07)
	(2.50E-03,2.70E-06)
	(5.00E-03,1.07E-05)
	(1.00E-02,4.21E-05)
	(2.00E-02,1.62E-04)
	(4.00E-02,6.01E-04)
}; 
\addplot[black,line width=1pt] coordinates {
	(9E-3,0.00001)
	(18E-3,0.00004)
	(18E-3,0.00001)
	(9E-3,0.00001)
}; 
\legend{SBM,body-fitted}
\end{loglogaxis}
\node[text width=0.2cm] at (4.3,2.7) {\footnotesize $2$};
\node[text width=0.2cm] at (3.9,2) {\footnotesize $1$};
\end{tikzpicture}
\caption{Conv. rate of $ \, \| \, \bs{u} - \bs{u}^h \, \|_{0,\Om}$.}
\label{fig:efig3}
\end{subfigure}
\hspace{0.36cm}
\begin{subfigure}[b]{.3\linewidth}\centering
	\centering
	\begin{tikzpicture}[scale= 0.62]
	\begin{loglogaxis}[
	grid=major,x post scale=1.0,y post scale =1.0,
	tick label style={font=\large},
	xlabel style={font=\large},
	xlabel={$h$},
	legend style={font=\normalsize,legend cell align=left},
	legend pos= north west,
	axis equal	]
	\addplot[ mark=x,blue,line width=1pt] coordinates {
		(1.25E-03,5.43E-05)
		(2.50E-03,1.54E-04)
		(5.00E-03,4.37E-04)
		(1.00E-02,1.25E-03)
		(2.00E-02,3.49E-03)
		(4.00E-02,9.81E-03)
	}; 
	\addplot[ mark=x,red,line width=1pt] coordinates {
		(1.25E-03,5.40E-05)
		(2.50E-03,1.53E-04)
		(5.00E-03,4.35E-04)
		(1.00E-02,1.24E-03)
		(2.00E-02,3.52E-03 )
		(4.00E-02,9.87E-03)
	}; 
	\addplot[black,line width=1pt] coordinates {
		(5E-3,0.00015)
		(10E-3,0.00045)
		(10E-3,0.00015)
		(5E-3,0.00015)
	}; 
	\legend{SBM,body-fitted}
	\end{loglogaxis}
	\node[text width=0.2cm] at (4,1.85) {\footnotesize $1.5$};
	\node[text width=0.2cm] at (3.55,1.15) {\footnotesize $1$};
	\end{tikzpicture}
	\caption{Conv. rate of $ \, \| \, p - p^h \, \|_{0,\Om} $.}
	\label{fig:efig4}
\end{subfigure}
	\caption{Convergence rates for the Stokes problem test case. \label{fig:Stokes_conv_rates}}
\end{figure}
\begin{table}[hbt!]\centering
	\begin{tabular}{M{1.8cm} | M{2.6cm}  M{1.2cm} |  M{2.3cm}  M{1.2cm} | M{2.1cm}  M{1.2cm} }
		\multicolumn{7}{c}{Body-fitted}  \\  \cline{1-7}
		Mesh Size & $ \, \| \, \bs{\varepsilon}(\bs{u}- \bs{u}^h) \, \|_{0,\Om} $ & Rate &  $ \, \| \, \bs{u} - \bs{u}^h \, \|_{0,\Om}$ & Rate & $ \, \| \, p - p^h \, \|_{0,\Om} $ & Rate   \\
		\hline
		4.00E-02 & 1.39E-02 & - & 6.01E-04 & - & 9.87E-03  & -  \\
		2.00E-02 & 6.68E-03 & 1.06 & 1.62E-04 & 1.89 & 3.52E-03  & 1.49  \\
		1.00E-02 & 3.26E-03 & 1.03 & 4.21E-05 & 1.95 & 1.24E-03 & 1.51 \\
		5.00E-03 & 1.61E-03 & 1.02 & 1.07E-05 & 1.98 & 4.35E-04 & 1.51  \\
		2.50E-03 & 8.01E-04 & 1.01 & 2.70E-06 & 1.99 & 1.53E-04 & 1.51   \\
		1.25E-03 & 3.99E-04 & 1.00 & 6.77E-07 & 1.99 & 5.40E-05 & 1.50  \\
		\multicolumn{7}{c}{ } \\
		\multicolumn{7}{c}{ SBM } \\  \cline{1-7}
		Mesh Size & $ \, \| \, \bs{\varepsilon}(\bs{u}- \bs{u}^h) \, \|_{0,\ti{\Om}}$ & Rate &  $ \, \| \, \bs{u} - \bs{u}^h \, \|_{0,\tO}$ & Rate & $ \, \| \, p - p^h \, \|_{0,\tO} $ & Rate \\
		\hline
		4.00E-02 & 1.34E-02 & - & 7.93E-04 & - & 9.81E-03  & -  \\
		2.00E-02 & 6.57E-03 & 1.03 & 2.08E-04 & 1.93 & 3.49E-03  & 1.49  \\
		1.00E-02 & 3.23E-03 & 1.02 & 5.36E-05 & 1.96 & 1.25E-03 & 1.48  \\
		5.00E-03 & 1.60E-03 & 1.01 & 1.36E-05 & 1.97 & 4.37E-04 & 1.51  \\
		2.50E-03 & 7.99E-04 & 1.01 & 3.41E-06 & 2.00 & 1.54E-04 & 1.50   \\
		1.25E-03 & 3.99E-04 & 1.00 & 8.54E-07 & 2.00 & 5.43E-05 & 1.50  \\
	\end{tabular}
	\caption{Convergence rates for Stokes flow problem using the body-fitted and SBM approaches. \label{table3} }
\end{table}
%
%
\subsection{Stokes flow problem}
\label{sec:Stokes}
We then computed a solution to the Stokes flow problem defined on the same domain $\Om$ of the Poisson problem, and with the manufactured solution proposed in~\cite{atallah2020analysis}, and given as
\[
\begin{cases}
p(x,y) = x^{2} e^{xy} + y^{2} \; , \\
u_{x}(x,y) = - (-0.2x^{3} -0.2x^{2} +x + 1 ) \; \cos(y) \; , \\
u_{y}(x,y) =  (-0.6x^{2} -0.4x + 1 ) \; \sin(y) \; .
\end{cases}
\]
The fluid viscosity is set as $\mu = 1$, the Nitsche penalty is set as $\alpha = 2.5$ and the pressure stabilization parameter is chosen as $\gamma=1$.
Dirichlet conditions are applied to all boundaries, besides the left leg of the trapezoid, where a Neumann condition is enforced.
Figure~\ref{fig:efig1} and~\ref{fig:efig2} show the numerical solutions of the pressure and velocity, respectively. The $L^{2}$-norm of the solution errors, for the SBM and body-fitted algorithms are reported in Figure~\ref{fig:Stokes_conv_rates} and Table~\ref{table3}.
As for the numerical tests presented in Section~\ref{sec:Poisson}, the condition $\inf_{\tGD} \bs{\nu} \cdot \bs{n} > 0$ has no effect on the convergence rates of the solution error, shown in more detail in Table~\ref{table3}. In fact, the $L^2$-norm of the SBM error converges quadratically, hence faster than the theoretically predicted rate $3/2$.

\section{A three-dimensional numerical test for Stokes flow in complex geometry}
\label{sec:3DNumerical_Results}
In this section, we solve the Stokes flow problem in a three-dimensional domain given by a ``sponge-like'' cavity~\cite{sponge2018} shown in Figure~\ref{fig:spongefig1}.
This complex geometry contains a large number of holes, internal channels and sharp corners, and is a better representative of typical engineering applications.
Moreover, the domain considered here is represented in STL format (ie, a set of disconnected triangular facets) and generating grids on geometries that are not ``water-tight'' may be even more challenging for state-of-the-art (adaptive) meshing algorithms.

This geometry is immersed in a domain $\Om = [-10,-10] \times [-10,-10] \times [-11,7] $ with a total of approximately 7.5 and 25.2 million tetrahedral elements for a coarse and high resolution grid, whose surrogate boundaries are shown in Figure~\ref{fig:spongefig3} and Figure~\ref{fig:spongefig4}, respectively. 
Table~\ref{table4} accounts for the absolute and relative number of surrogate faces of the coarse (Figure~\ref{fig:spongefig3}) and fine (Figure~\ref{fig:spongefig4}) grids for which the geometric resolution condition $\inf_{\tGD} \bs{\nu} \cdot \bs{n} > 0$  is violated. 
For the numerical setup, we choose the same viscosity $\mu$ and stabilization parameters as the ones in Section~\ref{sec:Stokes}.

Figures~\ref{fig:spongefig2} and~\ref{fig:spongefig5} display the streamlines profile around the true geometry colored with pressure contours for the grids given in Figure~\ref{fig:spongefig3} and Figure~\ref{fig:spongefig4} respectively. These smooth and well-behaved numerical results indicate that the SBM can robustly capture the behavior of Stokes flow across very complex geometries despite the increasing number of faces violating the geometric resolution condition as the mesh is refined.  
\begin{figure}[t!]\centering	
	\begin{subfigure}[hbt]{0.32\textwidth}\centering
		\includegraphics[width=0.95 \linewidth]{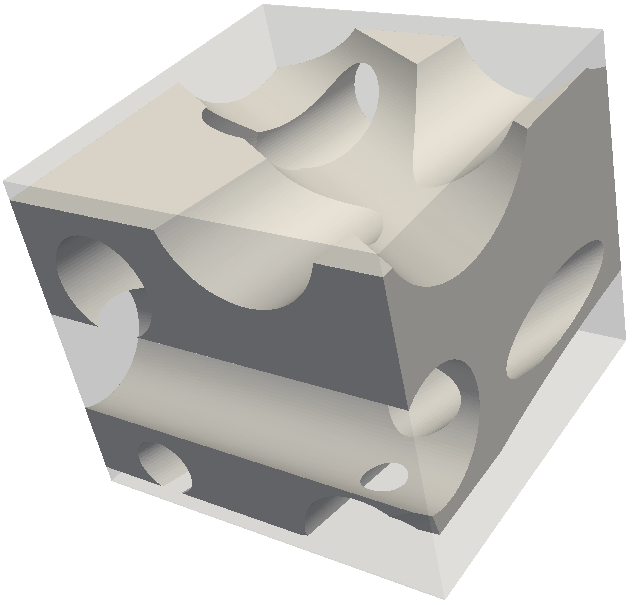}
		\caption{ True domain and true boundary.}
		\label{fig:spongefig1}
	\end{subfigure}
	\begin{subfigure}[hbt]{0.32\textwidth}\centering
		\includegraphics[width=0.95\linewidth]{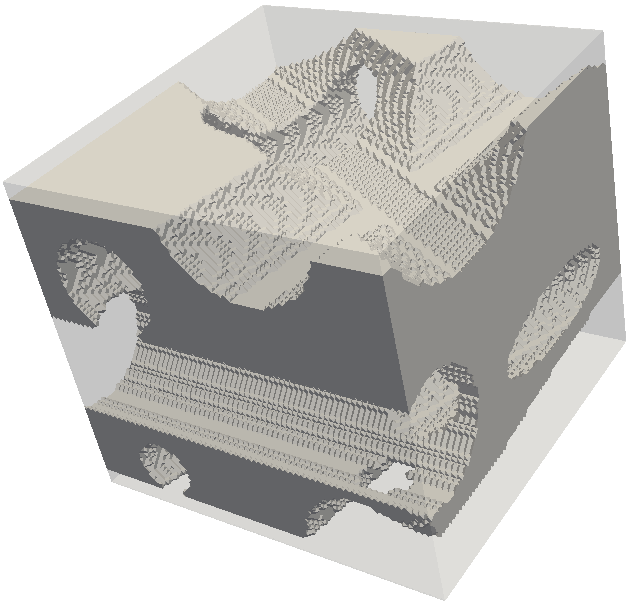}
		\caption{ Surrogate boundary (coarse grid).}
		\label{fig:spongefig3}
	\end{subfigure}
	\begin{subfigure}[hbt]{0.32\textwidth}\centering
		\includegraphics[width=0.95\linewidth]{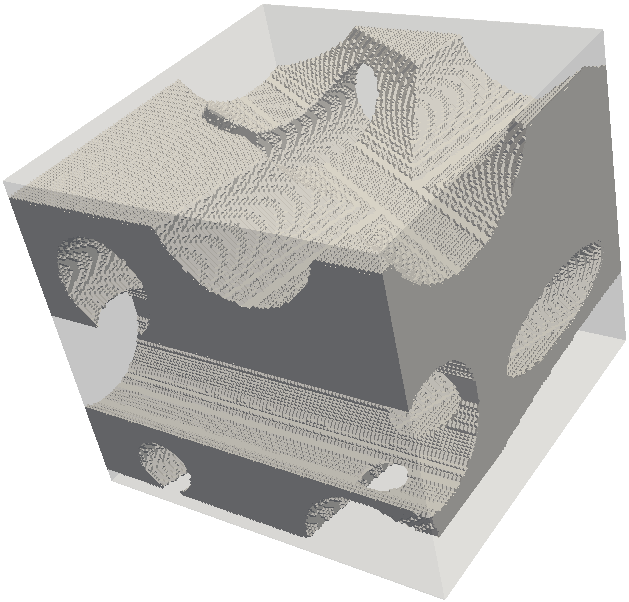}
		\caption{Surrogate boundary (fine grid).}
		\label{fig:spongefig4}
	\end{subfigure}
	\\[.5cm] 
	\begin{subfigure}[hbt]{0.36\textwidth}\centering
		\includegraphics[width=1\linewidth]{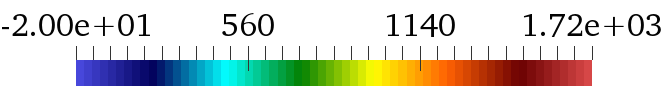}
	\end{subfigure}
\\[.1cm]
	\begin{subfigure}[hbt]{0.42\textwidth}\centering
		\includegraphics[width=1.0\linewidth]{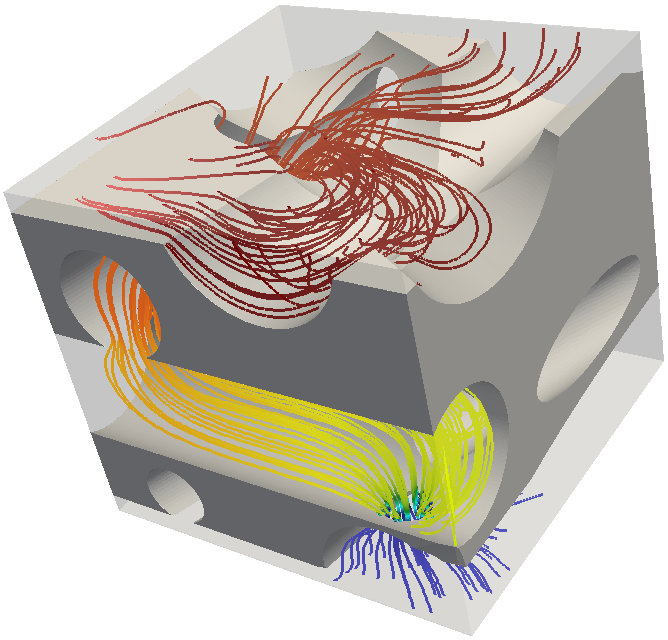}
		\caption{ Streamlines (colored with pressure) over the true geometry for the coarse grid.}
		\label{fig:spongefig2}
	\end{subfigure}
\qquad
\qquad
	\begin{subfigure}[hbt]{0.42\textwidth}\centering
		\includegraphics[width=1.0\linewidth]{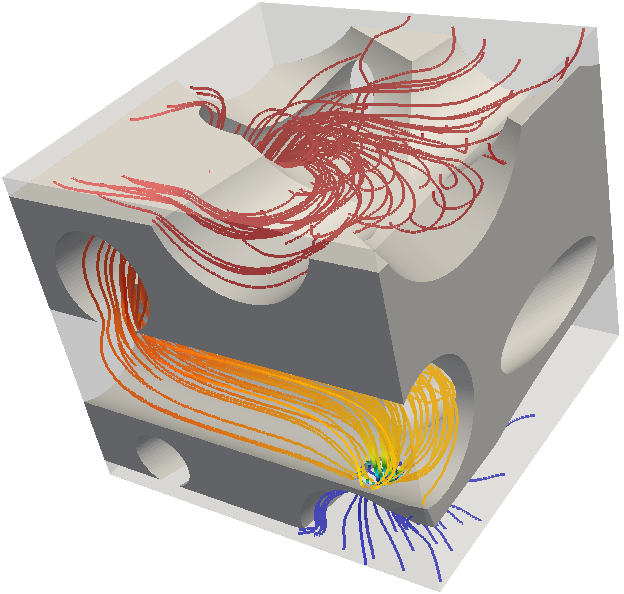}
		\caption{ Streamlines (colored with pressure) over the true geometry for the fine grid.}
		\label{fig:spongefig5}
	\end{subfigure}
	\caption{True and surrogate boundary of the sponge-like domain, with flow streamlines colored by pressure contours.}
	\label{fig:sponge_all}
\end{figure}
\begin{table}[hbt!]\centering
	\begin{tabular}{M{1.5cm} | M{3.0cm}  | M{6.0cm}  | M{5.6cm}  }
		\multicolumn{3}{c}{}  \\  
		Grid & No. of elements & No. of surrogate faces with $\bs{\nu} \cdot \bs{n} \leq 0$  & Percentage of total surrogate faces   \\
		\hline
		coarse & $\sim 7.5$ million & 234 & 0.196\%   \\
		fine & $\sim 25.2$ million & 577 & 0.189\%   \\
	\end{tabular}
	\caption{Number of surrogate faces with  $\bs{\nu} \cdot \bs{n} \leq 0$ for the coarse and fine grids shown in Figures~\ref{fig:spongefig3} and~\ref{fig:spongefig4}, respectively. \label{table4} }
\end{table}

\section{Summary \label{sec:summary} }
We have provided an enhanced analysis of well-posedness and accuracy for the SBM in the case of the Poisson and Stokes operators. The key improvement with respect to previous work are the removal of an assumption about the angle between the normal to the surrogate boundary and the corresponding normal to the true boundary. 
In addition, we have shown that no penalties on the tangential derivative of the Dirichlet boundary condition are required to prove stability and accuracy of the SBM. In addition, particular to the Stokes operator, incorporating an incompressibility constraint stabilization term was also deemed as unnecessary. Furthermore, in the enhanced $L^{2}$-error estimates, we removed the assumption of convexity of the surrogate domain $\tO$, relying instead on a conventional assumption of smoothness of the true domain $\Om$ (which could be replaced by convexity of the true domain $\Om$). These aspects are important in advancing the understanding and development of the SBM, in that they increase the flexibility and simplicity of the method, particularly when the geometry of the boundaries is highly complex. 
We have also performed a number of tests to verify the findings of our theoretical analysis.

\section*{Acknowledgments}
The support of the U.S. Department of Energy, Office of Science, Advanced Scientific Computing Research under Early Career Research Program Grant SC0012169 and the Army Research Office (ARO) under Grant W911NF-18-1-0308 is gratefully acknowledged. CC performed this research in the framework of the Italian MIUR Award ``Dipartimenti di Eccellenza 2018-2022" granted to the Department of Mathematical Sciences, Politecnico di Torino (CUP: E11G18000350001), and with the support of the Italian MIUR PRIN Project 201752HKH8-003. 
He is a member of the Italian INdAM-GNCS research group.

\section*{References}
\bibliographystyle{plain}
\bibliography{./SBM_2_0_arXiv} 

\begin{thebibliography}{10}

\bibitem{arnold1982interior}
Douglas~N Arnold.
\newblock An interior penalty finite element method with discontinuous
  elements.
\newblock {\em SIAM Journal on Numerical Analysis}, 19(4):742--760, 1982.

\bibitem{arnold2002unified}
Douglas~N Arnold, Franco Brezzi, Bernardo Cockburn, and L~Donatella Marini.
\newblock Unified analysis of discontinuous {G}alerkin methods for elliptic
  problems.
\newblock {\em SIAM Journal on Numerical Analysis}, 39(5):1749--1779, 2002.

\bibitem{TheoreticalPoissonAtallahCanutoScovazzi2020}
Nabil~M Atallah, Claudio Canuto, and Guglielmo Scovazzi.
\newblock {Analysis of the Shifted Boundary Method for the Poisson Problem in
  General Domains}.
\newblock 2020.

\bibitem{atallah2020analysis}
Nabil~M Atallah, Claudio Canuto, and Guglielmo Scovazzi.
\newblock {Analysis of the Shifted Boundary Method for the Stokes problem}.
\newblock {\em Computer Methods in Applied Mechanics and Engineering},
  358:112609, 2020.

\bibitem{bertoluzza2005fat}
Silvia Bertoluzza, Mourad Ismail, and Bertrand Maury.
\newblock The fat boundary method: {S}emi-discrete scheme and some numerical
  experiments.
\newblock In {\em Domain decomposition methods in science and engineering},
  pages 513--520. Springer, 2005.

\bibitem{bertoluzza2011analysis}
Silvia Bertoluzza, Mourad Ismail, and Bertrand Maury.
\newblock Analysis of the fully discrete fat boundary method.
\newblock {\em Numerische Mathematik}, 118(1):49--77, 2011.

\bibitem{boffi2003finite}
Daniele Boffi and Lucia Gastaldi.
\newblock A finite element approach for the immersed boundary method.
\newblock {\em Computers \& structures}, 81(8):491--501, 2003.

\bibitem{bramble1972projection}
James~H Bramble, Todd Dupont, and Vidar Thom{\'e}e.
\newblock Projection methods for dirichlet?s problem in approximating polygonal
  domains with boundary-value corrections.
\newblock {\em Mathematics of Computation}, 26(120):869--879, 1972.

\bibitem{bramble1994robust}
James~H Bramble and J~Thomas King.
\newblock A robust finite element method for nonhomogeneous dirichlet problems
  in domains with curved boundaries.
\newblock {\em mathematics of computation}, 63(207):1--17, 1994.

\bibitem{bramble1996finite}
James~H Bramble and J~Thomas King.
\newblock A finite element method for interface problems in domains with smooth
  boundaries and interfaces.
\newblock {\em Advances in Computational Mathematics}, 6(1):109--138, 1996.

\bibitem{burman2010ghost}
Erik Burman.
\newblock Ghost penalty.
\newblock {\em Comptes Rendus Mathematique}, 348(21-22):1217--1220, 2010.

\bibitem{burman2015cutfem}
Erik Burman, Susanne Claus, Peter Hansbo, Mats~G Larson, and Andr{\'e} Massing.
\newblock {CutFEM: {D}iscretizing geometry and partial differential equations}.
\newblock {\em International Journal for Numerical Methods in Engineering},
  104(7):472--501, 2015.

\bibitem{burman2018shape}
Erik Burman, Daniel Elfverson, Peter Hansbo, Mats~G Larson, and Karl Larsson.
\newblock Shape optimization using the cut finite element method.
\newblock {\em Computer Methods in Applied Mechanics and Engineering},
  328:242--261, 2018.

\bibitem{burman2014unfitted}
Erik Burman and Miguel~A Fern{\'a}ndez.
\newblock {An unfitted Nitsche method for incompressible fluid--structure
  interaction using overlapping meshes}.
\newblock {\em Computer Methods in Applied Mechanics and Engineering},
  279:497--514, 2014.

\bibitem{burman2010fictitious}
Erik Burman and Peter Hansbo.
\newblock {Fictitious domain finite element methods using cut elements: I. A
  stabilized {L}agrange multiplier method}.
\newblock {\em Computer Methods in Applied Mechanics and Engineering},
  199(41-44):2680--2686, 2010.

\bibitem{burman2012fictitious}
Erik Burman and Peter Hansbo.
\newblock {Fictitious domain finite element methods using cut elements: II. A
  stabilized Nitsche method}.
\newblock {\em Applied Numerical Mathematics}, 62(4):328--341, 2012.

\bibitem{burman2018cut}
Erik Burman, Peter Hansbo, and Mats Larson.
\newblock A cut finite element method with boundary value correction.
\newblock {\em Mathematics of Computation}, 87(310):633--657, 2018.

\bibitem{burman2017cut}
Erik Burman, Peter Hansbo, and Mats~G Larson.
\newblock A cut finite element method with boundary value correction for the
  incompressible {S}tokes equations.
\newblock In {\em European Conference on Numerical Mathematics and Advanced
  Applications}, pages 183--192. Springer, 2017.

\bibitem{burman2019dirichlet}
Erik Burman, Peter Hansbo, and Mats~G Larson.
\newblock Dirichlet boundary value correction using {L}agrange multipliers.
\newblock {\em arXiv preprint arXiv:1903.07104}, 2019.

\bibitem{cockburn2010boundary}
Bernardo Cockburn, Deepa Gupta, and Fernando Reitich.
\newblock {Boundary-conforming discontinuous {G}alerkin methods via extensions
  from subdomains}.
\newblock {\em Journal of Scientific Computing}, 42(1):144, 2010.

\bibitem{cockburn2014priori}
Bernardo Cockburn, Weifeng Qiu, and Manuel Solano.
\newblock A priori error analysis for hdg methods using extensions from
  subdomains to achieve boundary conformity.
\newblock {\em Mathematics of Computation}, 83(286):665--699, 2014.

\bibitem{cockburn2012solving}
Bernardo Cockburn and Manuel Solano.
\newblock {Solving Dirichlet boundary-value problems on curved domains by
  extensions from subdomains}.
\newblock {\em SIAM Journal on Scientific Computing}, 34(1):A497--A519, 2012.

\bibitem{duster2008finite}
Alexander D{\"u}ster, Jamshid Parvizian, Zhengxiong Yang, and Ernst Rank.
\newblock The finite cell method for three-dimensional problems of solid
  mechanics.
\newblock {\em Computer methods in applied mechanics and engineering},
  197(45):3768--3782, 2008.

\bibitem{AErn:2004a}
Alexandre Ern and Jean-Luc Guermond.
\newblock {\em Theory and Practice of Finite Elements}, volume 159 of {\em
  Applied Mathematical Sciences}.
\newblock Springer New York, 2004.

\bibitem{girault2012finite}
Vivette Girault and Pierre-Arnaud Raviart.
\newblock {\em Finite element methods for Navier-{S}tokes equations: {T}heory
  and algorithms}, volume~5.
\newblock Springer Science \& Business Media, 2012.

\bibitem{glowinski1994fictitious}
Roland Glowinski, Tsorng-Whay Pan, and Jacques Periaux.
\newblock A fictitious domain method for {D}irichlet problem and applications.
\newblock {\em Computer Methods in Applied Mechanics and Engineering},
  111(3-4):283--303, 1994.

\bibitem{hansbo2002unfitted}
Anita Hansbo and Peter Hansbo.
\newblock An unfitted finite element method, based on {N}itsche's method, for
  elliptic interface problems.
\newblock {\em Computer methods in applied mechanics and engineering},
  191(47):5537--5552, 2002.

\bibitem{hollig2003finite}
Klaus H{\"o}llig.
\newblock {\em Finite element methods with B-splines}.
\newblock SIAM, Philadelphia, 2003.

\bibitem{hollig2001weighted}
Klaus H{\"o}llig, Ulrich Reif, and Joachim Wipper.
\newblock Weighted extended b-spline approximation of {D}irichlet problems.
\newblock {\em SIAM Journal on Numerical Analysis}, 39(2):442--462, 2001.

\bibitem{Hughes198785}
Thomas J.~R. Hughes and Leopoldo~P. Franca.
\newblock A new finite element formulation for computational fluid dynamics:
  {VII}. {T}he {S}tokes problem with various well-posed boundary conditions:
  {S}ymmetric formulations that converge for all velocity/pressure spaces.
\newblock {\em Computer methods in applied mechanics and engineering}, 65(1):85
  -- 96, 1987.

\bibitem{kamensky2017immersogeometric}
David Kamensky, Ming-Chen Hsu, Yue Yu, John~A Evans, Michael~S Sacks, and
  Thomas~JR Hughes.
\newblock {Immersogeometric cardiovascular fluid--structure interaction
  analysis with divergence-conforming B-splines}.
\newblock {\em Computer Methods in Applied Mechanics and Engineering},
  314:408--472, 2017.

\bibitem{lozinski2016new}
Alexei Lozinski.
\newblock {A new fictitious domain method: Optimal convergence without cut
  elements}.
\newblock {\em Comptes Rendus Mathematique}, 354(7):741--746, 2016.

\bibitem{main2018shifted0}
A~Main and G~Scovazzi.
\newblock {The shifted boundary method for embedded domain computations. Part
  I: {P}oisson and {S}tokes problems}.
\newblock {\em Journal of Computational Physics}, 372:972--995, 2018.

\bibitem{main2018shifted}
A~Main and G~Scovazzi.
\newblock {The shifted boundary method for embedded domain computations. Part
  II: Linear advection--diffusion and incompressible Navier--{S}tokes
  equations}.
\newblock {\em Journal of Computational Physics}, 372:996--1026, 2018.

\bibitem{massing2015nitsche}
Andr{\'e} Massing, Mats Larson, Anders Logg, and Marie Rognes.
\newblock {A Nitsche-based cut finite element method for a fluid-structure
  interaction problem}.
\newblock {\em Communications in Applied Mathematics and Computational
  Science}, 10(2):97--120, 2015.

\bibitem{nitscheweak}
J.~A. Nitsche.
\newblock {Uber ein Variationsprinzip zur Losung {D}irichlet-Problemen bei
  Verwendung von Teilraumen, die keinen Randbedingungen unteworfen sind}.
\newblock {\em Abh. Math. Sem. Univ., Hamburg}, {36}:9--15, 1971.

\bibitem{parvizian2007finite}
Jamshid Parvizian, Alexander D{\"u}ster, and Ernst Rank.
\newblock Finite cell method.
\newblock {\em Computational Mechanics}, 41(1):121--133, 2007.

\bibitem{rangarajan2014universal}
Ramsharan Rangarajan and Adri{\'a}n~J Lew.
\newblock Universal meshes: A method for triangulating planar curved domains
  immersed in nonconforming meshes.
\newblock {\em International Journal for Numerical Methods in Engineering},
  98(4):236--264, 2014.

\bibitem{ruberg2012subdivision}
T~R{\"u}berg and F~Cirak.
\newblock Subdivision-stabilised immersed b-spline finite elements for moving
  boundary flows.
\newblock {\em Computer Methods in Applied Mechanics and Engineering},
  209:266--283, 2012.

\bibitem{ruberg2014fixed}
T~R{\"u}berg and F~Cirak.
\newblock A fixed-grid b-spline finite element technique for fluid--structure
  interaction.
\newblock {\em International Journal for Numerical Methods in Fluids},
  74(9):623--660, 2014.

\bibitem{schott2015face}
B~Schott, U~Rasthofer, V~Gravemeier, and WA~Wall.
\newblock A face-oriented stabilized {N}itsche-type extended variational
  multiscale method for incompressible two-phase flow.
\newblock {\em International Journal for Numerical Methods in Engineering},
  104(7):721--748, 2015.

\bibitem{song2018shifted}
Ting Song, Alex Main, Guglielmo Scovazzi, and Mario Ricchiuto.
\newblock The shifted boundary method for hyperbolic systems: Embedded domain
  computations of linear waves and shallow water flows.
\newblock {\em Journal of Computational Physics}, 369:45--79, 2018.

\bibitem{Tong2000MLSandia}
C~Tong and R~Tuminaro.
\newblock Ml2.0 smooth aggregation user's guide.
\newblock Technical report, Sandia National Laboratories, 2000.

\bibitem{sponge2018}
Mudrak V.
\newblock Part description and specifications for sponge.
\newblock https://grabcad.com/library/sponge.

\bibitem{xu2016tetrahedral}
Fei Xu, Dominik Schillinger, David Kamensky, Vasco Varduhn, Chenglong Wang, and
  Ming-Chen Hsu.
\newblock {The tetrahedral finite cell method for fluids: Immersogeometric
  analysis of turbulent flow around complex geometries}.
\newblock {\em Computers \& Fluids}, 141:135--154, 2016.

\end{thebibliography}
\appendix

\section{Some useful inequalities}
\label{sec:appendix_a}

Hereafter, we collect some well-known inequalities that are used in the paper.

\begin{thm}[Trace Theorem] \label{thm:TraceTheorem}
Assume $A \subset \mathbb{R}^{n_d}$ is open and bounded and $\partial A $ is Lipschitz. Then the trace operator $T : H^1(A) \to L^2(\partial A)$ such that $Tw = w_{|\partial A}$ satisfies
\begin{align}
	\| \, w \, \|_{L^{2}(\partial A) }^2  = \| \, Tw \, \|_{L^{2}(\partial A) }^2  \leq C \left( l(A)^{-1} \, \|\,w \, \|_{0,A}^2 + l(A) \, | \, w  \, |_{1,A}^2 \right), \quad \forall w \in H^1(A),
\end{align}
where $C$ is a constant that may depend on the shape of $A$  but not on its size, and $l(A)=\mathrm{meas}(A)^{1/n_d}$ is a characteristic length of the domain $A$.
\end{thm}

Let $\ti{\mathcal{T}}^h$ be the regular triangulation introduced in Section \ref{sec:sbmDef}, and let $H^k(\tO,\ti{\mathcal{T}}^h)= \prod_{T \in \ti{\mathcal{T}}^h} H^k(T)$ be the `broken' Sobolev space of order $k \geq 0$ with semi-norm $| \, v \, |_{k,\tO, \ti{\mathcal{T}}^h}= \sum_{T \in \ti{\mathcal{T}}^h} |\, v \, |_{k,T}$.
For the sake of simplicity, here and in the rest of the paper we use the symbol $| \, hv \, |_{k,\tO, \ti{\mathcal{T}}^h}$ to indicate the scaled quantity $\sum_{T \in \ti{\mathcal{T}}^h} | \, h_T v \, |_{k,T}$.
The general trace theorem above can be particularized to functions belonging to such spaces as follows.
\begin{subequations}
\begin{thm}[Scaled trace inequalities] \label{lemma:BoundaryTraceIneq}
There exists a constant $c_I>0$ independent of the mesh size such that for any element $T \subset \tO$ with an edge $\gamma_T \subset \tG$ one has 
\begin{align}
\label{eq:TraceInequality0_eT}
 \| \, h_T^{1/2} \, w \, \|^2_{0, \gamma_T} 
 &\leq \;
 c_I \left(  \| \, w \, \|^2_{0,T} +  | \, h_T \, w \, |^2_{1,T}  \right) \, , \quad \forall w \in H^1(T) \; .
\end{align}
Summing over all the elements with at least one of their edges on the boundary $\tG$, we obtain
\begin{align}
\label{eq:TraceInequality0_G}
 \| \, h^{1/2} \, w \, \|^2_{0, \tG} 
 &\leq \;
 c_I \left(  \| \, w  \, \|^2_{0,\tO} + | \, h \, w \, |^2_{1,\tO, \ti{\mathcal{T}}^h}  \right) \, , \quad \forall w \in H^1(\tO,\ti{\mathcal{T}}^h) \;.
\end{align}
\end{thm}
Combining these inequalities component-wise, one gets analogous results for vector- or tensor-valued functions.
\begin{thm}[Scaled vector/tensor trace inequalities] 
\label{theo:BoundaryTraceIneq}
There exists a constant $C_I>0$ independent of the mesh size such that 
\begin{align}
\label{eq:TraceInequality1_e}
\| \, h_{T} \, \nabla w  \cdot \bs{\nu} \, \|^2_{0, \gamma_T} 
& \leq \;
C_I \left( | \, w \, |^2_{1,T} +  | \, {h_T} \, w \, |^2_{2,T} \right) \, , \quad \ \ \forall w \in H^2(T) \; , 
\\ 
\label{eq:TraceInequality1_G}
\| \, h^{1/2} \, \nabla w  \cdot \bs{\nu} \, \|^2_{0, \tG} 
&\leq \;
C_I \left( | \, w \, |^2_{1,\tO} +  | \, {h} \, w \, |^2_{2,\tO, \ti{\mathcal{T}}^h} \right) \, , \quad \ \ \forall w \in H^1(\tO)\cap H^2(\tO,\ti{\mathcal{T}}^h) \; , 
\\  
\label{eq:StrainTraceInequality0}
\| \, h^{1/2} \, \bs{\varepsilon} (\bs{w}) \ti{\bs{n}}\,  \|_{0, \tG}^2 
&\leq \;
C_{I} \left( \|  \, \bs{\varepsilon}(\bs{w}) \, \|^2_{0,\tO} +  |\, h \,  \bs{\varepsilon}(\bs{w}) \, |^2_{1,\tO, \ti{\mathcal{T}}^h} \right) 
\nonumber \\
&\leq \;
C_{I} \left( | \, \bs{w} \, |^2_{1,\tO} + | \, h \, \bs{w} \, |^2_{2,\tO, \ti{\mathcal{T}}^h} \right) \, , \quad \ \ \forall \bs{w} \in (H^1(\tO)\cap H^2(\tO,\ti{\mathcal{T}}^h) )^{n_d}
\; .
\end{align}
In \eqref{eq:TraceInequality1_e} and \eqref{eq:TraceInequality1_G}, $\bs{\nu}$ denotes any unit vector field defined on the boundary.
\end{thm}
\end{subequations}
Using the equivalence of norms in a finite dimensional space, we obtain the following trace inequalities for piecewise affine functions.
\begin{subequations}
\begin{thm}[Discrete trace inequalities]
\label{theo:discretetrace}
There exist constants $c_I, C_I >0$ independent of the mesh size, such that
\begin{align}
\label{eq:TraceInequality0_e}
 \| \sqrt{h_T} \, w \, \|^2_{0; \gamma_T} 
 &\leq \;
 c_I \, \| \, w \, \|^2_{0,T}  \, , \quad \forall w \in \mathcal{P}^1(T)  \; ,
 \end{align}
and, for all vector functions $\bs{w}^h$ belonging to the space of piecewise linear and globally continuous functions over the mesh $\ti{\mathcal{T}}$,
\begin{align}
\label{eq:TraceInequality1} 
\| \, h^{1/2} \, \nabla \bs{w}^{h} \bs{\nu}  \, \|^2_{0, \tG} 
  &\leq \;
C_I  \, \| \, \nabla \bs{w}^{h} \, \|^2_{0,\tO}  \;, 
 \\
\| \sqrt{h} \, \nabla \cdot \bs{w}^{h} \, \|^2_{0, \tG} 
&\leq \;
C_I  \, \| \, \nabla \cdot \bs{w}^{h} \, \|^2_{0,\tO} \;, 
 \\
\| \, h^{1/2} \, \bs{\varepsilon} (\bs{w}^{h}) \ti{\bs{n}} \, \|_{0, \tG}^2 
&\leq  \;
C_{I} \, \| \, \bs{\varepsilon}(\bs{w}^{h}) \, \|^2_{0,\tO} \;.
\end{align}
In the second inequality, $\bs{\nu}$ denotes any unit vector field defined on the boundary.
\end{thm}
\end{subequations}
Next, we recall two classical Poincar\'{e}-type inequalities~\cite{arnold1982interior,arnold2002unified} below:
\begin{subequations}
\begin{thm}[Poincar\'{e} inequality]
\label{lemma:Poincare}
Assume that $\tO$ is a bounded connected open subset of $\mathbb{R}^{n_{d}}$ with Lipschitz boundary $\partial \tO$. 
There exists a constant $C_P > 0$, depending only on $\tO$ (and in particular, independent of $h$), such that for all $u \in H^1(\tO)$ 
\begin{equation}
\label{eq:Poincare1}
 \| \, u \, \|_{0, \tO} \leq C_P \, l(\tO) \, \left( \| \, \nabla u \, \|_{0,\tO} + \| \, h^{-1/2} \, u \, \|_{0, \tG} \right) \; . 
\end{equation}
\end{thm}
An alternative version of the Poincar\'{e} inequality holds for functions of bounded average.
\begin{thm}[Poincar\'{e} inequality for functions of bounded average]
\label{lemma:Poincare2}
Assume that $\tO$ is a bounded connected open subset of $\mathbb{R}^{n_{d}}$ with Lipschitz boundary $\partial \tO$. 
There exists a constant $\exists C_P' > 0$, depending only on $\tO$ (and in particular, independent of $h$), such that for all $p \in H^1(\tO)$ satisfying $\int_{\tO} p =0$, one has 
\begin{equation}
\label{eq:Poincare2}
 \| \, p \, \|_{0, \tO} \leq C_P' \, l(\tO) \, \| \, \nabla p \, \|_{0,\tO} \; .
\end{equation}
\end{thm}
\end{subequations}
Finally, we recall two inequalities of the Korn type for $H^1$-vector fields.
\begin{thm}[Korn's inequalities]
\label{lemma:Korn}
Let $\tO$ be a domain in $\mathbb{R}^{n_{d}}$ with ${n_{d}}\geq2$.
There exists a constant $C_K > 0$ such that for all $\bs{u} \in H^1(\tO)^{n_d}$, 
\begin{subequations}
\begin{equation}
\label{eq:korn1}
 \| \, \bs{u} \, \|_{H^1(\tO)}^2 \leq C_K \left(  \| \, \bs{u} \, \|_{0,\tO}^2 + l(\tO)^2 \, \| \,  \bs{\varepsilon}(\bs{u}) \, \|_{0,\tO}^2 \right)
\; .
\end{equation}
Furthermore, if $\ti{\Gamma} \subseteq \partial \tO$ has positive $(n_d-1)$-dimensional measure, there exists a constant $\bar{C}_K > 0$ independent of any $h \leq l(\tO)$ such that for all $\bs{u} \in H^1(\tO)^{n_d}$, 
\begin{equation}
\label{eq:korn2}
 \| \, \bs{u} \, \|_{H^1(\tO)}^2 \leq \bar{C}_K \, l(\tO)^2 \, \left(  \| \, h^{-1/2} \, \bs{u} \, \|_{0,\tG}^2 +  \| \, \bs{\varepsilon}(\bs{u}) \, \|_{0,\tO}^2 \right)
\; .
\end{equation}
\end{subequations}
\end{thm}

\end{document}